\documentclass[12pt]{amsart}
\usepackage{amssymb}
\usepackage{amsfonts}
\usepackage{latexsym}
\usepackage{amscd}
\usepackage[mathscr]{euscript}
\usepackage{xy} \xyoption{all}
\usepackage[active]{srcltx}

\usepackage{tkz-graph}
\usepackage{tikz-cd}
\usetikzlibrary{matrix,arrows,calc,automata}
\usetikzlibrary{decorations.pathreplacing}

\usepackage{tikz}

\usepackage{color,graphics}

\usepackage[shortlabels]{enumitem}
\usepackage{url}
\usepackage{float}

\newcommand{\N}{{\mathbb{N}}}
\newcommand{\Z}{{\mathbb{Z}}}

\newcommand{\C}{{\mathbb{C}}}

\newcommand{\uloopr}[1]{\ar@'{@+{[0,0]+(-4,5)}@+{[0,0]+(0,10)}@+{[0,0] +(4,5)}}^{#1}}
\newcommand{\uloopd}[1]{\ar@'{@+{[0,0]+(5,4)}@+{[0,0]+(10,0)}@+{[0,0]+ (5,-4)}}^{#1}}
\newcommand{\dloopr}[1]{\ar@'{@+{[0,0]+(-4,-5)}@+{[0,0]+(0,-10)}@+{[0, 0]+(4,-5)}}_{#1}}
\newcommand{\dloopd}[1]{\ar@'{@+{[0,0]+(-5,4)}@+{[0,0]+(-10,0)}@+{[0,0 ]+(-5,-4)}}_{#1}}

\newcommand{\calV}{{\mathcal V}}

\newcommand{\Lab}{L^{{\rm ab}}}

\newcommand{\luloop}[1]{\ar@'{@+{[0,0]+(-8,2)}@+{[0,0]+(-10,10)}@+{[0, 0]+(2,2)}}^{#1}}

\newcommand{\dotedge}{\ar@{.}}
\newcommand{\eqedge}{\ar@{=}}

\newcommand{\act}{\curvearrowright}

\newcommand{\mon}[1]{\calV(#1)} 

\newcommand{\So}{{\rm Source}}
\DeclareMathOperator{\Idem}{Idem} \DeclareMathOperator{\Tr}{Tr}

\newcommand{\ideal}{\hspace{0.5mm}\triangleleft \hspace{0.5mm}}

\newcommand{\andspace}{\quad \text{and} \quad}
\newcommand{\If}{\text{if }}
\newcommand{\ini}{\mathfrak{i}}
\newcommand{\ter}{\mathfrak{t}}
\newcommand{\niceblue}{rgb:red,1;green,2;blue,3}
\newcommand{\nicegreen}{rgb:red,1;green,4;blue,1}
\newcommand{\nicered}{rgb:red,4;green,1;blue,2}

\newcommand{\ul}{\underline}

\DeclareMathOperator{\id}{id}

\numberwithin{equation}{section}

\theoremstyle{plain}
\newtheorem{theorem}{Theorem}[section]
\newtheorem*{theorem*}{Theorem}
\newtheorem{lemma}[theorem]{Lemma}
\newtheorem*{lemma*}{Lemma}
\newtheorem{proposition}[theorem]{Proposition}
\newtheorem{corollary}[theorem]{Corollary}
\newtheorem*{corollary*}{Corollary}

\theoremstyle{definition}
\newtheorem{definition}[theorem]{Definition}
\newtheorem{example}[theorem]{Example}

\newtheorem{remark}[theorem]{Remark}
\newtheorem*{remark*}{Remark}

\newtheorem*{assumption*}{Assumption}

\newtheorem{construction}[theorem]{Construction}

\newtheorem{notation}[theorem]{Notation}

\begin{document}

\null\vskip-1cm

\title[Ideal structure]{Convex subshifts, separated Bratteli diagrams, and ideal structure of tame separated graph algebras}%
\author{Pere Ara}
\address{Departament de Matem\`atiques, Universitat Aut\`onoma de Barcelona,
08193 Bellaterra (Barcelona), Spain.} \email{para@mat.uab.cat}
\author{Matias Lolk}
\address{Department of Mathematical Sciences, University of Copenhagen, 2100 Copenhagen, Denmark.}\email{lolk@math.ku.dk}
\thanks{The first named author was partially supported by DGI-MINECO (Spain) through the grant
MTM2014-53644-P. The second named author was supported by the Danish National Research Foundation through the Centre for Symmetry and Deformation (DNRF92)} 
\date{\today}

\begin{abstract}
We introduce a new class of partial actions of free groups on
totally disconnected
compact Hausdorff spaces, which 
we call \textit{convex subshifts}. These serve as an abstract framework for 
the partial actions associated with finite separated graphs in much the same way as classical subshifts
generalize the edge shift of a finite graph. We define the notion of
a \textit{finite type} convex subshift and show that any such
subshift is Kakutani equivalent to the partial action associated
with a finite bipartite separated graph. We then study the ideal structure of both the {\it full} and the {\it reduced tame graph C*-algebras}, 
$\mathcal{O}(E,C)$ and $\mathcal O^r(E,C)$,
of a separated graph $(E,C)$, and of the {\it abelianized Leavitt path algebra} $\Lab _K(E,C)$ as well. 
These algebras are the (reduced) crossed products with respect to the above-mentioned 
partial actions, and we prove that there is a lattice isomorphism between the lattice of induced ideals and the lattice of hereditary $D^{\infty}$-saturated 
subsets of a certain infinite separated graph $(F_{\infty},D^{\infty})$ built from $(E,C)$, called the {\it separated Bratteli diagram} of $(E,C)$.
We finally use these tools to study simplicity and primeness of 
the tame separated graph algebras.  
\end{abstract}

\maketitle

\tableofcontents

\section{Introduction}

The study of C*-algebras associated to partial actions of groups on topological spaces is an important current line of investigation, 
see for instance \cite{Exel} and the references therein.  This of course includes the more traditional setting of globally
defined group actions, and is in turn generalized by the often useful setting of C*-algebras associated to \'etale 
topological groupoids \cite{Renault, PatGraphByRuy, CombinByRuy, Abadie2, Li, SW}.  

In \cite[Section 5]{ExelAmena}, Exel describes
any Cuntz-Krieger C*-algebra $\mathcal O _A$ as a crossed product of a commutative
C*-algebra by a \emph{partial action} of a non-abelian free group. This may be interpreted as follows. 
Given a $\{0,1\}$-matrix $A\in M_p(\{ 0,1 \})$, one may consider the one-sided shift of finite type $ X_A$
consisting of the infinite sequences $(x_n)\in \{1,\dots , p \}^{\N}$ such that $A(x_n,x_{n+1}) = 1$ for all $n$. Then, a partial action of the free group $\mathbb F_p$
on the space $X_A$ can be naturally defined so that $\mathcal O _A\cong C(X_A)\rtimes \mathbb F_p$. (This action is spelled out at the beginning of Section \ref{sect:ConvexSubshifts}.) 
This description can be generalized to the more general context of graph C*-algebras, see e.g. \cite[Theorem 37.8]{Exel}, \cite{GR}.

One may also consider two-sided subshifts $\mathcal X\subseteq \{1,\dots , p \}^{\Z}$ (see e.g. \cite{LM}), where the shift map is a true homeomorphism, and thus the associated C*-algebra is simply the 
crossed product $C(\mathcal X)\rtimes \Z$, where the action is induced by the shift. The C*-algebra associated to a given two-sided shift is quite different from the C*-algebra
associated to the corresponding one-sided shift. For instance, in the case of the full shift on $\{1,\dots , p\}$, we get the Cuntz algebra $\mathcal O_p$ from the one-sided shift, and we get the group C*-algebra 
$C^*(\Z_p \wr \Z)$ of the lamplighter group $\Z_p \wr \Z$
from the two-sided shift. 

In this paper, we propose a unified approach to one-sided and two-sided subshifts, under the general notion of a {\it convex subshift}. Given a finite alphabet $A$, 
define $\mathcal C(A)$ to be the space of all the 
(right-)convex
subsets of the free group $\mathbb F(A)$ on $A$ which contain $1$. Then the {\it full convex shift}
on $A$ is defined using the natural action of $\mathbb F(A)$ on $\mathcal C(A)$  (see Definition \ref{def:convexsubshift} for the precise definition).
A convex subshift is just the restriction of the full convex shift to a closed invariant subspace,
and we say that a convex subshift is of \textit{finite type} if it can be obtained by forbidding finitely many balls (or patterns).  

We show that this notion is equivalent to the study of the dynamical systems associated to separated graphs, a concept recently coined by the first-named author and Ruy Exel \cite{AE}. 
Recall that a  {\it separated graph} is a pair $(E,C)$ consisting of a directed
graph $E$ and a set $C=\bigsqcup _{v\in E^0} C_v$, where each $C_v$
is a partition of the set of edges whose terminal vertex is $v$.
By \cite[Theorem 9.1]{AE}, the study of the dynamical systems associated to separated graphs can be reduced to the case of {\it bipartite separated graphs}, which are those separated graphs 
$(E,C)$ such that $E^0= E^{0,0}\sqcup E^{0,1}$ is the disjoint union of two layers $E^{0,0}$ and $E^{0,1}$, and  $s(E^1)= E^{0,1}$,  $r(E^1)= E^{0,0}$, where $s$ and $r$ are the source and range maps respectively.  
Given a finite bipartite separated graph $(E,C)$, there exists a partial action $\theta^{(E,C)}$ of the free group $\mathbb F = \mathbb F (E^1)$ on a zero-dimensional metrizable compact space $\Omega (E,C)$
such that the tame C*-algebra $\mathcal O (E,C)$ is isomorphic to the partial crossed product $C(\Omega (E,C))\rtimes \mathbb F$ (\cite{AE}) --
this in turn allows for the definition of a \textit{reduced} tame graph $C^*$-algebra $\mathcal{O}^r(E,C)$ as the corresponding reduced crossed product.
A similar result holds for the abelianized Leavitt path algebra $\Lab_K(E,C)$ over a field with involution $K$, allowing us to study both types of algebras at the same time.

One of our main results is Theorem \ref{thm:GraphRep}, where we prove that any convex subshift of finite type is Kakutani equivalent to
the dynamical system coming from a finite bipartite separated graph. This provides a far-reaching generalization of the well-known result 
that every shift of finite type is conjugate to an edge shift of a finite graph \cite[Chapter 2]{LM}. 

The $K$-theory of the C*-algebras associated to separated graphs has been computed in \cite{AE2}. For a finite bipartite separated graph $(E,C)$, the $K_0$-group of 
both the reduced and the full tame graph C*-algebras of $(E,C)$ can be computed in terms of an infinite separated graph $(F_{\infty}, D^{\infty})$, which is the union of 
a sequence of finite bipartite separated graphs $(E_n,C^n)$. The infinite separated graph $(F_{\infty}, D^{\infty})$ has a structure which resembles very much the one 
of a usual Bratteli diagram. This leads us to define the new concept of a  {\it separated Bratteli diagram} (Definition \ref{def:sepBratelli}), and to refer to the graph $(F_{\infty}, D^{\infty})$ as the 
separated Bratteli diagram associated to $(E,C)$. 

Given a crossed product $\mathcal{O}=A \rtimes G$, we say that an ideal $J \ideal \mathcal{O}$ is \textit{induced} if $J=(J \cap A) \rtimes G$, and we denote 
by $\text{Ind}(\mathcal{O})$ the 
lattice of induced ideals. We show that the structure of induced ideals of the tame graph algebras of a finite separated graph $(E,C)$ is completely determined by its
associated separated Bratteli diagram $(F_{\infty}, D^{\infty})$. Concretely we obtain lattice isomorphisms
$$\text{Ind}(\Lab_K(E,C)) \cong \text{Ind}(\mathcal{O}(E,C)) \cong \text{Ind}(\mathcal{O}^r(E,C)) \cong \mathcal H (F_{\infty}, D^{\infty}),$$
where $\mathcal H (F_{\infty}, D^{\infty})$ is the lattice of hereditary $D^{\infty}$-saturated subsets of $F_{\infty}^0$ (see Theorem \ref{thm:idealsOr}).
This generalizes the well-known result for ordinary graph C*-algebras \cite{BHRSByRuy, BPRSByRuy}.

In sharp contrast with the situation for ordinary graphs, the quotient algebra $\mathcal O (E,C)/I(H)$ 
of an ideal generated by a hereditary $D^{\infty}$-saturated subset $H$ can be described in terms of the tame algebra 
of a separated graph only when $H$ is of {\it finite type}, meaning that $H$ can be generated by a hereditary $C^n$-saturated subset of vertices of the separated graph $(E_n, C^n)$
for some $n\ge 0$. Otherwise, the quotient $\mathcal O (E,C)/I(H)$ can be described by means of a crossed product of a free group on a compact invariant 
subset of $\Omega (E,C)$, but it will in general not be a tame graph C*-algebra of a separated graph. Even so, its $K$-theory can be computed by using a corresponding separated Bratteli diagram
$(F_{\infty}/H, D^{\infty}/H)$ (Theorem \ref{thm:K-theorytamequotients}). This justifies the introduction of general separated Bratteli diagrams, and indicates that interesting examples (some of them
of a pathological behaviour) can occur if we allow the consideration of hereditary $D^{\infty}$-saturated of infinite type; see the comment after Example \ref{exam:cycles}.

We present in some detail an example which is connected with the full two-sided shift. We believe this example is very useful to understand the various aspects of the theory that have been described
above. It is also the example that allows to study the group algebra of the lamplighter group from the perspective of the theory of separated graphs. Using this, and generalizing \cite{AGreal}, the first-named author
and Joan Claramunt have obtained a concrete approximation of this group algebra by a sequence of finite-dimensional algebras \cite{AC}.

We also study general (i.e.~not necessarily induced) ideals of $\mathcal{O}^r(E,C)$. We say that an ideal $J$ is of \textit{finite type} if the corresponding induced ideal $I(H)=(J \cap C(\Omega(E,C))) \rtimes_r \mathbb{F}$ arises from a hereditary $D^\infty$-saturated set $H$ of finite type. By studying a weakening of topological freeness that we dub \textit{relative strong topological freeness}, we are able to compute the lattice of finite type ideals. We also observe that our techniques, when applied to non-separated graphs, 
yield a complete characterization of the ideal structure.

We close the paper by using our tools to perform a study of simplicity and primeness in tame graph algebras of separated graphs. 

In this paper, we consider three types of tame graph $*$-algebras associated to separated graphs, namely the {\it full tame graph C*-algebra} $\mathcal O (E,C)$, the {\it reduced tame
graph C*-algebra} $\mathcal O ^r(E,C)$, and the {\it abelianized Leavitt path algebra} $\Lab _K(E,C)$ over a field with involution $K$. All three types of algebras have been introduced in \cite{AE}.
The first two types generalize graph C*-algebras \cite{Raeburn}, in the sense that, denoting by $\mathcal T $ the trivial partition on each $r^{-1}(v)$, we have $\mathcal O (E, \mathcal T) = \mathcal O^r(E, \mathcal T) = C^*(E)$, where $C^*(E)$ is the graph C*-algebra of $E$. The third type generalizes Leavitt path algebras \cite{AA1, AMP}, in the sense that $\Lab_K (E,\mathcal T) = L_K(E)$, where $L_K(E)$ is the Leavitt path algebra
of $E$.

\bigskip

\textbf{Contents.} We now describe the contents of the paper in more
detail. In section~\ref{sect:prels}, we recall relevant definitions
and constructions from the existing theory of algebras associated
with separated graphs, in particular the main results of \cite{AE}.
We also introduce the notion of a {\it separated Bratteli diagram}.
In section~\ref{sect:ConvexSubshifts}, we introduce a class of
partial actions that we call \textit{convex subshifts}. These serve
as an abstract framework for the partial actions associated with
finite separated graphs in much the same way as classical subshifts
generalize the edge shift of a finite graph. We define the notion of
a \textit{finite type} convex subshift and show that any such
subshift is Kakutani equivalent to the partial action associated
with a finite bipartite separated graph
(Theorem~\ref{thm:GraphRep}). Using the language of convex
subshifts, we are also able to explain the precise relationship
between the partial action of a finite bipartite separated graph
$(E,C)$ and its successors $(E_n,C^n)$ for $n \ge 1$
(Theorem~\ref{thm:NBallIsNGraph}). In
section~\ref{sec:structure-ideals}, we begin our investigation of
the lattice of induced ideals of the algebras $\Lab(E,C)$,
$\mathcal{O}(E,C)$ and $\mathcal{O}^r(E,C)$. The main result
(Theorem~\ref{thm:idealsOr}) identifies this lattice with the
lattice $\mathcal{L}(M(F_\infty,D^\infty))$ of order ideals of the
monoid $M(F_\infty,D^\infty)$, as well as the lattice of hereditary
and $D^\infty$-saturated subsets of $(F_\infty,D^\infty)$, providing
both an algebraic and a graph theoretic perspective. In the
following section, we study the specific ideals associated with
hereditary and $C$-saturated subsets of $(E,C)$, and we show that
the corresponding quotients arise as separated graph algebras from
the quotient graph (Theorem~\ref{thm:quotientalg}). In
section~\ref{sec:InducedIdeals}, we combine the work of the previous
sections to describe the quotient by an arbitrary induced ideal as a
limit of separated graph algebras
(Proposition~\ref{prop:LimitOfOs}), and we provide a dynamical
description of this approximation. We also pay significant attention
to a concrete example that illustrates the general theory
(Examples~\ref{ex:lamplighter}, \ref{exam:cycles}, and
\ref{exam:ruy}). In particular, we show that the crossed product of
a two-sided finite type subshift can be realised as a full corner of
the tame graph C*-algebra associated with a finite bipartite
separated graph (Proposition~\ref{prop:subshifts}).
We proceed to study general ideals of finite type in section~\ref{sect:GeneralIdeals}, and we provide a complete characterisation in terms of graph theoretic 
data (Theorem~\ref{thm:GeneralIdeals}). This relies on an abstract study of \textit{relatively strongly topologically free} partial actions that we initiate.
We then study $\mathcal{V}$-simplicity of the separated graph algebras in
section~\ref{sec:Vsimplicity}, showing that (up to Morita
equivalence) the tame graph C*-algebras degenerate to either
classical graph C*-algebras or group C*-algebras of free groups if
$M(F_\infty,D^\infty) \cong \mathcal{V}(\Lab(E,C))$ is order simple
(Theorem~\ref{thm:dichotomy}).
In section~\ref{sect:Primeness}, we finally establish a criterion for $\Omega(E,C)$ to be a Cantor space (Proposition~\ref{prop:Cantor}), 
and in this setting we characterize primeness of the algebras $\Lab_K(E,C)$ and $\mathcal{O}^r(E,C)$ (Theorem~\ref{thm:Prime}).

\section{Preliminary definitions}
\label{sect:prels}

In this preliminary section, we will recall all the relevant definitions and constructions from the theory of algebras associated with separated graphs. The reader should note that we use 
the same conventions as in \cite{AE} and \cite{AE2}, but opposite to those of \cite{AG} and \cite{AG2}. One important consequence is that all paths should be read from the right.

\begin{definition}{\rm (\cite{AG})} \label{defsepgraph}
A \emph{separated graph} is a pair $(E,C)$ where $E=(E^0,E^1,r,s)$ is a directed graph,
$C=\bigsqcup _{v\in E^ 0} C_v$, and $C_v$ is a partition of
$r^{-1}(v)$ (into non-empty subsets) for every
vertex $v$. In case $v$ is a source, i.e. $r^{-1}(v)=\emptyset$, we take $C_v$ to be the empty
family of subsets of $r^{-1}(v)$. Given an edge $e$, we shall use the notation $X_e$ for the element of $C$ containing $e$.

If all the sets in $C$ are finite, we say that $(E,C)$ is a \emph{finitely separated} graph. This necessarily holds if $E$ is column-finite (that is, if $r^{-1}(v)$ is a finite set for every
$v\in E^0$.)

The set $C$ is a \emph{trivial separation} of $E$ in case $C_v=\{r^{-1}(v)\}$ for each $v\in E^0\setminus \So (E)$. In that case, $(E,C)$ is called a \emph{trivially separated graph} or a \emph{non-separated graph}.

Finally, $(E,C)$ is called \textit{bipartite} if the vertex set admits a partition $E^0=E^{0,0} \sqcup E^{0,1}$ with $s(E^1) = E^{0,1}$ and $r(E^1) = E^{0,0}$.
\end{definition}

As the first of many different algebras assocated with separated graphs, we now define the Leavitt path algebra.

\begin{definition}
\label{def:LPASG} {\rm Let $(K, *)$ be a field with involution. The
{\it Leavitt path algebra} of the separated graph $(E,C)$ with
coefficients in the field $K$ is the $*$-algebra $L_K(E,C)$ with
generators $\{ v, e\mid v\in E^0, e\in E^1 \}$, subject to the
following relations:}
\begin{enumerate}[leftmargin=2cm,rightmargin=2cm]
\item[(V)] $vv^{\prime} = \delta_{v,v^{\prime}}v$ \ and \ $v=v^*$ \ for all $v,v^{\prime} \in E^0$ ,
\item[(E)] $r(e)e=es(e)=e$ \ for all $e\in E^1$ ,
\item[(SCK1)] $e^*e'=\delta _{e,e'}s(e)$ \ for all $e,e'\in X$, $X\in C$, and
\item[(SCK2)] $v=\sum _{ e\in X }ee^*$ \ for every finite set $X\in C_v$, $v\in E^0$.
\end{enumerate}
\end{definition}

The Leavitt path algebra $L_K(E)$ is just $L_K(E,C)$ where $C_v= \{r^{-1}(v)\}$ if $r^{-1}(v)\ne \emptyset $ and $C_v=\emptyset $ if $r^{-1}(v)=\emptyset$. An arbitrary field can be considered as a
field with involution by taking the identity as the involution. However, our ``default" involution over the complex numbers $\C$ will be the complex conjugation, and we will write $L(E,C):=L_\mathbb{C}(E,C)$.

We now recall the definition of the graph C*-algebra $C^*(E,C)$,
introduced in \cite{AG2}.

\begin{definition} The \emph{graph C*-algebra} of a separated graph $(E,C)$ is the C*-algebra $C^*(E,C)$  with generators $\{ v, e \mid v\in E^0,\ e\in E^1 \}$, subject to the relations (V), (E), (SCK1), (SCK2). In other words, $C^*(E,C)$ is the enveloping C*-algebra of $L(E,C)$.
\end{definition}

In case $(E,C)$ is trivially separated, $C^*(E,C)$ is just the classical graph C*-algebra $C^*(E)$. There is a unique *-homomorphism  $L(E,C) \rightarrow C^*(E,C)$ sending the generators of $L(E,C)$ to their canonical images in $C^*(E,C)$. This map is injective by \cite[Theorem 3.8(1)]{AG2}.

Since both $L_K(E,C)$ and $C^*(E,C)$ are universal objects with respect to the same sets of generators and relations, they can be studied in much the same way. A remarkable difference is that the non-stable $K$-theory $\mon{L_K(E,C)}$ \textit{has} been computed for any separated graph $(E,C)$, while the structure of $\mon{C^*(E,C)}$ is still unknown. However, it is conjectured in \cite{AG2} that the
natural map $L(E,C)\to C^*(E,C)$ induces an isomorphism
$\mon{L(E,C)}\to \mon{C^*(E,C)}$. See \cite[Section 6]{Aone-rel}
for a short discussion on this problem. We now describe $\mon{L_K(E,C)}$ as the \textit{graph monoid} of $(E,C)$.

\begin{definition}
\label{def:graphmonoid-sepgraph} Given a finitely separated graph
$(E,C)$, the \textit{graph monoid} $M(E,C)$ is the abelian monoid
with generators $a_v$ for $v \in E^0$ and relations $a_v=\sum_{e \in
X} a_{s(e)}$ for all $v \in E^0$ and $X \in C_v$.
\end{definition}

Note that there is a natural map $M(E) \to \mon{L_K(E,C)}$ given by $a_v \mapsto [v]$, and this is in fact an isomorphism by \cite[Theorem 4.3]{AG}.

One issue with the algebras $L_K(E,C)$ and $C^*(E,C)$ is that the set of partial isometries represented by the edges $E^1$ is 
not \textit{tame}, that is, products of edges and their adjoints are not in general partial isometries. This motivated Ruy Exel and the first named author to introduce certain quotients $\Lab_K (E,C)$ 
and $\mathcal{O}(E,C)$ of $L_K(E,C)$ and $C^*(E,C)$, respectively, that we shall now describe.

\begin{definition}
For a finitely separated graph $(E,C)$, let $S$ denote the multiplicative subsemigroup of $L_K(E,C)$ generated by $E^1 \cup (E^1)^*$, and define an ideal (respectively a closed ideal)
$$J = \langle \alpha\alpha^*\alpha - \alpha \mid \alpha \in S \rangle = \langle [\alpha\alpha^*,\beta\beta^*] \colon \alpha,\beta \in S \rangle$$
of $L_K(E,C)$ (respectively of $C^*(E,C)$). Then we set
$$\Lab _K (E,C) := L_K(E,C)/J \andspace \mathcal{O}(E,C):=C^*(E,C)/J.$$
Observe that modding out $J$ precisely forces $E^1$ to be tame in these quotients.
\end{definition}

We now recall the main construction of \cite{AE} which will play an
important role in what is to come.

\begin{definition}
\label{def:Finftyandothers}
 Let $(E,C)$ denote a finite bipartite separated
graph, and write
$$C_u=\{X_1^u,\ldots,X_{k_u}^u\}$$
for all $u \in E^{0,0}$. Then $(E_1,C^1)$ is the finite bipartite separated graph defined by
\begin{itemize}
\item $E_1^{0,0}:=E^{0,1}$ and $E_1^{0,1}:=\{v(x_1,\ldots,x_{k_u}) \mid u \in E^{0,0}, x_j \in X_j^u\}$,
\item $E^1:=\{\alpha^{x_i}(x_1,\ldots,\widehat{x_i},\ldots,x_{k_u}) \mid u \in E^{0,0}, i=1,\ldots,k_u, x_j \in X_j^u \}$,
\item $r_1(\alpha^{x_i}(x_1,\ldots,\widehat{x_i},\ldots,x_{k_u})):=s(x_i)$ and $s_1(\alpha^{x_i}(x_1,\ldots,\widehat{x_i},\ldots,x_{k_u})):=v(x_1,\ldots,x_{k_u})$,
\item $C^1_v:=\{X(x) \mid x \in s^{-1}(v)\}$, where
$$X(x_i):=\{\alpha^{x_i}(x_1,\ldots,\widehat{x_i},\ldots,x_{k_u}) \mid x_j \in X_j^u \text{ for } j \ne i\}.$$
\end{itemize}
A sequence of finite bipartite separated graphs $\{(E_n,C^n)\}_{n \ge 0}$ with $(E_0,C^0):=(E,C)$ is then defined inductively by letting $(E_{n+1},C^{n+1})$ denote the $1$-graph of $(E_n,C^n)$. Finally, $(F_n,D^n)$ denotes the union $\bigcup_{i=0}^n (E_n,C^n)$, and $(F_\infty,D^\infty)$ is the infinite layer graph
$$(F_\infty,D^\infty):=\bigcup_{n=0}^{\infty} (F_n,D^n) = \bigcup_{n=0}^\infty (E_n,C^n).$$
Observe that $(F_n,D^n)$ is a finite separated graph, while $(F_\infty,D^\infty)$ is a finitely separated graph.
\end{definition}

By \cite[Theorem 5.1 and Theorem 5.7]{AE}, there are canonical surjective $*$-homomorphisms $*$-homomorphisms
$$L_K(E_n,C^n) \to L_K(E_{n+1},C^{n+1}) \andspace C^*(E_n,C^n) \to C^*(E_{n+1},C^{n+1})$$
such that
$$L_K^{ab}(E,C) \cong \varinjlim_n L_K(E_n,C^n) \andspace \mathcal{O}(E,C) \cong \varinjlim_n C^*(E_n,C^n).$$
On the level of monoids, the induced monoid homomorphism
$$M(E_n,C^n) \cong \mathcal{V}(L_K(E_n,C^n)) \to \mathcal{V}(L_K(E_{n+1},C^{n+1}) \cong M(E_{n+1},C^{n+1})$$
is a unitary embedding, which refines the defining relations of $M(E_n,C^n)$. Consequently, the quotient map $L_K(E,C) \to \Lab_K (E,C)$ induces a (universal) refinement
$$M(E,C) \cong \mathcal{V}(L_K(E,C)) \to \mathcal{V}(\Lab_K (E,C)) \cong M(F_{\infty},D^{\infty}).$$

One of the main advantages of dealing with the quotients $\Lab_K (E,C)$ and $\mathcal{O}(E,C)$ is that, by \cite[Corollary 6.12]{AE}, they admit descriptions as crossed products
$$\Lab_K (E,C) \cong C_{K}(E,C) \rtimes \mathbb{F} \andspace \mathcal{O}(E,C) \cong C(\Omega(E,C)) \rtimes \mathbb{F}$$
of a topological partial action $\theta^{(E,C)} \colon \mathbb{F} \act \Omega(E,C)$, referred to as the \textit{canonical partial $(E,C)$-action}. Here, $\mathbb{F}$ is the free group generated by $E^1$, $\Omega(E,C)$ is a certain compact, zero-dimensional, metrisable space, and $C_{K}(\Omega(E,C))$ denotes the $*$-algebra of continuous functions $\Omega(E,C) \to K$ when $K$ is given the discrete topology, that is, the $*$-algebra of locally constant functions $\Omega(E,C) \to K$. Every vertex $v \in F_{\infty}^0$ corresponds to a compact open subset $\Omega(E,C)_v \subset \Omega(E,C)$ such that $\Omega(E,C)=\bigsqcup_{v \in E_n^0} \Omega(E,C)_v$ for all $n$.

The crossed product description enables the definition of yet another tame graph $C^*$-algebra.

\begin{definition}
Let $(E,C)$ denote a finite bipartite separated graph. Then $\mathcal{O}^r(E,C)$ is the reduced crossed product $C(\Omega(E,C)) \rtimes_r \mathbb{F}$ of the canonical partial $(E,C)$-action.
\end{definition}

The separated graph $(F_{\infty}, D^{\infty})$ resembles very much
the structure of a {\it Bratteli diagram} \cite{Brat}, but
incorporating separations (and up to a change of the convention of
the direction of the arrows). We formalize this with the following
definition.

\begin{definition}
\label{def:sepBratelli} A separated (or colored) Bratteli diagram is
an infinite separated graph $(F,D)$. The vertex set $F^0$ is the
union of finite, non-empty, pairwise disjoint sets $F^{0,j}$, $j\ge
0$. Similarly, the edge set $F^1$ is the union of a sequence of
finite, non-empty, pairwise disjoint sets $F^{1,j}$, $j\ge 0$. The
range and source maps satisfy $r(F^{1,j})=F^{0,j}$ and
$s(F^{1,j})=F^{0,j+1}$ for all $j\ge 0$.
\end{definition}

A Bratteli diagram is just a separated Bratteli diagram with the
trivial separation. The graph $(F_{\infty}, D^{\infty})$ associated
to a finite bipartite separated graph $(E,C)$ is an example of a
separated Bratteli diagram, and we will refer to it as the {\it
separated Bratteli diagram} of the separated graph $(E,C)$. We will
find later other examples of separated Bratteli diagrams, related to
quotients of tame graph C*-algebras, see Theorem
\ref{thm:K-theorytamequotients}.

The {\it graph monoid} of a Bratteli diagram $(F,D)$ is the graph
monoid $M(F,D)$ of the finitely separated graph $(F,D)$, see
Definition \ref{def:graphmonoid-sepgraph}. The Grothendieck group
$G(F,D)$ of the monoid $M(F, D)$ is an analogue of the dimension
group associated to a Bratteli diagram, but it is not a dimension
group in general. It is quite sensible to equip $G(F,D)$ with the
structure of a pre-ordered abelian group, taking $G(F,D)^+ :=
\varphi (M(F,D))$, where $\varphi  \colon M(F,D) \to G(F,D)$ is the
canonical map.

The separated Bratteli diagram of a finite bipartite separated graph
computes the $K_0$ of the tame graph algebras of the graph, as
follows.

\begin{theorem}
\label{thm:K-theoryBrat} Let $(E,C)$ be a finite bipartite separated
graph, and let $(F_{\infty}, D^{\infty})$ be its separated Bratteli
diagram.
\begin{enumerate}
\item[\textup{(1)}] There is a natural group isomorphism
$$K_0(\mathcal O (E,C))\cong K_0(\mathcal O ^r(E,C)) \cong
G(F_{\infty}, D^{\infty}).$$
\item[\textup{(2)}] There is a natural isomorphism $\mathcal V (\Lab _K(E,C))\cong
M(F_{\infty},D^{\infty})$, and so an isomorphism of pre-ordered
abelian groups
$$K_0(\Lab _K(E,C)) \cong G(F_{\infty}, D^{\infty})$$
for any field with involution $K$.
\end{enumerate}
\end{theorem}

\begin{proof}
(2) follows from \cite[Theorem 5.1]{AE}, and (1) follows from
\cite[Theorem 4.4(c), Corollary 6.9]{AE2}.
\end{proof}

In the following, we shall recall a very handy description of the
canonical partial $(E,C)$-action introduced in \cite{AE}, although
we will make a slight change to the original definition, coherent
with the conventions of \cite{Lolk1} and \cite{Lolk2}. We remark
that while we only consider finite bipartite graphs, the dynamical
picture can be extended to arbitrary finitely separated graphs
\cite[Definition 2.6 and Theorem 2.10]{Lolk1}. For a comprehensive
study of the general theory of partial actions and their crossed
products, we refer the reader to \cite{Exel}.

\begin{definition}
Suppose that $(E,C)$ is a separated graph, and let $\widehat{E}$ denote the \textit{double} of $E$, i.e. the graph with
$$\widehat{E}^0:=E^0, \widehat{E}^1:=E^1 \sqcup (E^1)^{-1}, \hat{r}(e):=r(e)=\hat{s}(e^{-1}) \andspace \hat{r}(e^{-1}):=s(e)=:\hat{s}(e).$$
A path in the double of $E$ (with the convention that a path is read from the right to the left) is called \textit{an admissible path} if
\begin{itemize}
\item $e \ne f$ for every subpath $ef^{-1}$,
\item $X_e \ne X_f$ for every subpath $e^{-1}f$.
\end{itemize}
We define the range and source of an admissible path to simply be
the range and source in the double, and we view the vertices as the
set of trivial admissible paths.

A \textit{closed path} in $(E,C)$ is a non-trivial admissible path
$\alpha$ with $r(\alpha)=s(\alpha)$, and $\alpha$ is called a
\textit{cycle} if the concatenated word $\alpha\alpha$ is an
admissible path as well. Either way, we shall say that $\alpha$ is
\textit{based} at $r(\alpha)=s(\alpha)$. A cycle is called
\textit{simple} if the only vertex repetition occurs at the end.
\end{definition}

\begin{definition}\label{def:DynamicalPicture}
Suppose that $(E,C)$ is finite bipartite separated graph, and let $\mathbb{F}$ denote the free group on $E^1$. Given $\xi \subset \mathbb{F}$ and $\alpha \in \xi$, the \textit{local configuration} $\xi_\alpha$ of $\xi$ at $\alpha$ is defined as
$$\xi_\alpha :=\{\sigma \in E^1 \sqcup (E^1)^{-1} \mid \sigma \in \xi \cdot \alpha^{-1}\}.$$
Then $\Omega(E,C)$ is the set of $\xi \subset \mathbb{F}$ satisfying the following:
\begin{enumerate}
\item[(a)] $1 \in \xi$.
\item[(b)] 
$\xi$ is \textit{right-convex}: In view of (a), this exactly means that if $e_n^{\varepsilon_n} \cdots e_1^{\varepsilon_1} \in \xi$ for $e_i \in E^1$ and $\varepsilon_i \in \{\pm 1\}$, then $e_m^{\varepsilon_m} \cdots e_1^{\varepsilon_1} \in \xi$ as well for any $1 \le m < n$.
\item[(c)] For every $\alpha \in \xi$, one of the following holds:
\begin{enumerate}
\item[(c1)] $\xi_\alpha = s^{-1}(v)$ for some $v \in E^{0,1}$.
\item[(c2)] $\xi_\alpha = \{e_X^{-1} \mid X \in C_v\}$ for some $v \in E^{0,0}$ and $e_X \in X$.
\end{enumerate}
\end{enumerate}
Observe that for $\xi \in \Omega(E,C)$, every $\alpha \in \xi$ is an admissible path. $\Omega(E,C)$ is made into a topological space by regarding it as a subspace of $\{0,1\}^\mathbb{F}$. Thus it becomes a compact, zero-dimensional, metrisable space, and a topological partial action $\theta=\theta^{(E,C)} \colon \mathbb{F} \act \Omega(E,C)$ with compact open domains is then defined by setting
\begin{itemize}
\item $\Omega(E,C)_\alpha := \{\xi \in \Omega(E,C) \mid \alpha^{-1} \in \xi\}$,
\item $\theta_\alpha(\xi) := \xi \cdot \alpha^{-1}$ for $\xi \in \Omega(E,C)_{\alpha^{-1}}$.
\end{itemize}
This partial action is equivalent to the one defined in \cite{AE} under the map $\xi \mapsto \xi^{-1}$. We choose to invert the elements as we want a common terminology and notation for both the algebraic and the topological setting. The compact, open sets corresponding to the vertices of $E$ are given by
$$\Omega(E,C)_v = \left\{\begin{array}{cl}
\Omega(E,C)_{e^{-1}} \text{ for some } e \in s^{-1}(v) & \text{if } v \in E^{0,1} \\ \bigsqcup_{e \in X}\Omega(E,C)_e \text{ for some } X \in C_v & \text{if } v \in E^{0,0}
\end{array} \right. ,$$
and we note that this is independent of $e$ and $X$ by (c1) and (c2), respectively.  The reader may think of $\Omega(E,C)_v$ as the set of configurations "starting" in $v$, and we shall regard $1 \in \xi$ as the trivial path $v$ when $\xi \in \Omega(E,C)_v$. See Remark~\ref{rem:Omegav} for a description of $\Omega(E,C)_v$ when $v \in E_n^0$ for $n \ge 1$.
\end{definition}

\section{Convex subshifts}\label{sect:ConvexSubshifts}
Given a finite alphabet $A$ and any set $\mathcal{F}$ of finite words in $A$, one can consider the space $X_{\mathcal{F}} \subset A^{\mathbb{N}}$ of infinite 
one-sided sequences that do not contain any words from $\mathcal{F}$. In classical symbolic dynamics, $X_\mathcal{F}$ is then usually turned into a dynamical 
system by equipping it with the one-sided shift $\sigma$, but for the purpose of constructing interesting algebras, one would typically add additional dynamical structure. 
Specifically, a sequence may not only be shifted to the left, but new letters may also be introduced in the beginning, provided that we stay inside $X_\mathcal{F}$, of course.

Formally, this dynamical system can be regarded as a partial action with the free group $\mathbb{F}(A)$ acting on $X_\mathcal{F}$: If $a \in A$ and $x \in X_\mathcal{F}$, then $a$ can act on $x$ by $a.x=ax$ whenever $ax \in X_\mathcal{F}$, while the inverse $a^{-1}$ can act on points $ax \in X_\mathcal{F}$ by $a^{-1}.ax=x=\sigma(ax)$. It is a standard fact that any subshift of finite type (meaning that $\mathcal{F}$ is finite) arises as the edge shift of some finite graph $E$ with no sinks, and the graph $C^*$-algebra may then be recovered as the partial crossed product $C^*(E) \cong C(X_\mathcal{F}) \rtimes \mathbb{F}(A)$ (see for instance \cite[Theorem 36.20]{Exel} or \cite[Theorem 3.1]{CL}). If $E$ is a finite directed graph, one can also consider the \textit{boundary path space} $\partial E$, where -- in addition to the set of right infinite paths -- one also includes the set of finite paths ending in a sink. Then there is natural partial action of $\mathbb{F}(E^1)$ on $\partial E$ defined as above, and we still have $C^*(E) \cong C(\partial E) \rtimes \mathbb{F}(E^1)$. In fact, a similar description exists for arbitrary graphs.

In this section, we shall introduce a class of partial actions that one might consider as a generalisation of both the above partial actions on one-sided sequence spaces (including those of infinite type) and 
the boundary path spaces of finite graphs. We will see later (Example \ref{ex:lamplighter}, Proposition \ref{prop:subshifts}) that also two-sided shifts can be recasted in this language.  
The basic idea is to give up the linear nature of a sequence and allow trees instead. The dynamics still arise from shifting, but there are usually many possible shifting directions with no one being canonical. 
We will refer to these dynamical systems as \textit{convex subshifts} for reasons that should become apparent soon.

The motivation for introducing such systems is two-fold: On one hand, it provides a framework for the dynamical systems associated with finite separated graphs in which one has far more flexibility. 
On the other hand, the dynamical systems of finite bipartite separated graphs encompass a vast class of convex subshifts, in fact all \textit{finite type} convex subshifts up to Kakutani-equivalence 
(see Theorem~\ref{thm:GraphRep}). As such, convex subshifts are to finite bipartite separated graphs as one-sided subshifts are to finite graphs. 
Finally, we shall see how the construction of $(E_n,C^n)$ from $(E,C)$ corresponds to a natural construction in the realm of convex subshifts.

The contents of this section are closely related to classical subshifts of free groups on the alphabet $\{0,1\}$. Indeed, any convex subshift is the restriction of a free group subshift to a partial action, 
but the convexity requirement that we impose is \textit{not} of finite type; hence a convex subshift would typically be regarded as an infinite type subshift. Moreover, it is most beneficial 
for us to define everything from scratch so that the formal framework of convex subshifts is similar to that of partial actions associated with separated graphs.

Throughout this section, $A$ will be a finite alphabet and $\mathbb{F}=\mathbb{F}(A)$ will be the free group on $A$. Though some definitions still make sense for infinite alphabets, we choose to deal only with finite ones for the sake of simplicity.

\begin{definition}
\label{def:convexsubshift}
Denote by $\mathcal{C}=\mathcal{C}(A)$ the set of right-convex subsets $\xi \subset \mathbb{F}$ for which $1 \in \xi$. 
As in Definition~\ref{def:DynamicalPicture}, right-convexity in this setting simply means that if $a_n^{\varepsilon_n} \cdots a_1^{\varepsilon_1} \in \xi$ for $a_i \in A$ and $\varepsilon_i \in \{\pm 1\}$, then $a_m^{\varepsilon_m} \cdots a_1^{\varepsilon_1} \in \xi$ as well for all $1 \le m < n$.
We topologize $\mathcal{C}$ as a subspace of $\mathcal{P}(\mathbb{F}) \cong \{0,1\}^{\mathbb{F}}$, making it into a compact, zero-dimensional, and metrizable space (with many isolated points). The \textit{full convex shift on $A$} is then the partial action $\mathbb{F} \act \mathcal{C}$ given by
\begin{itemize}
\item $\mathcal{C}_\alpha := \{ \xi \in \mathcal{C} \mid \alpha^{-1} \in \xi\}$ for all $\alpha \in \mathbb{F}$.
\item $\alpha.\xi := \xi \cdot \alpha^{-1}$ for $\xi \in \mathcal{C}_{\alpha^{-1}}$.
\end{itemize}
Viewing each $\xi \in \mathcal{C}$ as a tree, rooted in $1$ and labelled relative to the root, the action of any $\alpha \in \xi$ on $\xi$ is thus simply given by moving the root to $\alpha$ and relabelling accordingly.
\end{definition}

\begin{remark}
We could also define $\mathcal{C}$ to be the set of
\textit{left-convex}
subsets with 
$$\mathcal{C}_\alpha=\{\xi \in \mathcal{C} \mid \alpha \in \xi\} \andspace \alpha.\xi=\alpha \cdot \xi,$$
but we choose the above convention to match the one used for separated graphs. It is also very convenient that with our convention, $\alpha$ can act on $\xi$ if and only if $\alpha \in \xi$ (as opposed to $\alpha^{-1} \in \xi$). By \cite[Section 4]{ELQ} and \cite[Proposition 4.5]{ExelLaca}, the full convex shift is conjugate to the universal action for semi-saturated partial representations of $\mathbb{F}$.
\end{remark}

\begin{definition}
A \textit{convex subshift} is the restriction of the full convex shift $\mathbb{F} \act \mathcal{C}$ to any closed invariant subspace $\Omega \subset \mathcal{C}$.
\end{definition}

\begin{definition}
An \textit{$n$-ball} is an element $B \in \mathcal{C}$ such that $\vert \alpha \vert \le n$ for every $\alpha \in B$, together with the \textit{radius} $r(B)=n$; a \textit{ball} is then simply an $n$-ball for some $n$. Note that $B$ as a set does not determine the radius $r(B)$, so one has to specify this.  Now if $\xi \in \mathcal{C}$, we can always consider the $n$-ball
$$\xi^n := \{ \alpha \in \xi \colon \vert \alpha \vert \le n\},$$
and if $B$ is a given $n$-ball, we shall write $\xi \not\equiv B$ if $(\alpha.\xi)^n \ne B$ for all $\alpha \in \xi$. Finally, if $\Omega$ is any convex subshift, we will write
$$\mathcal{B}_n(\Omega):=\{ \xi^n \mid \xi \in \Omega\} \andspace \mathcal{B}(\Omega):=\bigsqcup_{n \ge 0} \mathcal{B}_n(\Omega)$$
for the set of \textit{allowed} $n$-balls and \textit{allowed} balls, respectively.
\end{definition}

With all the relevant terminology in place, we can give an example of a convex subshift.

\begin{definition}
Let $\mathcal{F}$ denote any set of balls; we can then define a convex subshift by
$$\Omega^\mathcal{F}:=\{ \xi \in \mathcal{C} \mid \xi \not\equiv B \text{ for all } B \in \mathcal{F}\},$$
and we shall refer to this as the convex subshift obtained from \textit{forbidding} $\mathcal{F}$.
\end{definition}

In fact, this is not \textit{an} example, but rather \textit{the} example.

\begin{proposition}
If $\mathbb{F} \act \Omega$ is a convex subshift, then $\Omega=\Omega^\mathcal{F}$ for some set of balls $\mathcal{F}$.
\end{proposition}
\begin{proof}
Define $\mathcal{F}$ to be all the balls that do not occur in $\Omega$, i.e. let
$$\mathcal{F}:=\mathcal{B}(\mathcal{C}) \setminus \mathcal{B}(\Omega).$$
Clearly $\Omega \subset \Omega^\mathcal{F}$, so let us consider the reverse inclusion. Given any $\xi \in \Omega^\mathcal{F}$ and $n \ge 1$, it is enough to check that $\xi^n = \eta^n$ for some $\eta \in \Omega$ since $\Omega$ is closed. But since $\xi^n \notin \mathcal{F}$, we must have $\xi^n \in \mathcal{B}_n(\Omega)$, so $\xi^n = \eta^n$ for some $\eta \in \Omega$ as desired.
\end{proof}

\begin{definition}\label{def:FiniteType}
A convex subshift $\mathbb{F} \act \Omega$ is of \textit{finite type} if $\Omega=\Omega^\mathcal{F}$ for some finite set of forbidden balls $\mathcal{F}$. Observe that for such a convex subshift, we can safely assume that all balls of $\mathcal{F}$ have the same radius $R$; we recognize 
this by saying that $\Omega$ is \textit{$R$-step}. Observe that in this situation, $\Omega$ is \textit{generated} by $\mathcal{B}_R(\Omega)$ in the following sense: An element $\xi \in \mathcal{C}$ belongs to $\Omega$ if and only if $(\alpha.\xi)^R \in \mathcal{B}_R(\Omega)$ for all $\alpha \in \xi$.
\end{definition}

Before venturing on, we need to discuss how one might compare partial actions of different groups.

\begin{definition}\label{def:DynConj}
Consider partial actions $\theta \colon G \act \Omega$ and $\theta' \colon H \act \Omega'$ of discrete groups on topological spaces. Then $\theta$ and $\theta'$ are called \textit{dynamically equivalent} and we shall write $\theta \approx \theta'$, if their transformation groupoids $\mathcal{G}_\theta$ and $\mathcal{G}_{\theta'}$ (see for instance \cite[Example 2.3]{Lolk2}) are isomorphic as topological groupoids. We now spell out exactly what this means: There is a homeomorphism $\varphi \colon \Omega \to \Omega'$ and continuous maps
$$a \colon \bigcup_{g \in G} \{g\} \times \Omega_{g^{-1}} \to H \andspace b \colon \bigcup_{h \in H} \{h\} \times \Omega'_{h^{-1}} \to G$$
such that
\begin{enumerate}
\item $\varphi(x) \in \Omega'_{a(g,x)^{-1}}$ and $\varphi(g.x)=a(g,x).\varphi(x)$,
\item $\varphi^{-1}(y) \in \Omega_{b(h,y)^{-1}}$ and $\varphi^{-1}(h.y)=b(h,y).\varphi^{-1}(y)$,
\item $b(a(g,x),\varphi(x))=g$ and $a(b(h,y),\varphi^{-1}(y))=h$,
\item $a(g'g,x)=a(g',g.x)a(g,x)$ if $g.x \in \Omega_{g'^{-1}}$,
\item $b(h'h,y)=b(h',h.y)b(h,y)$ if $h.y \in \Omega_{h'^{-1}}$
\end{enumerate}
for all $g \in G$, $h \in H$, $x \in \Omega_{g^{-1}}$, $y \in \Omega'_{h^{-1}}$.
\end{definition}

\begin{remark}
\label{rem:GroupoidAlgebras} Given a locally compact Hausdorff
\'etale groupoid $\mathcal{G}$, one can associate to it both a \textit{universal} and a \textit{reduced} groupoid $C^*$-algebra
(see \cite{Renault}), denoted $C^*(\mathcal{G})$ and
$C^*_r(\mathcal{G})$, respectively. If $\mathcal{G}=\mathcal{G}_\theta$ for a partial action $\theta \colon
G \act \Omega$ of a discrete group on a locally compact Hausdorff space, then there are identifications $C^*(\mathcal{G}_\theta) \cong C_0(\Omega) \rtimes_\theta G$ and $C^*_r(\mathcal{G}_\theta) \cong C_0(\Omega) \rtimes_{\theta,r} G$ by \cite[Theorem 3.3]{Abadie2} and \cite[Proposition 2.2]{Li}. For an \textit{ample} groupoid $\mathcal{G}$ and any field $K$ with involution, there is also a purely algebraic analogue $K\mathcal{G}$, known as the \textit{Steinberg algebra} of $\mathcal{G}$ \cite{Steinberg}. If a partial action $\theta$ as above acts on a totally disconnected space, then $\mathcal{G}_\theta$ is ample and $K\mathcal{G}_\theta \cong C_K(\Omega) \rtimes_\theta G$, where $C_K(\Omega)$ denotes the algebra of compactly supported continuous functions $\Omega \to K$, when $K$ is endowed with the discrete topology. Hence if two partial actions $\theta$ and ${\theta'}$ as above are dynamically equivalent, i.e. $\mathcal{G}_\theta \cong \mathcal{G}_{\theta'}$, then they have isomorphic crossed products by base-preserving isomorphisms. In this light, groupoids provide a very flexible framework for identifying the crossed products of partial actions.
\end{remark}

\begin{definition}
Consider partial actions $\theta \colon G \act \Omega$ and ${\theta'} \colon H \act \Omega'$ of discrete groups on topological spaces along with a group homomorphism $\Psi \colon G \to H$. A continuous map $\varphi \colon \Omega \to \Omega'$ is then called $\Psi$\textit{-equivariant} if
\begin{enumerate}
\item $\varphi(\Omega_g) \subset \Omega'_{\Psi(g)}$ for all $g \in G$,
\item $\varphi(g.x)=\Psi(g).\varphi(x)$ for all $x \in \Omega_{g^{-1}}$.
\end{enumerate}
The pair $(\varphi,\Psi)$ is called a \textit{conjugacy} if $\Psi$ is an isomorphism and $\varphi$ admits a $\Psi^{-1}$-equivariant inverse. However, conjugacy is often too rigid a notion and we therefore introduce another type of equivalence in between conjugacy and dynamical equivalence:
The pair $(\varphi,\Psi)$ is called a \textit{direct dynamical equivalence} if
\begin{enumerate}[(a)]
\item $\varphi$ is a homeomorphism,
\item $\Omega_g \cap \Omega_{g'} = \emptyset$ for all $g \ne g'$ with $\Psi(g)=\Psi(g')$,
\item $\Omega'_h=\bigcup_{\Psi(g)=h} \varphi(\Omega_g)$ for all $h \in H$,
\end{enumerate}
and in this case we will write $\theta \xrightarrow{\approx} {\theta'}$.
\end{definition}

Now let us see that our choice of name and notation is justified.

\begin{proposition}\label{prop:DirQCImpliesQC}
If $\theta \xrightarrow{\approx} {\theta'}$, then $\theta \approx {\theta'}$.
\end{proposition}

\begin{proof}
Apply the notation from above. We then set $a(g,x):=\Psi(g)$ for all $g \in G$, $x \in \Omega_{g^{-1}}$, and given $h \in H$, $y \in \Omega'_{h^{-1}}$, we define $b(h,y)$ to be the unique element of $G$ satisfying
$$\Psi(b(h,y))=h \andspace y \in \varphi(\Omega_{b(h,y)^{-1}}).$$
Then
$$\varphi^{-1}(h.y)=\varphi^{-1}(\Psi(b(h,y)).y)=b(h,y).\varphi^{-1}(y)$$
and
$$b(a(g,x),\varphi(x))=b(\Psi(g),\varphi(x))=g,$$
so (1)-(4) of Definition~\ref{def:DynConj} surely hold. Finally, if $h' \in H$ and $h.y \in \Omega'_{h'^{-1}}$, then
$$\Psi(b(h',h.y)b(h,y))=h'h$$
and
\begin{align*}
y &\in h^{-1}.\big(\varphi(\Omega_{b(h',h.y)}) \cap \Omega'_h \big)=\Psi(b(h,y))^{-1}.\big(\varphi(\Omega_{b(h',h.y)}) \cap \Omega'_{\Psi(b(h,y))} \big) \\
&= \varphi \big(b(h,y)^{-1}.(\varphi(\Omega_{b(h',h.y)}) \cap \Omega'_{\Psi(b(h,y))}) \big)= \varphi \big(\Omega_{b(h,y)^{-1}b(h',h.y)^{-1}} \cap \Omega_{b(h,y)^{-1}}\big),
\end{align*}
hence
$$b(h'h,y)=b(h',h.y)b(h,y)$$
as desired.
\end{proof}

\begin{remark}\label{rem:DirDynEq}
If $(\varphi,\Phi)$ is a direct dynamical 
equivalence from $\theta \colon G \act \Omega$ to $\theta' \colon H \act \Omega'$, then the induced isomorphisms $C_K(\Omega) \rtimes G \to C_K(\Omega') \rtimes H$ and $C_0(\Omega) \rtimes_{(r)} G \to C_0(\Omega') \rtimes_{(r)} H$ are simply given by $f \delta_g \mapsto f \circ \varphi^{-1} \vert_{\Omega_{\Phi(g)}'} \delta_{\Phi(g)}$ with inverses $f \delta_h \mapsto \sum_{\Phi(g)=h} f \circ \varphi \vert_{\Omega_g} \delta_g$.
\end{remark}

In fact, any dynamical equivalence is a result of two direct dynamical equivalences:

\begin{proposition}
If $\theta \approx {\theta'}$ for partial actions of $G$ and $H$, respectively, then
$$\theta \xleftarrow{\approx} \gamma \xrightarrow{\approx} {\theta'}$$
for some partial action $\gamma$ of $G \times H$.
\end{proposition}

\begin{proof}
By otherwise replacing ${\theta'}$ with the conjugate partial action $\varphi \circ {\theta'} \circ \varphi^{-1}$, we may assume that $\theta$ and $\theta'$ act on the same space $\Omega$ and $\varphi=\id_\Omega$. Writing $a_g:=a(g,-)$ and $b_h:=b(h,-)$, we then first define domains by
$$\Omega_{(g,h)}:=a_{g^{-1}}^{-1}(h^{-1})=b_{h^{-1}}^{-1}(g^{-1})$$
for $(g,h) \in G \times H$. To see that the above equality holds, assume that $x \in a_{g^{-1}}^{-1}(h^{-1})$, i.e. $x \in \Omega_g$ with $a(g^{-1},x)=h^{-1}$. Then $x \in \Omega_h$ and
$$b(h^{-1},x)=b(a(g^{-1},x),x)=g^{-1},$$
hence $x \in b_{h^{-1}}^{-1}(g^{-1})$, so $a_{g^{-1}}^{-1}(h^{-1})=b_{h^{-1}}^{-1}(g^{-1})$ from symmetry. Then define the action of $(g,h)$ by
$$(g,h).x:=g.x=a(g,x).x=h.x$$
for all $x \in \Omega_{(g,h)^{-1}}$. It is straightforward to verify that this does indeed define a partial action $\gamma \colon G \times H \act \Omega$. Then simply observe that the pairs $(\text{id}_\Omega,\pi_G)$ and $(\text{id}_\Omega,\pi_H)$, where $\pi_G$ and $\pi_H$ denote the projections onto $G$ and $H$, respectively, are direct dynamical equivalences.
\end{proof}

Returning to the world of convex shifts, we recall that any $n$-step subshift $X$ can be recoded into being $1$-step using higher block shifts; one simply replaces the original alphabet with the $n$-blocks $\mathcal{B}_n(X)$ via the map
$$x \mapsto [x_1 \ldots x_n][x_2 \ldots x_{n+1}][x_3 \ldots x_{n+2}] \ldots . $$
In the following, we shall make a similar construction in the world of convex shifts -- although with \textit{blocks} replaced by \textit{balls}.

\begin{construction}
Let $\theta \colon \mathbb{F} \act \Omega$ denote any convex subshift, let $n \ge 1$ and consider the finite alphabet
$$A^{[n \colon \Omega]}:=\Big\{\big[(a.\xi)^n \xleftarrow{a} \xi^n \big] \mid \xi \in \Omega, a \in A \text{ such that } a \in \xi \Big\},$$
where each $\big[(a.\xi)^n \xleftarrow{a} \xi^n \big]$ is just a formal symbol. As every symbol is typically represented by many different configurations $\xi$, we will simply use the notation $[B \xleftarrow{a} B']$ where $B,B' \in \mathcal{B}_n(\Omega)$ in the future. We then consider the corresponding free group $\mathbb{F}^{[n \colon \Omega]}:=\mathbb{F}(A^{[n \colon \Omega]})$ and introduce the notation
$$[B' \xleftarrow{a^{-1}} B]:=[B \xleftarrow{a} B']^{-1}.$$
Observe that if $\Omega \subset \Lambda$ is an inclusion of convex subshifts over $A$, then we obtain corresponding inclusions $A^{[n \colon \Omega]} \subset A^{[n \colon \Lambda]}$ and $\mathbb{F}^{[n \colon \Omega]} \subset \mathbb{F}^{[n \colon \Lambda]}$. When $\Omega= \mathcal{C}$ is the full convex shift, we will simply write $A^{[n]} := A^{[n \colon \mathcal{C}]}$ and $\mathbb{F}^{[n]}:=\mathbb{F}^{[n \colon \mathcal{C}]}$, and define a group homomorphism
$$\Psi_n  \colon \mathbb{F}^{[n]} \to \mathbb{F}  \quad \text{given by} \quad \Psi_n([B \xleftarrow{a} B']):= a.$$
By a slight abuse of notation, we will also refer to $\Psi_n$ when we really mean the restriction of $\Psi_n$ to the subgroup $\mathbb{F}^{[n \colon \Omega]}$.

Our aim is to define a replacement for the above block-encoding, more specifically a map
$$\phi_n \colon \mathcal{C}(A) \to \mathcal{C}(A^{[n]}).$$
Given $\xi \in \Omega$, we first set $\phi_n(\xi,1):=1 \in \mathbb{F}^{[n]}$ and proceed to define $\phi_n(\xi,\alpha) \in \mathbb{F}^{[n]}$ for $1 \ne \alpha \in \xi$. Writing $\alpha=s_m \cdots s_1$ and $B_k=((s_k \cdots s_1).\xi)^n$ for $k \le m$ so that $B_0=\xi^n$ and $B_m=(\alpha.\xi)^n$, we then set
$$\phi_n(\xi,\alpha):=[B_m \xleftarrow{s_m} B_{m-1}] \cdot [B_{m-1} \xleftarrow{s_{m-1}} B_{m-2}] \cdots [B_1 \xleftarrow{s_1} B_0],$$
allowing us to define $\phi_n$ by
$$\phi_n(\xi):=\{\phi_n(\xi,\alpha) \mid \alpha \in \xi\}.$$
It is clear from the construction that $\phi_n(\xi)$ is a right-convex subset of $\mathbb{F}^{[n]}$ containing $1$, so that $\phi_n(\xi) \in \mathcal{C}(A^{[n]})$. Observe that if $\xi \in \Omega$, then $\phi_n(\xi,\alpha) \in \mathbb{F}^{[n \colon \Omega]}$ for all $\alpha \in \xi$, so that $\phi_n$ restricts to a map $\Omega \to \mathcal{C}(A^{[n \colon \Omega]})$. The \textit{$n$-ball subshift} $\theta^{[n]}$ of $\theta$ is then simply the restricted action of $\mathbb{F}^{[n \colon \Omega]}$ on the image $\Omega^{[n]}:=\phi_n(\Omega)$. We observe that $\phi_n$ has a $\Psi_n$-equivariant inverse $\psi_n \colon \Omega^{[n]} \to \Omega$ given by
$$\psi_n(\eta) :=\{\Psi_n(\beta) \mid \beta \in \eta\},$$
and as both $\phi_n$ and $\psi_n$ are obviously continuous, they are in fact homeomorphisms of $\Omega$ and $\Omega^{[n]}$. In particular, it follows that $\Omega^{[n]}$ is compact and invariant under the action of $\mathbb{F}^{[n \colon \Omega]}$, hence a convex subshift.

Using the same notation as above, we also define a map
$$\tilde{\Psi}_n \colon \{\beta \in \eta \mid \eta \in \mathcal{C}^{[n]}\}=\{\phi_n(\xi,\alpha) \mid \xi \in \mathcal{C},\alpha \in \xi\} \to \mathcal{C}$$
by
$$\tilde{\Psi}_n\big([B_m \xleftarrow{s_m} B_{m-1}] \cdots [B_1 \xleftarrow{s_1} B_0] \big) :=\bigcup_{k=0}^m (s_k \cdots s_1)^{-1}.B_k$$
and note that $\tilde{\Psi}_n(\phi_n(\xi,\alpha)) \subset \xi$ for all $\alpha \in \xi$. Consequently,
$$\psi_n(\eta)=\bigcup_{\beta \in \eta} \tilde{\Psi}_n(\beta)$$
for any $\eta \in \mathcal{C}^{[n]}$.
\end{construction}

In the following we shall see that passing to higher ball shift does indeed allow one to recode any finite type convex shift into a $1$-step convex shift. First though, we have to deal with the higher ball shifts of the full convex shift.

\begin{lemma}\label{lem:NBallOfFull}
The $n$-ball subshift $\mathbb{F}^{[n]} \act \mathcal{C}^{[n]}$ of the full convex shift is $1$-step.
\end{lemma}

\begin{proof}
It should be clear from the above construction that the $1$-balls in $\mathcal{C}^{[n]}$ are exactly the sets of the form
$$\{[B(s)\xleftarrow{s} B]\}_{s \in B} \cup \{1\}$$
for some $B,B(s) \in \mathcal{B}_n(\mathcal{C})$ such that $[B(s)\xleftarrow{s} B] \in A^{[n]} \cup (A^{[n]})^{-1}$ for all $s \in B$, and we claim that these sets in fact generate $\mathcal{C}^{[n]}$ as a $1$-step convex subshift of $\mathbb{F}^{[n]} \act \mathcal{C}(A^{[n]})$. To this end, we let $\eta \in \mathcal{C}(A^{[n]})$ and assume that $(\beta.\eta)^1$ is of this form for any $\beta \in \eta$. Now set
$$\xi:= \{\Psi_n(\beta) \mid \beta \in \eta\}$$
and observe that $\xi \in \mathcal{C}$: we ultimately wish to show that $\eta = \phi_n(\xi) \in \mathcal{C}^{[n]}$. Letting $B \in \mathcal{B}_n(\mathcal{C})$ denote the $n$-ball as above corresponding to $\eta^1$, we will first show that $\xi^n=B$. Observe that if
$$[B_m' \xleftarrow{s_m} B_{m-1}][B_{m-1}' \xleftarrow{s_{m-1}} B_{m-2}] \cdots [B_2' \xleftarrow{s_2} B_1][B_1' \xleftarrow{s_1} B] \in \eta,$$
then our assumption implies that $B_k=B_k'$ and $s_k \ne s_{k+1}^{-1}$ for all $1 \le k \le m-1$. In particular, $\vert \Psi_n(\beta) \vert = \vert \beta \vert$ for all $\beta \in \eta$, so that
$$\xi^n=\{\Psi_n(\beta) \mid \beta \in \eta^n\}.$$
Now if
$$\beta=[B_m \xleftarrow{s_m} B_{m-1}][B_{m-1} \xleftarrow{s_{m-1}} B_{m-2}] \cdots [B_2 \xleftarrow{s_2} B_1][B_1 \xleftarrow{s_1} B] \in \eta^n,$$
then of course $\Psi_n(\beta)=s_m \cdots s_1 \in B$, and if, conversely, $s_m \cdots s_1 \in B$ for $m \le n$, then an inductive application of our standing assumption implies the existence of some $\beta$ as above, hence $\xi^n = B$ as desired. Finally, applying our observation to $\beta.\eta$ for some arbitrary
$$\beta=[B_m \xleftarrow{s_m} B_{m-1}][B_{m-1} \xleftarrow{s_{m-1}} B_{m-2}] \cdots [B_2 \xleftarrow{s_2} B_1][B_1 \xleftarrow{s_1} B] \in \eta,$$
we see that $(\Psi_n(\beta).\xi)^n=B_m$, and so $\eta=\phi_n(\xi)$ from the way that $\phi_n$ is defined.

\end{proof}

The next lemma shows that if $n < R$, then we can recover the $R$-ball $\xi^R$ from the $(R-n)$-ball $\phi_n(\xi)^{R-n}$.
\begin{lemma}\label{lem:PsiLemma}
If $\xi \in \mathcal{C}$ and $n < R$, then
$$\xi^R = \bigcup_{\beta \in \phi_n(\xi)^{R-n}}\tilde{\Psi}_n(\beta).$$
\end{lemma}

\begin{proof}
We apply the notation of the above construction with $m \le R-n$. For the inclusion $\supset$, it is enough to check that
$$(s_k \cdots s_1)^{-1}.B_k \subset \xi^R$$
for all $k \le m$, and this holds simply because $B_k \subset (s_k \cdots s_1).\xi^R$. For the reverse inclusion, take $\gamma \in \xi^R$. If $\vert \gamma \vert \le R-n$, then $\gamma \in \Psi_n(\phi_n(\xi)^{R-n}) \subset \bigcup_{\beta \in \phi_n(\xi)^{R-n}}\tilde{\Psi}_n(\beta)$, so we may assume that $\vert \gamma \vert > R-n$ and write $\gamma=\mu\alpha$ with $\vert \alpha \vert = R-n$. But then $\mu \in (\alpha.\xi)^n$, so
$$\gamma =\mu\alpha\in \alpha^{-1}.(\alpha.\xi)^n \subset \bigcup_{\beta \in \phi_n(\xi)^{R-n}}\tilde{\Psi}_n(\beta).$$
\end{proof}

\begin{definition}
Let $n < R$ and $B \in \mathcal{B}_R(\mathcal{C})$; we then define
$$B^{[n]}:=\phi_n(B)^{R-n}=\{\phi_n(B,\alpha) \mid \alpha \in B, \vert \alpha \vert \le R-n\} \in \mathcal{B}_{R-n}(\mathcal{C}^{[n]}).$$
\end{definition}

\begin{lemma}\label{lem:RadiusReduction}
If $n < R$, $B \in \mathcal{B}_R(\mathcal{C})$ and $\xi \in \mathcal{C}$, then
$$\xi^R=B \text{ if and only if } \phi_n(\xi)^{R-n}=B^{[n]}.$$
\end{lemma}

\begin{proof}
Assuming $\xi^R=B$, we immediately see that
\begin{align*}
\phi_n(\xi)^{R-n} &=\{\phi_n(\xi,\alpha) \mid \alpha \in \xi\}^{R-n}=\{\phi_n(\xi,\alpha) \mid \alpha \in \xi, \vert \alpha \vert \le R-n\} \\
&= \{ \phi_n(B,\alpha) \mid \alpha \in B, \vert \alpha \vert \le R-n\}=B^{[n]}.
\end{align*}
On the other hand, if $\phi_n(\xi)^{R-n}=B^{[n]}$, then
$$\xi^R =  \bigcup_{\beta \in \phi_n(\xi)^{R-n}}\tilde{\Psi}_n(\beta)= \bigcup_{\beta \in \phi_n(B)^{R-n}}\tilde{\Psi}_n(\beta)=B^R=B$$
by Lemma~\ref{lem:PsiLemma}.
\end{proof}

\begin{corollary}\label{cor:StepReduction}
The allowed $R$-balls of the $n$-ball shift of a convex subshift $\theta \colon \mathbb{F} \act \Omega$ is given by
$$\mathcal{B}_R(\Omega^{[n]})=\{B^{[n]} \mid B \in \mathcal{B}_{R+n}(\Omega)\}.$$
Moreover, if $\Omega$ is $R$-step and $n < R$, then $\Omega^{[n]}$ is $(R-n)$-step. In particular, $\Omega^{[R-1]}$ is $1$-step.
\end{corollary}

\begin{proof}
The first claim follows immediately from Lemma~\ref{lem:RadiusReduction}. For the second one, we must also refer to Lemma~\ref{lem:NBallOfFull}.
\end{proof}

We can now prove that every finite type convex subshift can be recoded into a $1$-step convex subshift.

\begin{proposition}\label{prop:QCto1Step}
If $\theta \colon \mathbb{F} \act \Omega$ is a convex subshift and $n \ge 1$, then $\theta^{[n]} \xrightarrow{\approx} \theta$.
In particular, any finite type convex subshift is directly dynamically equivalent to a $1$-step convex subshift.
\end{proposition}

\begin{proof}
We have already established that the map $\psi_n$ is a $\Psi_n$-equivariant homeomorphism $\Omega^{[n]} \to \Omega$. Now assume that
$$\beta=[B_m \xleftarrow{s_m} B_{m-1}] \cdot [B_{m-1} \xleftarrow{s_{m-1}} B_{m-2}] \cdots [B_1 \xleftarrow{s_1} B_0]$$
and
$$\beta'=[B_m' \xleftarrow{s_m} B_{m-1}'] \cdot [B_{m-1}' \xleftarrow{s_{m-1}} B_{m-2}'] \cdots [B_1' \xleftarrow{s_1} B_0']$$
are distinct elements of $\mathbb{F}^{[n \colon \Omega]}$ satisfying $\Psi_n(\beta)=\Psi_n(\beta')$. Then $B_k \ne B_k'$ for some $k$, and
$$\big((s_k \cdots s_1).\psi_n(\eta) \big)^n=B_k \ne B_k'= \big((s_k \cdots s_1).\psi_n(\eta') \big)^n$$
for all $\eta \in \Omega^{[n]}_{\beta}$, $\eta' \in
\Omega^{[n]}_{\beta'}$, hence $\Omega^{[n]}_{\beta} \cap
\Omega^{[n]}_{\beta'} = \emptyset$ as desired. Finally, given any
$\alpha \in \mathbb{F}$ and $\xi \in \Omega_\alpha$, we can consider
$\phi_n(\xi) \in \Omega^{[n]}$. Then $\beta := \phi_n(\xi,\alpha)
\in \phi_n(\xi)$ satisfies $\Psi_n(\beta)=\alpha$ and
$$\xi=\psi_n(\phi_n(\xi)) \in \psi_n(\Omega_\beta^{[n]}),$$
hence $\Omega_\alpha = \bigcup_{\Psi_n(\gamma)=\alpha}
\psi_n (\Omega^{[n]}_\gamma )$. We conclude that $\theta^{[n]}
\xrightarrow{\approx} \theta$. The second part of the claim now
follows directly from Proposition~\ref{prop:DirQCImpliesQC} and
Corollary~\ref{cor:StepReduction}.
\end{proof}

One issue we have not yet dealt with is the identification of the $n$-ball shift and the $n$-fold $1$-ball shift of a given convex subshift.

\begin{proposition}\label{prop:Nfold1isN}
Consider any convex subshift $\theta \colon \mathbb{F} \act \Omega$. Then the $n$-ball convex subshift $\theta^{[n]}$ is conjugate to the $n$-fold $1$-ball convex subshift $\theta^{[1] \cdots [1]}$.
\end{proposition}

\begin{proof}
We will show that $\theta^{[n+1]} \cong \theta^{[1][n]}$ for all $n \ge 1$ from which the claim follows inductively. We first define a map $\Phi \colon A^{[n+1\colon \Omega]} \to A^{[1\colon \Omega][n \colon \Omega^{[1]}]}$ by
$$\Phi([B_2 \xleftarrow{s} B_1]) := \big[ B_2^{[1]} \xleftarrow{B_2^1 \xleftarrow{s} B_1^1} B_1^{[1]} \big]=\big[\phi_1(B_2)^n \xleftarrow{\phi_1(B_1,s)} \phi_1(B_1)^n \big],$$
which is easily seen to be a bijection by Lemma~\ref{cor:StepReduction}, so we obtain an induced isomorphism $\Phi \colon \mathbb{F}^{[n+1\colon \Omega]} \to \mathbb{F}^{[1\colon \Omega][n\colon \Omega^{[1]}]}$. The pair $(\varphi,\Phi)$, where $\varphi \colon \mathcal{C}(A^{[n+1 \colon \Omega]}) \to \mathcal{C}(A^{[1\colon \Omega][n\colon \Omega^{[1]}]})$ is given by
$$\varphi(\eta):=\{\Phi(\beta) \mid \beta \in \eta\},$$
then defines a conjugacy. In order to see that it restricts to a conjugacy $\Omega^{[n+1]} \to \Omega^{[1][n]}$, it suffices to check that
$$\varphi(B^{[n+1]})=B^{[1][n]}$$
for all $R \ge 1$ and $B \in \mathcal{B}_{R+n+1}(\Omega)$ due to Corollary~\ref{cor:StepReduction}. Recalling that
$$B^{[n+1]} = \phi_{n+1}(B)^R=\{\phi_{n+1}(B,\alpha) \mid \alpha \in B, \vert \alpha \vert \le R\}$$
and
$$B^{[1][n]}=\phi_n(\phi_1(B)^{R+n})^R = \{\phi_n(\phi_1(B),\phi_1(B,\alpha)) \mid \alpha \in B, \vert \alpha \vert \le R\},$$
in fact we only have to verify that
$$\Phi(\phi_{n+1}(B,\alpha))=\phi_n(\phi_1(B),\phi_1(B,\alpha))$$
for all $\alpha \in B$ with $\vert \alpha \vert \le R$. Writing $\alpha = s_m \cdots s_1$ and
$$t_k:=\big[(s_k \cdots s_1.B)^1 \xleftarrow{s_k} (s_{k-1} \cdots s_1.B)^1\big]=\phi_1(s_{k-1} \cdots s_1.B,s_k)$$
so that
$$\phi_1(s_k \cdots s_1.B)=t_k \cdots t_1.\phi_1(B) \andspace \phi_1(B,\alpha)=t_m \cdots t_1,$$
we then see that
\begin{align*}
&\Phi(\phi_{n+1}(B,\alpha)) = \Phi\Big( \big[(s_m \cdots s_1.B)^{n+1} \xleftarrow{s_m} (s_{m-1} \cdots s_1.B)^{n+1}\big] \cdots \big[(s_1.B)^{n+1} \xleftarrow{s_1} B^{n+1} \big] \Big) \\
&=\big[\phi_1(s_m \cdots s_1.B)^n \xleftarrow{\phi_1(s_{m-1} \cdots s_1.B,s_m)} \phi_1(s_{m-1} \cdots s_1.B)^n  \big] \cdots \big[\phi_1(s_1.B)^n \xleftarrow{\phi_1(B,s_1)} \phi_1(B)^n \big] \\
&= \big[ (t_m \cdots t_1.\phi_1(B))^n \xleftarrow{t_m} (t_{m-1} \cdots t_1.\phi_1(B))^n\big] \cdots \big[(t_1.\phi_1(B))^n \xleftarrow{t_1} \phi_1(B)^n \big] \\
&=\phi_n(\phi_1(B),t_m \cdots t_1)=\phi_n(\phi_1(B),\phi_1(B,\alpha))
\end{align*}
as desired.
\end{proof}

The following theorem explains the precise relationship between the partial actions $\theta^{(E,C)}$ and $\theta^{(E_n,C^n)}$, adding a dynamical dimension to \cite[Theorem 8.3]{AE}. Note that while \cite[Theorem 5.7]{AE} implies
$$\Lab(E_n,C^n) \cong \Lab(E,C) \andspace \mathcal{O}(E_n,C^n) \cong \mathcal{O}(E,C)$$
for any finite bipartite graph and $n \ge 1$, the corresponding result for $\mathcal{O}^r$ is not immediate.

\begin{theorem}\label{thm:NBallIsNGraph}
If $(E,C)$ is a finite bipartite graph and $n \ge 1$, then $\theta^{(E_n,C^n)}$ is conjugate to the $n$-ball convex subshift $(\theta^{(E,C)})^{[n]}$. In particular, there are base-preserving isomorphisms
$$\Lab(E_n,C^n) \cong \Lab(E,C), \quad \mathcal{O}(E_n,C^n) \cong \mathcal{O}(E,C) \andspace \mathcal{O}^r(E_n,C^n) \cong \mathcal{O}^r(E,C)$$
induced from the direct dynamical equivalence 
$$\theta^{(E_n,C^n)} \cong (\theta^{(E,C)})^{[n]} \xrightarrow{\approx} \theta^{(E,C)}.$$
The isomorphisms on $L_K^{\textup{ab}}$ and $\mathcal{O}$ coincides with the ones induced by $\Phi_n$ as defined in \cite{AE}. Moreover, there is a canonical bijective correspondence 
$$E_n^0 \ni v \mapsto B(v) \in \mathcal{B}_n(\Omega(E,C)),$$
and the homeomorphism $\Omega(E_n,C^n) \to \Omega(E,C)^{[n]} \to \Omega(E,C)$ restricts to a homeomorphism
$$\Omega(E_n,C^n)_v \to \{ \xi \in \Omega(E,C) \mid \xi^n=B(v)\}$$
for all $v \in E_n^0$. 
\end{theorem}

\begin{proof}
Set $A:=E^1$ and $\Omega:=\Omega(E,C)$ for notational simplicity; by Proposition~\ref{prop:Nfold1isN}, it suffices to verify our claims for $n=1$. Observe first that $\mathcal{B}_1(\Omega)=\mathcal{B}_1^0(\Omega) \sqcup \mathcal{B}_1^1(\Omega)$, where
$$\mathcal{B}_1^0(\Omega) := \big\{\{1\} \sqcup \{x_1,\ldots,x_{k_u}\}^{-1} \mid u \in E^{0,0}, x_j \in X_j^u \big\} \andspace \mathcal{B}_1^1(\Omega):=\{\{1\} \sqcup s^{-1}(v) \mid v \in E^{0,1}\},$$
using the standard notation. The alphabet $A^{[1\colon \Omega]}$ is therefore given by
$$A^{[1\colon \Omega]}=\big\{ \big[\{x_1,\ldots,x_{k_u}\}^{-1} \xleftarrow{x_i} s^{-1}(s(x_i)) \big] \mid u \in E^{0,0}, x_j \in X_j^u, i=1,\ldots,k_u \big\}$$
and we can define an isomorphism $\Phi \colon \mathbb{F}^{[1\colon \Omega]} \to \mathbb{F}(E_1^1)$ by
$$\big[\{x_1,\ldots,x_{k_u}\}^{-1} \xleftarrow{x_i} s^{-1}(s(x_i)) \big] \mapsto \alpha^{x_i}(x_1,\ldots,\widehat{x_i},\ldots,x_{k_u})^{-1}.$$
Since both convex shifts are $1$-step, we simply have to check that $\Phi$ induces a bijection
$$\mathcal{B}_1(\Omega^{[1]}) \to \mathcal{B}_1(\Omega(E_1,C^1)), \quad B \mapsto \Phi(B).$$
We first build some notation: Given any $u \in E^{0,0}$ and $(x_1,\ldots,x_{k_u}) \in \prod_{j=1}^{k_u}X_j^u$, we define
$$B_1^0(x_1,\ldots,x_{k_u}):=\Big\{\big[\{1\} \sqcup s^{-1}(s(x_i)) \xleftarrow{x_i^{-1}}\{1\} \sqcup\{x_1,\ldots,x_{k_u}\}^{-1} \big] \colon i=1,\ldots,k_u \Big\}.$$
Also, for all $x_i \in X_i^u$, we set
$$Z(x_i):=\Big\{\{x_1,\ldots,\widehat{x_i},\ldots,x_{k_u}\}^{-1} \mid x_j \in X_j^u, j \ne i  \Big\}$$
and define
$$B_1^1(v,\{z(e)\}_{e \in s^{-1}(v)}):=\Big\{\big[\{1\} \sqcup z(e) \xleftarrow{e} \{1\} \sqcup s^{-1}(v)\big] \colon e \in s^{-1}(v) \Big\}$$
for any $v \in E^{0,1}$ and $z(e) \in Z(e)$. It is easily checked that $\mathcal{B}_1(\Omega^{[1]}) = \mathcal{B}_1^0(\Omega^{[1]}) \sqcup \mathcal{B}_1^1(\Omega^{[1]})$, where
$$\mathcal{B}_1^0(\Omega^{[1]}) := \Big\{ B_1^0(x_1,\ldots,x_{k_u}) \mid u \in E^{0,0}, (x_1,\ldots,x_{k_u}) \in \prod_{j=1}^{k_u} X_j^u \Big\}$$
and
$$\mathcal{B}_1^1(\Omega^{[1]}):=\Big\{B_1^1(v,\{z(e)\}_{e \in s^{-1}(v)}) \mid v \in E^{0,1}, z(e) \in Z(e) \Big\}.$$
Now observe that
\begin{align*}
\Phi(B_1^0(x_1,\ldots,x_{k_u})) &=\{1\} \sqcup \big\{\alpha^{x_i}(x_1,\ldots,\widehat{x_i},\ldots,x_{k_u}) \mid i = 1,\ldots,k_u \big\} \\
&=\{1\} \sqcup s_1^{-1}(v(x_1,\ldots,x_{k_u}))
\end{align*}
hence $\Phi$ maps $B_1^0(\Omega^{[1]})$ onto the $1$-balls of type
(c1) (cf. Definition \ref{def:DynamicalPicture}). Likewise, we see that
$$\Phi(B_1^1(v,\{z(e)\}_{e \in s^{-1}(v)}))=\{1\} \sqcup \big\{ \alpha^e(z(e))^{-1} \mid e \in s^{-1}(v) \big\},$$
so $\Phi$ maps $B_1^1(\Omega^{[1]})$ onto the $1$-balls of type (c2). We conclude that the pair $(\varphi,\Phi)$ with $\varphi \colon \Omega^{[1]} \to \Omega(E_1,C^1)$ given by
$$\varphi(\xi):=\{\Phi(\alpha) \mid \alpha \in \xi\}$$
is a conjugacy of $\mathbb{F}^{[1\colon \Omega]} \act \Omega^{[1]}$ and $\mathbb{F}(E_1^1) \act \Omega(E_1,C^1)$.
Consequently, there is a direct dynamical equivalence $\theta^{(E_n,C^n)} \cong (\theta^{(E,C)})^{[n]} \xrightarrow{\approx} \theta^{(E,C)}$, which induces 
base-preserving isomorphisms by Remark \ref{rem:GroupoidAlgebras}. It follows from Remark~\ref{rem:DirDynEq} that the isomorphisms on $\Lab$ and $\mathcal{O}$ are exactly the ones of \cite{AE}. \medskip \\
We now turn to the second part of the claim and set
$$B(v):=\left\{\begin{array}{cl}
\{1\} \sqcup s^{-1}(v) & \If v \in E_1^{0,0}=E^{0,1} \\
\{1\} \sqcup \{x_1,\ldots,x_{k_u}\}^{-1} & \If v=v(x_1,\ldots,x_{k_u}) \in E_1^{0,1}
\end{array}\right.$$
for every $v \in E_1^0$; this clearly defines a bijective
correspondence between $E_1^0$ and $\mathcal{B}_1(\Omega(E,C))$.
Note that the homeomorphism $\Omega(E_1,C^1) \to \Omega(E,C)$ is
induced by the group homomorphism $\mathbb{F}(E_1^1) \to
\mathbb{F}(E^1)$ given by $\alpha^{x_i}(x_1,\ldots,\widehat{x_i},\ldots,x_{k_u}) \mapsto x_i^{-1}$, and it is easily checked that it
maps $\Omega(E_1,C^1)_v$ onto $\{\xi \in \Omega(E,C) \mid
\xi^1=B(v)\}$.
\end{proof}

\begin{remark}\label{rem:Omegav}
Theorem~\ref{thm:NBallIsNGraph} shows that for $v \in E_n^0$ with $n \ge 1$, the subspace $\Omega(E,C)_v$ may be described as $\Omega(E,C)_v = \{ \xi \in \Omega(E,C) \mid \xi^n=B(v)\}$.
\end{remark}

Any partial action on a topological space may be viewed as the restriction of a global action \cite[Theorem 2.5]{Abadie}. The globalisation is not Hausdorff in general \cite[Proposition 2.10]{Abadie}, but whenever it is, one may consider the relationship between the $C^*$-algebras of the partial action and its globalisation. However, it is also of natural interest to study restrictions of \textit{partial actions}, in particular in cases where there is no Hausdorff globalisation, and they play a natural role in our main theorem about convex subshifts.

\begin{definition}[{See also \cite[Definition 3.1]{CRS}, \cite[Definition 2.17]{Li2}}]
If $\theta \colon G \act \Omega$ is a partial action on a topological space and $U \subset \Omega$ is an open subset, then we denote by $\theta\vert_U$ the restricted partial action $G \act U$ with domains
$$U_g:=\theta_g(U \cap \Omega_{g^{-1}}) \cap U,$$
and $U$ is called $G$\textit{-full} if
$$X=\{g.x \mid g \in G, x \in U \cap X_{g^{-1}}\}.$$
Finally, two partial actions $\theta \colon G \act \Omega$ and $\gamma \colon H \act \Omega'$ are called \textit{Kakutani-equivalent} if there exist clopen subspaces $K \subset \Omega$ and $K' \subset \Omega'$, resp. $G$- and $H$-full, such that $\theta\vert_K \approx \gamma\vert_{K'}$.
\end{definition}

If $\theta \colon G \act \Omega$ and $\theta' \colon G \act \Omega'$ are Kakutani equivalent partial actions on totally disconnected, locally compact spaces, then the groupoids $\mathcal{G}_\theta$ and $\mathcal{G}_{\theta'}$ are Kakutani equivalent in the sense of \cite[Definition 3.1]{CRS} and hence groupoid equivalent by \cite[Theorem 3.2]{CRS}. It follows that there are Morita-equivalences
$$C_K(\Omega) \rtimes_\theta G \sim_M C_K(\Omega) \rtimes_{\theta'} H \andspace C_0(\Omega) \rtimes_{\theta,(r)} \sim_M C_0(\Omega') \rtimes_{\theta',(r)} H,$$
see for instance \cite[Theorem 2.8]{MRW}, \cite[Theorem 13]{SW} and \cite[Theorm 5.1]{CS}.

We can now state and prove the second main theorem about convex subshifts.

\begin{theorem}\label{thm:GraphRep}
If $\theta \colon \mathbb{F} \act \Omega$ is a convex subshift of finite type, then there is a finite bipartite separated graph $(E,C)$ such that $\theta^{(E,C)}$ and $\theta$ are Kakutani equivalent.
\end{theorem}

\begin{proof}
By Proposition~\ref{prop:QCto1Step}, we may assume that $\Omega$ is $1$-step. Now let $A^{[1\colon \Omega]}_+$ and $A^{[1\colon \Omega]}_-$ denote disjoint copies of the alphabet $A^{[1\colon \Omega]}$ with subscripts $+$ and $-$, and define a finite bipartite separated graph $(E,C)$ by
\begin{itemize}
\item $E^{0,0}:=\mathcal{B}_1(\Omega)$ and $E^{0,1}=A^{[1\colon \Omega]}$,
\item $E^1:=A^{[1\colon \Omega]}_+ \sqcup A^{[1\colon \Omega]}_-$,
\item $r([B \xleftarrow{a} B']_+):=B$ and $r([B \xleftarrow{a} B']_-):=B'$,
\item $s([B \xleftarrow{a} B']_+):=s([B \xleftarrow{a} B']_-)=[B \xleftarrow{a} B']$,
\item $C_B:=\{X_B(s) \mid 1 \ne s \in B\}$ for all $B \in E^{0,0}$, where for $a \in A$
\begin{align*}
X_B(a^{-1})& :=\{[B \xleftarrow{a} B']_+ \mid B' \in E^{0,0} \text{ such that } B \xleftarrow{a} B'\}, \\
X_B(a)& :=\{[B' \xleftarrow{a} B]_- \mid B' \in E^{0,0} \text{ such that } B' \xleftarrow{a} B\}.
\end{align*}
\end{itemize}
Then consider the group homomorphism $\Phi \colon  \mathbb{F}(E^1) \to \mathbb{F}^{[1\colon \Omega]}$ given by
$$\Phi([B \xleftarrow{a} B']_+):=[B \xleftarrow{a} B'] \andspace \Phi([B \xleftarrow{a} B']_-):=1$$
as well as the compact open $\mathbb{F}(E^1)$-full subspace
$$K:=\bigsqcup_{u \in E^{0,0}} \Omega(E,C)_u.$$
Equipping $K$ with the restricted partial action $\theta^{(E,C)} \vert_K \colon \mathbb{F}(E^1) \act K$, we claim that $\varphi \colon K \to \Omega^{[1]}$ given by
$$\varphi(\xi):=\{\Phi(\alpha) \mid \alpha \in \xi\}$$
defines a $\Phi$-equivariant homeomorphism, making the pair $(\varphi,\Phi)$ into a direct dynamical equivalence. First, let us check that $\varphi$ even maps into $\Omega^{[1]}$, so take any $\xi \in K$. Observe that any length two admissible path $\alpha \in \xi$ is of the form
$$[B \xleftarrow{a} B']_+ [B \xleftarrow{a} B']_-^{-1} \quad \text{or} \quad [B \xleftarrow{a} B']_- [B \xleftarrow{a} B']_+^{-1},$$
and these are mapped to $[B \xleftarrow{a} B']$ and $[B \xleftarrow{a} B']^{-1}$, respectively. It follows that any $\alpha \in \xi$ of length four is mapped to a length two word, hence any $\alpha \in \xi$ of length $2n$ is mapped to a word of length $n$. Note also that if $\alpha \in \xi$ has odd length, then
$$\alpha=[B \xleftarrow{a} B']_+^{-1}\beta \quad \text{or} \quad \alpha=[B \xleftarrow{a} B']_-^{-1}\beta$$
for some $[B \xleftarrow{a} B'] \in A^{[1\colon \Omega]}$, and in the latter case, $\Phi(\alpha)=\Phi(\beta)$. In the former case, there is a unique extension of length $\vert \alpha \vert + 1$ inside $\xi$, namely $[B \xleftarrow{a} B']_-[B \xleftarrow{a} B']_+^{-1}\beta \in \xi$, and $\Phi(\alpha)=\Phi([B \xleftarrow{a} B']_-\alpha)$, so in conclusion
$$\varphi(\xi)=\{\Phi(\alpha) \colon \alpha \in \xi, \vert \alpha \vert \text{ is even}\}.$$
In particular, $\varphi$ is continuous and $\varphi(\xi)^1=\Phi(\xi^2)$, so we only need to check that $\Phi(\xi^2) \in \mathcal{B}_1(\Omega^{[1]})$ for any $\xi \in K$. Assuming that $\xi \in \Omega(E,C)_B$, for every $1 \ne s \in B$ there is $B(s) \in \mathcal{B}_1(\Omega)$ such that
\begin{align*}
\Phi(\xi^2)&=\{1\} \sqcup \{[B \xleftarrow{a} B(a^{-1})]^{-1} \mid a \in A \cap B^{-1}\} \sqcup \{[B(a) \xleftarrow{a} B] \mid a \in A \cap B\} \\
&=\{1\} \sqcup \{[B(s) \xleftarrow{s} B] \mid 1 \ne s \in B\},
\end{align*}
and this is exactly an element of $\mathcal{B}_1(\Omega^{[1]})$. At this point we have verified that $\varphi$ is a well-defined continuous $\Phi$-equivariant map, and we now turn to the construction of an inverse. Define a group homomorphism $\Sigma \colon \mathbb{F}^{[1\colon \Omega]} \to \mathbb{F}(E^1)$ and a continuous $\Sigma$-equivariant map $\sigma \colon \Omega^{[1]} \to K$ by
$$\Sigma([B \xleftarrow{a} B']):=[B \xleftarrow{a} B']_+[B \xleftarrow{a} B']_-^{-1} \andspace \sigma(\eta):=\text{conv}\{\Sigma(\beta) \mid \beta \in \eta\},$$
where $\text{conv}(H)$ for a set $H \subset \mathbb{F}(E^1)$ denotes the convex closure. Observing that $\Phi$ is a one-sided inverse of $\Sigma$, it follows that $\Sigma$ is injective and hence that $\sigma$ is a continuous $\Sigma$-equivariant map into $\mathcal{C}(E^1)$; but, we still have to verify that $\sigma$ maps into $K$. Since $\sigma(\eta)^2=\text{conv}(\Sigma(\eta^1))$ for all $\eta \in \Omega^{[1]}$, 
it suffices to check that $\text{conv}(\Sigma(\eta^1)) \in \mathcal{B}_2(\Omega(E,C))$. By construction, $\eta^1$ is of the form
\begin{align*}
\eta^1 &=\{1\} \sqcup \{[B(s) \xleftarrow{s} B] \mid 1 \ne s \in B\} \\
&=\{1\} \sqcup \{[B \xleftarrow{a} B(a^{-1})]^{-1} \mid a \in A \cap B^{-1}\} \sqcup \{[B(a) \xleftarrow{a} B] \mid a \in A \cap B\}
\end{align*}
for some $B,B(s) \in \mathcal{B}_1(\Omega)$, so
\begin{align*}
\Sigma(\eta^1)&=\{1\} \sqcup \{[B \xleftarrow{a} B(a^{-1})]_-[B \xleftarrow{a} B(a^{-1})]_+^{-1} \mid a \in A \cap B^{-1}\} \\
&{} \quad \sqcup \{[B(a) \xleftarrow{a} B]_+[B(a) \xleftarrow{a} B]_-^{-1} \mid a \in A \cap B\}.
\end{align*}
Taking the convex closure of this, we clearly obtain a $2$-ball of $\Omega(E,C)$, hence
$$\sigma(\eta) \in \Omega(E,C)_B \subset K$$
as desired. We now claim that $\varphi$ and $\sigma$ are in fact mutual inverses. Noting that $\Sigma(\Phi(\alpha))=\alpha$ for $\alpha \in \xi$ of even length, we indeed have
$$\sigma(\varphi(\xi))=\text{conv}\{\Sigma(\Phi(\alpha)) \mid \alpha \in \xi, \vert \alpha \vert \text{ is even}\}=\text{conv}\{\alpha \mid \alpha \in \xi, \vert \alpha \vert \text{ is even}\} = \xi$$
and
$$\varphi(\sigma(\eta))=\varphi(\text{conv}\{\Sigma(\beta) \mid \beta \in \eta\})=\{\Phi(\Sigma(\beta)) \mid \beta \in \eta\}=\eta.$$
Letting $F:=\text{Im}(\Sigma) \le \mathbb{F}(E^1)$ so that $\Sigma=\Phi\vert_F^{-1}$, we conclude that the partial actions $F \act K$ and $\mathbb{F}^{[1\colon \Omega]} \act \Omega^{[1]}$ are conjugate. Finally observing that, by the above observations, $K_\alpha = \emptyset$ for all $\alpha \in \mathbb{F}(E^1) \setminus F$, we conclude that $(\varphi,\Phi)$ is indeed a direct dynamical equivalence $\theta^{(E,C)}\vert_K \xrightarrow{\approx} \theta^{[1]}$, from which we obtain the desired direct dynamical equivalence as the composition $\theta^{(E,C)}\vert_K \xrightarrow{\approx} \theta^{[1]} \xrightarrow{\approx} \theta$.
\end{proof}

In view of the above theorem, the study of convex subshifts of finite type boils down to the study of dynamical systems associated with finite bipartite graphs, at least up to Katutani equivalence. In the following sections, we shall see how one can extract information about the open/closed invariant subspaces from the graph, illustrating the usefulness of having a graph representation.

We end this section with an application of Theorem~\ref{thm:GraphRep} to a pair of concrete examples.

\begin{example}\label{ex:FullConvexShift}
Given a finite alphabet $A$, there is of course a finite bipartite separated graph $(E,C)$ corresponding to the full convex shift on $A$ as in Theorem~\ref{thm:GraphRep}. However, one can check that $\vert E^0 \vert= 4( \vert A \vert^4 + \vert A \vert^2)$ and $\vert E^1 \vert = 8 \vert A \vert^4$, so even when $\vert A \vert=2$ this is a fairly sizable graph. We shall therefore refrain from drawing it here.
\end{example}

\begin{example}
Consider the alphabet $A=\{a,b\}$ and the $1$-step subshift $\mathbb{F}_2 \act \Omega$ with $\mathcal{B}_1(\Omega)=\{u,v\}$, where $u=\{1,a^{\pm 1},b^{\pm 1}\}$ and $v=\{1,a^{\pm 1},b\}$ as illustrated just below.
\begin{figure}[htb]
\begin{center}\begin{tikzpicture}[scale=0.5]
 \SetUpEdge[lw         = 1.5pt,
            labelcolor = white]

  \tikzset{VertexStyle/.style = {}}
  \Vertex[x=-4,y=0,L=${u=}$]{u}
  \tikzset{VertexStyle/.style = { shape = circle,fill = black, minimum size=4pt, inner sep=0pt,outer sep=2pt}}

  \SetVertexNoLabel

  \Vertex[x=0,y=3]{t}
  \Vertex[x=0,y=-3]{b}  
  \Vertex[x=-3,y=0]{l}
  \Vertex[x=3,y=0]{r}  
  
  \tikzset{VertexStyle/.style = { shape = circle,fill = \nicered, minimum size=4pt, inner sep=0pt,outer sep=2pt}}  
  \Vertex[x=0,y=0]{c}  

  \tikzset{EdgeStyle/.style = {->,}}
  \Edge[label=$a$](l)(c)
  \Edge[label=$b$](b)(c)
  \Edge[label=$a$](c)(r)  
  \Edge[label=$b$](c)(t)  

\end{tikzpicture} \hspace{1cm}
\begin{tikzpicture}[scale=0.5]
 \SetUpEdge[lw         = 1.5pt,
            labelcolor = white]
  \tikzset{VertexStyle/.style = {}}
  \Vertex[x=-4,y=0,L=${v=}$]{v}
  
  \tikzset{VertexStyle/.style = { shape = circle,fill = black, minimum size=4pt, inner sep=0pt,outer sep=2pt}}
  \SetVertexNoLabel 
  \Vertex[x=0,y=3]{t}
  \Vertex[x=-3,y=0]{l}
  \Vertex[x=3,y=0]{r} 

  \tikzset{VertexStyle/.style = { shape = circle,fill = \nicered, minimum size=4pt, inner sep=0pt,outer sep=2pt}}  
  \Vertex[x=0,y=0]{c}    
     
  \tikzset{VertexStyle/.style = { shape = circle,fill = white, minimum size=5pt, inner sep=0pt,outer sep=3pt}}
  \Vertex[x=0,y=-3]{b}  

  \tikzset{EdgeStyle/.style = {->,}}
  \Edge[label=$a$](l)(c)
  \Edge[label=$a$](c)(r)  
  \Edge[label=$b$](c)(t)  

\end{tikzpicture}
\end{center}
\end{figure}
We will then describe the separated graph $(E,C)$ of Theorem~\ref{thm:GraphRep}. We have
$$E^{0,1}=A^{[1\colon \Omega]}=\{[u \xleftarrow{a} u],[u \xleftarrow{b} u],[v \xleftarrow{a} u],[u \xleftarrow{a} v],[u \xleftarrow{b} v],[v \xleftarrow{a} v]\}$$
and $E^1=A^{[1\colon \Omega]}_+ \sqcup A^{[1\colon \Omega]}_-$ with
\begin{align*}
r^{-1}(u)&=\{[u \xleftarrow{a} u]_\pm,[u \xleftarrow{b} u]_\pm,[v \xleftarrow{a} u]_-,[u \xleftarrow{a} v]_+,[u \xleftarrow{b} v]_+\}, \\
r^{-1}(v)&=\{[v \xleftarrow{a} u]_+,[u \xleftarrow{a} v]_-,[u \xleftarrow{b} v]_-,[v \xleftarrow{a} v]_\pm\},
\end{align*}
and the source map is simply the projection $E^1=A^{[1\colon \Omega]}_+ \sqcup A^{[1\colon \Omega]}_- \to A^{[1\colon \Omega]}=E^{0,1}$. The separation is given by
$$C_u=\{X_u(a),X_u(a^{-1}),X_u(b),X_u(b^{-1})\} \andspace C_v=\{X_v(a),X_v(a^{-1}),X_v(b)\},$$
where
\begin{align*}
X_u(a)&=\{[u \xleftarrow{a} u]_-,[v \xleftarrow{a} u]_-\},
&X_u(a^{-1})=\{[u \xleftarrow{a} u]_+,[u \xleftarrow{a} v]_+\}, \\
X_u(b)&=\{[u \xleftarrow{b} u]_-\},
&X_u(b^{-1})=\{[u \xleftarrow{b} u]_+,[u \xleftarrow{b} v]_+\}, \\
X_v(a)&=\{[u \xleftarrow{a} v]_-,[v \xleftarrow{a} v]_-\},
&X_v(a^{-1})=\{[v \xleftarrow{a} u]_+,[v \xleftarrow{a} v]_+\}, \\
X_v(b)&=\{[u \xleftarrow{b} v]_- \}\ .
\end{align*}
We can therefore picture $(E,C)$ as follows:
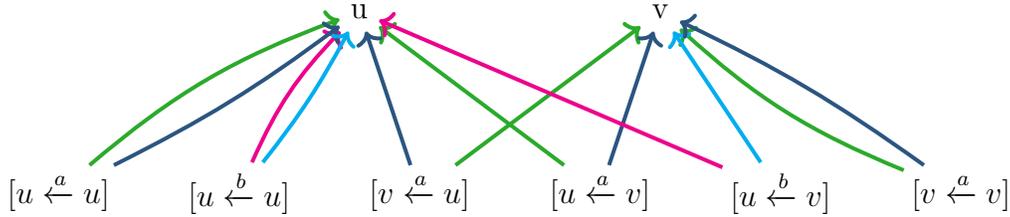
\begin{center}\begin{tikzpicture}[scale=0.8]
 \SetUpEdge[lw         = 1.5pt,
            labelcolor = white]
  \tikzset{VertexStyle/.style = { shape = rectangle,fill = white, minimum size=15pt, inner sep=3pt,outer sep=0pt}}

  \Vertex[x=5,y=3]{u}
  \Vertex[x=10,y=3]{v}
  \Vertex[L=${[u \xleftarrow{a} u]}$,x=0,y=0]{w1}
  \Vertex[L=${[u \xleftarrow{b} u]}$,x=3,y=0]{w2}
  \Vertex[L=${[v \xleftarrow{a} u]}$,x=6,y=0]{w3}
  \Vertex[L=${[u \xleftarrow{a} v]}$,x=9,y=0]{w4}
  \Vertex[L=${[u \xleftarrow{b} v]}$,x=12,y=0]{w5}
  \Vertex[L=${[v \xleftarrow{a} v]}$,x=15,y=0]{w6}

  \tikzset{EdgeStyle/.style = {->,bend left=10,color={\nicegreen}}}
  \Edge[](w1)(u)
  \Edge[](w6)(v)

  \tikzset{EdgeStyle/.style = {->,bend left=10,color={magenta}}}
  \Edge[](w2)(u)

  \tikzset{EdgeStyle/.style = {->,color={\nicegreen}}}
  \Edge[](w4)(u)
  \Edge[](w3)(v)

  \tikzset{EdgeStyle/.style = {->,color={\niceblue}}}
  \Edge[](w3)(u)
  \Edge[](w4)(v)

  \tikzset{EdgeStyle/.style = {->,color={cyan}}}
  \Edge[](w5)(v)

  \tikzset{EdgeStyle/.style = {->,color={magenta}}}
  \Edge[](w5)(u)
  \tikzset{EdgeStyle/.style = {->,bend right=5,color={\niceblue}}}
  \Edge[](w1)(u)
  \Edge[](w6)(v)
  \tikzset{EdgeStyle/.style = {->,bend right=5,color={cyan}}}
  \Edge[](w2)(u)
\end{tikzpicture}
\end{center}\vspace{-0.5cm}
\end{example}

\section{The lattice of induced ideals}
\label{sec:structure-ideals}

In this section, we describe the lattice of induced ideals of the
algebras $\Lab (E,C)$, $\mathcal O (E,C)$ and $\mathcal O ^r (E,C)$
for $(E,C)$ finite and bipartite in terms of graph-theoretic data,
specifically certain sets of vertices in the infinite layer graph
$(F_\infty,D^\infty)$ (see Definition \ref{def:Finftyandothers}).

We first settle the meaning of the various types of ideals that we shall encounter. When dealing with C*-algebras, we will consider only closed ideals, so that the word \textit{ideal} will 
mean \textit{closed ideal} in this case. For a general ring $R$, an ideal $I \ideal R$ is called a \textit{trace ideal} if it is generated by the entries of some set of idempotent 
matrices over $R$, and we denote the lattice of trace ideals by $\Tr(R)$. The lattice of idempotent-generated ideals $\Idem(R)$ then sits as a sublattice of $\Tr(R)$. Given a crossed product (algebraic or $C^*$-algebraic, 
reduced or universal) $\mathcal{O}=A \rtimes_{(r)} G$, we say that an ideal $J \ideal \mathcal{O}$ is \textit{induced} if $J=(J \cap A) \rtimes_{(r)} G$, and we denote 
by $\text{Ind}(\mathcal{O})$ the 
lattice of induced ideals. Finally, if $M$ is an abelian monoid, then a submonoid $I \subset M$ is called an \textit{order-ideal} if $x+y \in I$ implies $x,y \in I$. The lattice of order ideals of $M$ will be denoted by $\mathcal{L}(M)$.

The basic tools in our analysis are the following results.

\begin{theorem} \cite[Proposition 10.10]{AG}
\label{thm:ideals-alg}
For any ring $R$, there is a lattice isomorphism
$$ \mathcal L (\mathcal V (R))) \cong \Tr (R).$$
Moreover, if $\mathcal V (R)$ is generated by the classes $[e]$ of idempotents of $R$, then $\Tr(R)=\Idem (R)$.
\end{theorem}

\begin{theorem}[\cite{Exel}]
 \label{thm:ideals-reducedcrossed}
Let $G \act A$ be a partial action of a discrete group $G$ on a $C^*$-algebra. Then the map $J \mapsto J \cap A$ defines a bijective correspondence between $\textup{Ind}(A \rtimes_{(r)} G)$ and the lattice of invariant ideals of $A$. Moreover, if $I$ is an invariant ideal of $A$, then
$$(A \rtimes G)/(I \rtimes G) \cong (A/I) \rtimes G,$$
and if $G$ is exact, then $(A \rtimes_r G)/(I \rtimes_r G) \cong (A/I) \rtimes_r G$ as well.
 \end{theorem}

 \begin{proof}
If $J$ is an ideal of $A\rtimes_{(r)} G$, then $J \cap A$ is a $G$-invariant ideal of $A$ by \cite[Proposition 23.11]{Exel}, and if $I$ is an invariant ideal of $A$, then $I \rtimes_{(r)} G$ is an ideal of $A \rtimes_{(r)} G$ by \cite[Proposition 21.12 and 21.15]{Exel}. In particular, we have the above mentioned bijective correspondence. Moreover, $(A \rtimes G)/(I \rtimes G) \cong (A/I) \rtimes G$ by \cite[Proposition 21.15]{Exel}, and in case $G$ is exact, then $(A \rtimes_r G)/(I \rtimes_r G) \cong (A/I) \rtimes_r G$ as well by \cite[Theorem 21.18]{Exel}.
 \end{proof}

\begin{remark}\label{rem:AlgCrPrIndId}
We note that it is straightforward to prove a result completely analogous to Theorem~\ref{thm:ideals-reducedcrossed} for partial actions on $*$-algebras.
\end{remark}

We now recall some definitions from \cite{AG}, adapted to our choice of conventions. These notions generalize the classical notions of hereditary and saturated subsets of vertices, cf. \cite[Chapter 4]{Raeburn}, to the separated setting.

\begin{definition} \cite[Definition 6.3]{AG}
Let $(E,C)$ be a finitely separated graph. A subset $H$ of $E^0$ is said to be {\it hereditary} if for any $e\in E^1$, we have
$r(e)\in H$ implies $s(e) \in H$, and $H$ is said to be {\it $C$-saturated} if for any $v\in E^0$ and $X\in C_v$, $s(X)\subset H$ implies $v\in H$. We denote by $\mathcal H (E,C)$ the lattice of hereditary $C$-saturated subsets of $E^0$.
\end{definition}

\begin{theorem}
 \label{thm:VidealsLab}
 Let $(E,C)$ be a finite bipartite separated graph. Then there are lattice isomorphisms
 $$\Idem (\Lab_K (E,C)) \cong \mathcal L (M(F_{\infty}, D^{\infty})) \cong \mathcal H  (F_{\infty}, D^{\infty}) . $$
If $H\in \mathcal H (F_{\infty}, D^{\infty})$, the ideal $I(H)$ of $\Lab _K(E,C)$ associated to $H$ through this isomorphism is the ideal generated by all the projections
$\pi_{n,\infty} (v)$, where $v\in H\cap E_n^0$, and $\pi_{n,\infty}\colon L(E_n,C^n)\to \Lab_K (E,C)$ is the natural quotient map.
 \end{theorem}

\begin{proof}
 By \cite[Corollary 5.9]{AE}, we have an isomorphism $\mathcal V (\Lab (E,C))\cong M(F_{\infty}, D^{\infty})$. In particular,
 $\mathcal V (\Lab (E,C))$ is generated by the classes of the idempotents in $\Lab (E,C)$ corresponding to the vertices in
 $F_{\infty}$. By Theorem \ref{thm:ideals-alg}, we obtain
 $$\mathcal L (M(F_{\infty}, D^{\infty})) \cong \mathcal L (\mathcal V (\Lab (E,C))) \cong \Tr (\Lab (E,C)) = \Idem(\Lab (E,C)).$$
 On the other hand, by \cite[Corollary 6.10]{AG}, we have $\mathcal L (M(F_{\infty}, D^{\infty}))\cong \mathcal H (F_{\infty}, D^{\infty})$, so that we finally obtain
 a lattice  isomorphism $\Idem (\Lab (E,C)) \cong \mathcal H (F_{\infty}, D^{\infty})$.
 \end{proof}

Let $\Omega$ be a zero-dimensional metrizable locally compact Hausdorff
space, and let $\mathbb K=\mathbb K (\Omega)$ be the subalgebra of
$\mathcal P (\Omega)$ consisting of all the compact open subsets of $\Omega$.
Let $\theta \colon G \act \Omega$ be a partial action of a discrete group
$G$ by continuous transformations on $\Omega$ such that $\Omega_g\in \mathbb
K$ for all $g\in G$. Observe that $\mathbb K$ is then automatically
$G$-invariant. The (relative) type semigroup $S(\Omega,G,\mathbb K)$ has
been defined in \cite[Definition 7.1]{AE}, see also \cite{KN} and
\cite{RS}. The semigroup $S(\Omega,G,\mathbb K)$ is indeed a conical
refinement monoid, and we obtain the following description of its
lattice $\mathcal L (S(\Omega,G,\mathbb K))$ of order ideals.

\begin{lemma}
 \label{lem:order-ideals}
 Let $\theta \colon G \act \Omega$ be a partial action of a discrete group $G$ by
continuous transformations on $\Omega$ such that $\Omega_g\in \mathbb K$ for
all $g\in G$. Then there are mutually inverse, order-preserving maps
$$\varphi \colon \mathcal L (S(\Omega,G, \mathbb K)) \to \mathbb O ^G(\Omega),\qquad \psi \colon \mathbb O ^G (\Omega)\to \mathcal L (S(\Omega,G, \mathbb K)),$$
$$\varphi (I)= \bigcup \{ K\in \mathbb K\mid [K]\in I \}, \qquad  \psi (U)= \langle [K] \mid K\in \mathbb K, K\subseteq U \rangle, $$
where $\mathbb O ^G(\Omega)$ is the lattice of $G$-invariant open subsets of $\Omega$, and, for $T\subseteq S(\Omega,G, \mathbb K)$, $\langle T \rangle$ stands for the order ideal
of $S(\Omega,G, \mathbb K)$ generated by $T$.
 \end{lemma}

\begin{proof}
Write $S:=S(\Omega,G, \mathbb K)$ and take $I\in \mathcal L (S)$. Clearly
$U:=\varphi (I)$ is an open subset of $\Omega$. If $x\in \Omega_{g^{-1}}\cap U$
for some $g\in G$, then there is $K\in \mathbb K$ with $[K]\in I$
such that $x\in X_{g^{-1}}\cap K$. But now we have $\theta_g(x) \in
\theta_g(\Omega_{g^{-1}}\cap K)$ with
$$ [\theta_g(\Omega_{g^{-1}}\cap K)] = [\Omega_{g^{-1}} \cap K] \le [K] \in I ,$$
and so $[\theta_g(\Omega_{g^{-1}}\cap K)]\in I$ because $I$ is an order ideal of $S$. It follows that $U$ is $G$-invariant.

Let $U$ be an invariant open subset of $\Omega$. Then $\psi (U)\in \mathcal L (S)$ by definition of $\psi$. It is clear that $\varphi$ and $\psi$ are order-preserving maps. We have to show that $(\varphi \circ \psi)(U)= U$ and $(\psi \circ \varphi ) (I)=I$ for $U\in \mathbb O ^G(X)$ and $I\in \mathcal L (S)$. For $U\in \mathbb O^G(\Omega)$,  let $K\in \mathbb K$ be such that $[K]\in \psi (U)$. Then there are $K_1,\dots , K_r\in \mathbb K$ such that $K_j\subseteq U$ for $j=1,\dots , r$ and
$$ [K]\le [K_1]+[K_2]+\cdots + [K_r].$$
Using the refinement property of $S$ and the definition of the type semigroup, one obtains a decomposition $K=\sqcup _{i=1}^n K_i'$ such that $K_i'\in \mathbb K$ for each $i$, and $g_1,\dots , g_n\in G$ such that, for each $i$, $K_i'\subseteq \Omega_{g_i^{-1}}$ and $\theta _{g_i}(K_i') \subseteq K_j$ for some $j=1, \dots , r$. It follows that
$$K_i'\subseteq \theta_{g_i^{-1}}(K_j\cap \Omega_{g_i}) \subseteq   \theta_{g_i^{-1}}(U\cap \Omega_{g_i})\subseteq U,$$
where the last containment follows from the fact that $U$ is $G$-invariant. We deduce that $K\subseteq U$, and so $\varphi (\psi (U))\subseteq U$. The other containment $U\subseteq \varphi (\psi (U))$ follows from the fact that $\Omega$ has a basis of compact open subsets.

Finally, let $I\in \mathcal L (S)$. It is clear that $I\subseteq \psi (\varphi (I))$. To show the reverse inclusion, it is enough to check that, if $K\in \mathbb K$ and $K\subseteq \varphi (I)$, then $[K]\in I$. By compactness of $K$, there are $K_1,\dots , K_r\in \mathbb  K$ such that $[K_i]\in I$ and $K\subseteq \cup_{i=1}^r K_i$, and thus
$$[K]\le [K_1]+\dots +[K_r] \in I.$$
Since $I$ is an order ideal of $S$, we see that $[K]\in I$, as desired.
\end{proof}

We can now obtain a description of the lattice of induced ideals of tame graph algebras.

\begin{theorem}
 \label{thm:idealsOr}
 Let $(E,C)$ be a finite bipartite separated graph. Then there is a lattice isomorphism
$$\textup{Ind}(\Lab(E,C)) \cong \textup{Ind}(\mathcal{O}^{(r)}(E,C)) \cong \mathcal L (M(F_{\infty}, D^{\infty}))\cong \mathcal H (F_{\infty}, D^{\infty}).$$
Moreover for $H \in \mathcal H (F_{\infty}, D^{\infty})$, we have
$$\Lab(E,C)/I(H) \cong C_K(Z) \rtimes_{\theta\vert_Z^*} \mathbb{F} \andspace \mathcal O^{(r)}(E,C)/I(H) \cong C(Z) \rtimes _{(r), \theta|_{Z}^*}\mathbb F,$$
where $Z:= \Omega (E,C)\setminus U$ with $U:= \bigcup _{v\in H} \Omega(E,C)_v$.
\end{theorem}

\begin{proof}
 It follow from \cite[Theorem 7.4]{AE} that there is a natural isomorphism
 $$S:= S(\Omega (E,C), \mathbb F, \mathbb K) \cong M(F_{\infty}, D^{\infty}).$$
Combining this with Theorem \ref{thm:ideals-reducedcrossed} (or Remark~\ref{rem:AlgCrPrIndId}), \cite[Corollary 6.10]{AG} and Lemma \ref{lem:order-ideals}, we obtain
$$\textup{Ind}(\Lab(E,C)) \cong \textup{Ind}(\mathcal{O}^{(r)}(E,C)) \cong \mathbb O^{\mathbb F}(\Omega (E,C)) \cong \mathcal L (S) \cong \mathcal L (M(F_{\infty}, D^{\infty}))
\cong \mathcal H (F_{\infty}, D^{\infty}).$$
The last part follows from Theorem \ref{thm:ideals-reducedcrossed} and the definitions of the lattice isomorphisms.
\end{proof}

\begin{remark}
 \label{rem:algebraic-case}
We believe it is likely that Theorem \ref{thm:VidealsLab} generalizes to the setting of tame graph C*-algebras, at least for the reduced ones. This would mean that we have a lattice isomorphism
$$\text{Proj} (\mathcal O^r (E,C)) \cong \mathcal L (M(F_{\infty}, D^{\infty})) \cong \mathcal H  (F_{\infty}, D^{\infty}) , $$
where $\text{Proj} (\mathcal O^r (E,C))$ denotes the lattice of
ideals of $\mathcal O ^r(E,C)$ which are generated by their
projections. By Theorem \ref{thm:idealsOr}, this is equivalent to
saying that every ideal generated by projections is induced. In Section~\ref{sect:GeneralIdeals}, we will prove this for ideals $I \ideal \mathcal{O}^r(E,C)$ of \textit{finite type}.
\end{remark}

\section{The ideals associated to hereditary $C$-saturated subsets of $(E,C)$}
\label{sec:ideals-(E,C)}

In this section, we will analyze the induced ideals of $\Lab (E,C)$, $\mathcal O (E,C)$ and $\mathcal O ^r (E,C)$ arising
from hereditary $C$-saturated subsets of $E^0$, as opposed to the general study of ideals corresponding to
hereditary $D^{\infty}$-saturated subsets of $(F_{\infty}, D^{\infty})$. By Theorem~\ref{thm:NBallIsNGraph}, $(E,C)$ and $(E_n,C^n)$ give
rise to the same algebras for all $n\ge 0$, so we can apply the corresponding results to any hereditary $C^n$-saturated
subset of the separated graph $(E_n, C^n)$.

First, we shall give a concrete description of the hereditary and $D^\infty$-saturated closure of a subset $H \in \mathcal{H}(E,C)$ inside $(F_{\infty},D^\infty)$.

\begin{lemma}
 \label{lem:hersatfromE1}
 Let $(E,C)$ be a finite bipartite separated graph and take $H\in \mathcal H(E,C)$. Let
 $$H_1 := \{ s_1(X(x)) \mid x\in E^1 \, \, \,  {\rm and } \, \, \, s(x) \in H \} \cup (H \cap E_1^{0,0}).$$
Then $H_1 \in \mathcal H _{(E_1, C^1)}$ and $H\cup H_1 \in \mathcal
H_{(F_1, D^1)}$ is the hereditary closure of $H$ inside $(F_1,D^1)$.
\end{lemma}

\begin{proof}
$H_1$ is clearly hereditary; if $r_1(e) \in H \cap E_1^{0,0}=H \cap
E^{0,1}$, then $e \in X(x)$ for some $x \in E^1$ with $s(x)=r_1(e)$,
hence $s_1(e) \in s_1(X(x)) \subseteq H_1$. Next, we show
$C^1$-saturation. Suppose that $s(X(x))\subseteq H_1$ for some $x
\in E^1$, and write $w:=s(x)$. Also, set $x=x_i \in X_i^u$ with
$u:=r(x)$ and $C_u= \{ X_1^u, \dots , X_i^u,\dots , X_{k_u}^u \}$.
If $k_u= 1$, then necessarily $w\in H$ by the definition of $H_1$,
so suppose that $k_u>1$. If for some $j\ne i$, we have $s(x_j)\in H$
for all $x_j\in X_j^u$, then $u\in H$ by $C$-saturation, and so
$s(x) = s(x_i)\in H$ because $H$ is hereditary. Thus, we may assume
that, for all $j\ne i$, there exists $x_j\in X_j^u$ such that
 $s(x_j)\notin H$. Now, consider the vertex
 $$v:= v(x_1,\dots ,  x_{i-1}, x_i,x_{i+1}, \dots , x_{k_u})\in E_1^{0,1}.$$
Then $v\in H_1$ because $v= s( \alpha ^{x_i} (x_1, \dots ,  x_{i-1}, x_{i+1}, \dots , x_{k_u})) \in s(X(x_i)) =s(X(x))\subseteq H_1$. But by the definition of $H_1$,
 there must be $k\in \{ 1, \dots , k_{u}\}$ such that $s(x_k) \in H$.
 Hence $k=i$ and $w=s(x_i) \in H$, as desired. It is now clear that $H \cup H_1 \in \mathcal{H}(F_1,D^1)$,
 and $H \cup H_1$ is obviously nothing but the hereditary closure of $H$ inside $(F_1,D^1)$.
 \end{proof}

\begin{notation}
 \label{not:Hinfty}
Given $H\in \mathcal H (E,C) $, we define a sequence $H_n\in \mathcal H (E_n, C^n)$ in an inductive way, so that
 $$H_n := \{ s(X(x)) \mid x\in E_{n-1}^1 \, \, \,  {\rm and } \, \, \, s(x) \in H_{n-1} \} \cup (H_{n-1} \cap E_n^{0,0}).$$
 Then, by Lemma~\ref{lem:hersatfromE1}, $H_{\infty}: = \bigcup_{n=0}^{\infty} H_n \in \mathcal{H}(F_\infty,D^\infty)$ is the hereditary closure of $H$ inside $(F_\infty,D^\infty)$.
 \end{notation}

 We have thus showed the following lemma:

 \begin{lemma}
 \label{lem:hersatfromE}
 Let $(E,C)$ be a finite bipartite separated graph. Then there is an injective order-preserving map
 $$\mathcal H (E,C) \to \mathcal H (F_{\infty}, D^{\infty})$$
sending $H\in \mathcal H (E,C) $ to $H_{\infty}\in \mathcal H (F_{\infty}, D^{\infty}) $.
 Moreover, the ideal $I(H)$ of $\mathcal O^r (E,C)$ generated by $H$ is precisely the ideal $I(H_{\infty})$, and similar statements hold for $\mathcal O (E,C)$ and for $\Lab _K(E,C)$.
 \end{lemma}

We also mention the following easy description of the open, invariant subspace associated with $H \in \mathcal{H}(E,C)$ in terms of configuration spaces. Recall that if $\xi \in \Omega(E,C)_v$, then $1 \in \xi$ is regarded as the trivial path $v$, so that $r(1)=v$ in this situation.

\begin{lemma}\label{lem:OpenInvFromHS}
Let $(E,C)$ be a finite bipartite graph and let $H \in \mathcal{H}(E,C)$. Then
$$\bigcup_{v \in H_\infty} \Omega(E,C)_v =\{ \xi \in \Omega(E,C) \mid r(\alpha) \in H \text{ for some } \alpha \in \xi\}.$$
\end{lemma}

\begin{proof}
Since $I(H)=I(H_\infty)$, we have
$$U =\theta_\mathbb{F}\Big(\bigcup_{v \in H} \Omega(E,C)_v \Big)=\{ \xi \in \Omega(E,C) \mid r(\alpha) \in H \text{ for some } \alpha \in \xi\}.$$
\end{proof}

Given a finitely separated graph $(E,C)$ and a hereditary $C$-saturated subset $H$ of $E^0$, we denote by $E/H$ the subgraph of $E$ with
$(E/H)^0:= E^0\setminus H$ and $(E/H)^1:= \{ e\in E^1\mid s(e) \notin H \}$.
Similarly, for any subset $X\subseteq E^1$, define
$$X/H := X\cap s^{-1}(E^0\setminus H).$$
For $v\in (E/H)^0$,  we set
$$(C/H)_v := \{ X/H \mid X\in C_v  \},$$
which is a partition of $r_{E/H}^{-1}(v)$, and
$C/H := \bigsqcup_{v\in E^0\setminus H } (C/H)_v$. Observe
that $X/H\ne \emptyset $ for all $X\in C_v$ with $v\in E^0\setminus
H$, because $H$ is $C$-saturated.

 \begin{theorem}
 \label{thm:quotientalg}
 Let $(E,C)$ be a finite bipartite separated graph and let $H \in \mathcal{H}(E,C)$. Then there is a natural $*$-algebra isomorphism
$$\Lab _K (E,C)/I(H) \cong \Lab _K (E/H, C/H).$$
 Likewise, there are natural C*-algebra isomorphisms
 $$\mathcal O (E,C)/I(H) \cong \mathcal O (E/H, C/H) \andspace \mathcal O^r (E,C)/I(H) \cong \mathcal O^r (E/H, C/H).$$
 \end{theorem}

\begin{proof}
As observed in the proof of \cite[Corollary 3.12]{AG}, it is easy to show using
universal properties that the map sending $v+I(H) \mapsto v$ for $v\in (E/H)^0$ and $e+I(H)\mapsto e$ for $e\in (E/H)^1$
extends to a $*$-isomorphism
$$C^*(E,C)/I(H)\longrightarrow C^*(E/H, C/H).$$
Likewise, we obtain a $*$-isomorphism $ L _K(E,C)/I(H) \cong L_K(E/H, C/H)$ for any field with involution $K$. It is straightforward to check that
$$\big( E_1/H_1,C^1/H_1) = \big((E/H)_1,(C/H)^1 \big) .$$
Indeed we have that $(E_1/H_1)^{0,0}= (E/H)^{0,1}= ((E/H)_1)^{0,0}$ and that $(E_1/H_1)^{0,1}$ is the set of all the vertices
$v(x_1,x_2,\dots , x_k)$ such that $x_i\in X_i/H$ for all $i$, where $C_v= \{ X_1,X_2,\dots , X_k \}$ for some
$v\in E^{0,0}\setminus H = (E/H)^{0,0}$. This shows that $(E_1/H_1)^0= ((E/H)_1)^0$, and similarly
$(E_1/H_1)^1= ((E/H)_1)^1$ and $C^1/H_1= (C/H)^1$. We thus  obtain the following commutative diagram
\begin{center}
\begin{tikzpicture}[>=angle 90]
\matrix(a)[matrix of math nodes,
row sep=2em, column sep=2.5em,
text height=1.5ex, text depth=0.25ex]
{L(E,C)/I(H) & & L(E/H,C/H) \\
L(E_1,C^1)/I(H_1) & L(E_1/H_1,C^1/H_1) & L((E/H)_1,(C/H)^1) \\};
\path[->](a-1-1) edge node[above]{$\cong$} (a-1-3);
\path[->](a-2-1) edge node[above]{$\cong$} (a-2-2);
\path[->](a-2-2) edge node[above]{$\cong$} (a-2-3);
\path[->>](a-1-1) edge node[left]{} (a-2-1);
\path[->>](a-1-3) edge node[right]{} (a-2-3);
\end{tikzpicture},
\end{center}
where all the maps are the canonical ones. Applying this observation inductively gives
identifications $L(E_n,C^n)/I(H_n) \cong L((E/H)_n,(C/H)_n)$ commuting with the connecting homomorphisms, and hence we obtain an isomorphism
$$\Lab(E,C)/I(H) = L^{\text{ab}}(E,C)/I(H_\infty) \cong L^{\text{ab}}(E/H,C/H)$$
of the limits. The same proof applies to $\mathcal{O}$.

We now give a proof for $\mathcal O ^r$, which uses the dynamical interpretation of these algebras.
Let $\mathbb{F}'$ denote the free group on $(E/H)^1$, which we can regard as a subgroup of $\mathbb{F}$.
Let $U := \bigcup _{v\in H_{\infty}} \Omega (E,C)_v$  be the open invariant subset of $\Omega (E,C)$ associated to $H_{\infty}$, and set
$Z:= \Omega (E,C) \setminus U$. By Theorem \ref{thm:idealsOr}, we have
$$\mathcal{O}^r(E,C)/I(H) = \mathcal O ^r (E,C)/I(H_{\infty}) \cong C(Z) \rtimes  _{r,\theta|_{Z}^*} \mathbb F.$$
 Let us denote by $\theta '$ the restriction of $\theta $ to $Z$. If $x\in E^1$ and $s(x) \in H$, then the domain and codomain of $\theta'_x$ is empty, so we have that
 $$\mathcal O ^r (E,C)/I(H) \cong C(Z) \rtimes _{r,(\theta')^*} \mathbb F' ,$$
and we only have to show that the action $\theta '$ of $\mathbb F '$ on $Z$ is conjugate to $\theta^{(E/H,C/H)}$. Observe that by Lemma~\ref{lem:OpenInvFromHS},
$$Z=\{ \xi \in \Omega(E,C) \mid r(\alpha) \notin H \text{ for all } \alpha \in \xi\}.$$
We now claim that in fact $Z=\Omega(E/H,C/H)$, so let $\xi \in \Omega(E/H,C/H)$. Since every $\alpha \in \xi$ satisfies $r(\alpha) \in (E/H)^0=E^0 \setminus H$, we simply need to verify that the local configurations $\xi_\alpha$ are either of type (c1) or (c2) with respect to $(E,C)$. Assume first that $\xi$ is of type (c1) with respect to $(E/H,C/H)$, i.e. that $\xi_\alpha=s^{-1}_{E/H}(r(\alpha))$. Now since $H$ is hereditary, we must have
$$\xi_\alpha=s^{-1}_{E/H}(r(\alpha))=s^{-1}(r(\alpha)),$$
so $\xi_\alpha$ is indeed of type (c1) with respect to $(E,C)$. Next, assume that $\xi_\alpha$ is of type (c2) with respect to $(E/H,C/H)$, i.e. that $\xi_\alpha=\{e_{X/H}^{-1} \mid X/H \in (C/H)_{r(\alpha)}\}$ for some $e_{X/H} \in X/H$. From $H$ being $C$-saturated, we have $(C/H)_{r(\alpha)}=\{X/H \mid X \in C_{r(\alpha)}\}$, so setting $e_X:=e_{X/H}$ for all $X \in C_{r(\alpha)}$,
$$\xi_\alpha = \{e_{X/H}^{-1} \mid X/H \in (C/H)_{r(\alpha)}\} = \{e_X^{-1} \mid X \in C_{r(\alpha)}\}$$
is of type (c2) with respect to $(E,C)$. The converse inclusion is completely straightforward and does not use (explicitly) the assumptions about $H$ being
hereditary and $C$-saturated. Finally, since the dynamics is completely determined by the configurations, and the configuration spaces agree, we conclude that the two partial action are in
fact conjugate.

Observe that, by Theorem \ref{thm:idealsOr} and Remark \ref{rem:algebraic-case}(1), the same proof applies to
$\mathcal O$ and $\Lab$ respectively, so we obtain a second proof for those.
\end{proof}

\section{The ideals associated to hereditary $D^{\infty}$-saturated subsets of $(F_{\infty},D^{\infty})$}
\label{sec:InducedIdeals}

Recall from Section \ref{sec:structure-ideals} that every induced ideal of a tame graph algebra corresponds to a set $H \in \mathcal{H}(F_{\infty}, D^{\infty})$. In this section, we shall describe the induced ideal, or rather its quotient, in terms of the intersections $H \cap E_n^0$. We shall also consider a number of examples.

The proof of the following lemma is straightforward.

\begin{lemma}
 \label{lem:H_n}
 Let $H\in \mathcal H (F_{\infty}, D^{\infty})$, and, for each $n\ge 0$, set $H^{(n)} := H\cap E_n^0$.
 Then $H^{(n)}$ is a hereditary $C^n$-saturated subset of $E_n^0$,
 and $H= \cup_{n=0}^{\infty} H^{(n)}$.
\end{lemma}

We now describe the hereditary and $D^\infty$-saturated closure of $H^{(n)}$ inside $(F_\infty,D^\infty)$: Following Notation \ref{not:Hinfty}, we will denote by $H^{(n)}_{\infty}$ the hereditary closure of $H^{(n)}$ in $F_{\infty}^0$, that is,
$H^{(n)}_{\infty}= \cup _{m\ge n} H^{(n)}_m$, where
$$H^{(n)}_m= \{ s(X(x)) \mid x\in E_{m-1}^1 \, \, \,  {\rm and } \, \, \, s(x) \in H^{(n)}_{m-1} \} \cup (H^{(n)}_{m-1} \cap E_{m-1}^{0,1}).$$
Observe that $H^{(n)}_{\infty}$ is just the hereditary closure of $H^{(n)}$ in $F_{\infty}^0$, but we proved in Lemma \ref{lem:hersatfromE1} that it is
$(D^{\ge n})$-saturated, that is, if $v\in F^{0,m}$ with $m\ge n$ and $X\in C^m_v$ are such that $s(X)\subseteq H^{(n)}_{\infty}$, then $v\in H^{(n)}_{\infty}$. Now define 
$$H^n := (H \cap F_{n-1}^0) \cup H_\infty^{(n)}= (\bigcup_{i=0}^{n-1} H^{(i)}) \cup H_\infty^{(n)}$$ and observe that $H^n$ is exactly the hereditary and $D^\infty$-saturated closure of $H^{(n)}$ inside $(F_\infty,D^\infty)$.
Note also that $F_{\infty}\setminus H$ has the structure of a separated Bratteli
diagram, which we denote by $(F_{\infty}/H, D^{\infty}/H)$. Here
$$(F_{\infty}/H)^{0,n}=F^{0,n}\setminus H, \qquad D^{\infty}/H =
\bigsqcup_{n=0}^{\infty} C^n/H^{(n)} .$$

\begin{proposition}\label{prop:LimitOfOs}
Let $(E,C)$ denote a finite bipartite separated graph, let $H \in
\mathcal{H}(F_\infty,D^\infty)$ and apply the above notation. Also,
let $U_n$ denote the open invariant subspace of $\Omega(E,C)$
corresponding to $H^n$, and set $Z_n:=\Omega(E,C) \setminus U_n$ as
well as $Z := \bigcap_{n=0}^{\infty} Z_n$. Then
$$\mathcal{O}^r(E,C)/I(H) \cong C(Z) \rtimes_r \mathbb{F} \cong \varinjlim_n C(Z_n) \rtimes_r \mathbb{F} \cong \varinjlim_n \mathcal{O}^r(E_n/H^{(n)},C^n/H^{(n)}),$$
where the connecting homomorphisms are simply induced from restriction of functions. The same statement holds with $\mathcal O $ or $\Lab $ in place of $\mathcal O^r$.
\end{proposition}

\begin{proof}
Since $H^n$ is the hereditary and $D^\infty$-saturated closure of $H^{(n)}$, we have $I(H^n)=I(H^{(n)})$ as ideals of $\mathcal{O}^r(E,C)$. It follows that
$$\mathcal{O}^r(E_n/H^{(n)},C^n/H^{(n)}) \cong \mathcal{O}^r(E_n,C^n)/I(H^{(n)}) \cong \mathcal{O}^r(E,C)/I(H^n) \cong C(Z_n) \rtimes_r \mathbb{F}$$
from Theorem~\ref{thm:NBallIsNGraph} and Theorem~\ref{thm:quotientalg}. Now let $U$ denote the open invariant set corresponding to $H$. We have $H = \bigcup_{n=0}^{\infty} H^n$ by construction, so $U=\bigcup_{n=0}^\infty U_n$ and $Z=\Omega(E,C) \setminus U$. Thus, we obtain the identification
$$\mathcal{O}^r(E,C)/I(H) \cong C(Z) \rtimes_r \mathbb{F},$$
and clearly $C(Z) \rtimes_r \mathbb{F} \cong \varinjlim_n C(Z_n) \rtimes_r \mathbb{F}$. The same proof applies to $\mathcal{O}$ and $\Lab$.
\end{proof}

\begin{remark}
Recall from Theorem~\ref{thm:NBallIsNGraph} that every $v \in E_n^0$ for $n \ge 1$ corresponds to an $n$-ball $B(v) \in \mathcal{B}_n(\Omega(E,C))$. The open set $U_n$ of Proposition~\ref{prop:LimitOfOs} may therefore be described as the set of configurations $\xi \in \Omega(E,C)$ satisfying the following: There is some $v \in H^{(n)}$ and $\alpha \in \xi$ for which $\theta_\alpha(\xi)^n=B(v)$. In other words, we can describe the set $Z_n$ as
$$Z_n = \{ \xi \in \Omega(E,C) \mid \xi \not\equiv B(v) \text{ for all } v \in H^{(n)}\}.$$
Thus, the descending filtration $(Z_n)_n$ exactly removes the configurations with forbidden $n$-balls at the $n$'th step. In particular, we see that the restricted action $\theta\vert_Z \colon \mathbb{F} \act Z$ is of finite type in the sense of Definition~\ref{def:FiniteType} if and only if $H=H^n$ for some $n \ge 0$.

As another consequence, we see that \textit{every} convex shift $\theta \colon \mathbb{F} \act \Omega$ can be represented, up to Kakutani equivalence, by a separated Bratteli diagram: If $(E,C)$ is the graph representing the full convex shift (see Example~\ref{ex:FullConvexShift}), then $\theta$ is Kakutani equivalent to the restriction of $\theta^{(E,C)}$ to some closed invariant subspace $Z$. It follows that $\theta$ may 
be described as above by the separated Bratteli diagram $(F_\infty/H,D^\infty/H)$, where $H \in \mathcal{H}(F_\infty,D^\infty)$ corresponds to $U:=\Omega(E,C) \setminus Z$.
\end{remark}

We now describe $K_0$ of these quotient algebras in terms of their
associated separated Bratteli diagrams.

\begin{theorem}
\label{thm:K-theorytamequotients} Let $(E,C)$ be a finite bipartite
separated graph, let $H$ be a proper hereditary
$D^{\infty}$-saturated subset of $F_{\infty}$, and let
$(F_{\infty}/H, D^{\infty}/H)$ be its associated separated Bratteli
diagram.
\begin{enumerate}
\item[\textup{(1)}]
There is a natural group isomorphism
$$K_0(\mathcal O (E,C)/I(H))\cong K_0(\mathcal O ^r(E,C)/I(H)) \cong
G(F_{\infty}/H, D^{\infty}/H).$$
\item[\textup{(2)}] There is a natural isomorphism $\mathcal V (\Lab _K(E,C)/I(H))\cong
M(F_{\infty}/H,D^{\infty}/H)$, and so an isomorphism of pre-ordered
abelian groups
$$K_0(\Lab _K(E,C)/I(H)) \cong G(F_{\infty}/H, D^{\infty}/H)$$
for any field with involution $K$.
\end{enumerate}
\end{theorem}

\begin{proof} We adopt the above notation.
It is straightforward to check that, for $n\ge 1$ we have
$$ M(F_{\ge n}/H^{(n)}_{\infty}, D^{\ge n}/H^{(n)}_{\infty}) \cong
M(F_{\infty}/H^n, D^{\infty}/H^n)$$ and $$ G(F_{\ge
n}/H^{(n)}_{\infty}, D^{\ge n}/H^{(n)}_{\infty}) \cong
G(F_{\infty}/H^n, D^{\infty}/H^n) .$$ Therefore, it follows from
Theorem \ref{thm:K-theoryBrat} and an easy calculation that there
are commutative diagrams
$$\begin{CD} K_0(\mathcal O (E_n/H^{(n)},
C^n/H^{(n)})) @>\cong >>
G(F_{\infty}/H^n, D^{\infty}/H^n) \\
@VVV  @VVV \\
K_0(\mathcal O (E_{n+1}/H^{(n+1)}, C^{n+1}/H^{(n+1)})) @>\cong >>
G(F_{\infty}/H^{n+1}, D^{\infty}/H^{n+1})
\end{CD}$$
for all $n\ge 1$. Using Proposition \ref{prop:LimitOfOs}, we obtain
\begin{align*}
K_0( \mathcal O (E,C)/I(H)) & \cong \varinjlim_n K_0 (\mathcal O
(E_n/H^{(n)}, C^n/H^{(n)})) \\
& \cong \varinjlim_n G(F_{\infty}/H^n, D^{\infty}/H^n)\\
& = G(F_{\infty}/H, D^{\infty}/H).
\end{align*}
 This gives (1) for $\mathcal O$.
The same arguments give, using Theorem \ref{thm:K-theoryBrat}, the
result for $\mathcal O ^r$ and for $\Lab$.
\end{proof}

We record some useful properties of the monoid $M(F_{\infty}/H,D^{\infty}/H)$.

\begin{proposition}
 \label{prop:refinement}
Let $(E,C)$ be a finite bipartite
separated graph, let $H$ be a proper hereditary
$D^{\infty}$-saturated subset of $F_{\infty}$, and let
$(F_{\infty}/H, D^{\infty}/H)$ be its associated separated Bratteli
diagram. Then we have a natural monoid isomorphism
$$M(F_{\infty}/H, D^{\infty}/H)\cong M(F_{\infty},D^{\infty})/M(H),$$
where $M(H)$ is the order-ideal of $M(F_{\infty},D^{\infty})$ generated by $H$.
In particular, $M(F_{\infty}/H, D^{\infty}/H)$ is a refinement monoid.
  \end{proposition}

 \begin{proof}
  The isomorphism follows from \cite[Construction 6.8]{AG2}.
  Since $M(F_{\infty}, D^{\infty})$ is a refinement monoid, the quotient monoid
  $M(F_{\infty}, D^{\infty})/M(H)$ is also a refinement monoid, by \cite[Lemma 4.3]{AGOP}.
   \end{proof}

We now consider a number of examples.

\begin{example}
\label{exam:m,ndyn-system}
For integers $1\le m\le n$, define the separated graph
$(E(m,n),C(m,n))$, where
\begin{enumerate}
\item $E(m,n)^0 := \{v,w\}$ (with $v\ne w$).
\item $E(m,n)^1 :=\{\alpha_1,\dots , \alpha_n,\beta_1,\dots ,\beta_m\}$ (with $n+m$ distinct edges).
\item $s(\alpha_i)=s(\beta _j) =v$ and $r(\alpha _i)=r(\beta _j)=w$
for all $i$, $j$.
\item $C(m,n)= C(m,n)_v := \{X,Y\}$, where $X:= \{\alpha_1,\dots ,\alpha_n\}$ and $Y:=  \{\beta _1,\dots, \beta _m \}$.
\end{enumerate}
See Figure \ref{fig:m,nsepargraph} just below for a picture in the case $m=2$,
$n=3$. We refer the reader to \cite{AEK} and \cite[Example 9.3]{AE} for more information on this example.

\begin{figure}[H]
\begin{tikzpicture}[scale=0.8]
 \SetUpEdge[lw         = 1.5pt,
            labelcolor = white]
  \tikzset{VertexStyle/.style = {draw, shape = circle,fill = white, minimum size=15pt, inner sep=2pt,outer sep=1pt}}

  \Vertex[x=0,y=0]{u}
  \Vertex[x=0,y=3]{v}

  \tikzset{EdgeStyle/.style = {->,bend left=20,color={\nicered}}}  
  \Edge[](u)(v)
  \tikzset{EdgeStyle/.style = {->,bend left=55,color={\nicered}}}  
  \Edge[](u)(v)
  \tikzset{EdgeStyle/.style = {->,bend left=90,color={\nicered}}}  
  \Edge[](u)(v)    
  \tikzset{EdgeStyle/.style = {->,bend right=20,color={\niceblue}}}  
  \Edge[](u)(v) 
  \tikzset{EdgeStyle/.style = {->,bend right=50,color={\niceblue}}}  
  \Edge[](u)(v) 
  
\end{tikzpicture}
\caption{The separated graph $(E(2,3),C(2,3))$}
\label{fig:m,nsepargraph}
\end{figure}
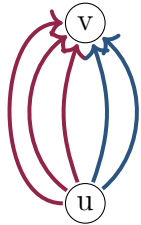

The C*-algebras $\mathcal O _{m,n}$ and $\mathcal O _{m,n}^r$
studied in \cite{AEK} are precisely the C*-algebras $\mathcal O
(E(m,n), C(m,n))$ and $\mathcal O ^r( E(m,n), C(m,n))$, respectively, in the
notation of the present paper. It was shown in \cite{AEK} that these two
C*-algebras are not isomorphic, and we will now show that the C*-algebra $\mathcal O^r_{m,n}$ is not simple if $m\ge 2$.
Let $(E,C) := (E(m,n), C(m,n))$. It is clear that $\mathcal H (E,C)$ only contains the trivial sets $\emptyset, E^0$. 

Adopting the notation of \cite{AEK}, we write
$$\Omega^u := \Omega (E, C) = X^u\sqcup Y^u,$$
where $X^u = \sqcup _{i=1}^n H^u_i = \sqcup _{j=1}^m V^u_j$, and $Y^u$ is homeomorphic to each of the sets $H^u_i$, $V_j^u$ by
homeomorphisms $h^u_i\colon Y^u\to H^u_i$ and $v^h_j\colon Y^u\to V^u_j$.
The maps $h^u_i,v^u_j$ are the universal maps defining the universal $(m,n)$-dynamical system. Indeed, we have
$h^u_i =\theta _{\alpha_i}$ and $v^u_j= \theta_{\beta_j}$ for all $i,j$.

Let us now describe the separated graph $(E_1,C^1)$. We have $E_1^{0,0}= \{ w \}$ and 
$$E_1^{0,1}= \{w_{ij} : 1\le i\le n , 1\le j \le m \}.$$
Now there are $n+m$ elements in $C^1_w$, namely $X_i:= X(\alpha_i)$ and $Y_j : = X(\beta_j)$ for $i=1,\dots , n$ and $j=1,\dots , m$.
Note that $|X_i|= m$ and $|Y_j| = n$. Moreover, $s(X_i) = \{ w_{ij} : j=1,\dots , m \}$ and $s(Y_j) = \{ w_{ij} : i= 1,\dots , n \}$.

If $m=n\ge 2$, then set $H:= \{ w_{ij} : i\ne j \}$ . Then $H$ is a maximal hereditary and $C^1$-saturated subset of $E_1^0$. Moreover, $(E_1/H, C^1/H)$ consists of $n$ cycles based at the vertex $w$, so that
$$\mathcal O^r_{n,n}/I(H) \cong \mathcal O ^r(E_1/H, C^1/H) \cong M_{n+1} ( C^*_{{\rm red}} (\mathbb F_n)), $$
which is a simple C*-algebra. However $\mathcal O_{n,n}/I(H) \cong C^*( \mathbb F_n)$ is not simple. (Incidentally, note that this gives another proof that
$\mathcal O^r_{n,n}\ne \mathcal O _{n,n}$ for $n\ge 2$.)

If $2\le m <n$, then set
$$H:= \{ w_{ij} : i\ne j \text{ and } 1\le i\le n, 1\le j \le m-1 \}\cup \{w_{im} : 1\le i\le m-1 \}.$$
Then $H$ is again a maximal hereditary $C^1$-saturated subset of $E_1^0$.
However, in this case the quotient C*-algebra $\mathcal O ^r_{m,n}/I(H)$ is not even $\mathcal V$-simple, as we will show in Section \ref{sec:Vsimplicity}.
It is therefore clear that the universal $(E_1/H,C^1/H)$-system is not equivalent to the $(m,n)$-system  $(X,Y)$ considered in the proof of
\cite[Proposition 3.9]{AEK}. Indeed, observe that $v_i^{-1}\circ h_i= {\text id}_Y$ for $i=1,\dots , m-1$ in that example, and this is not necessarily true in a $(E_1/H,C^1/H)$-system.
The $(m,n)$-dynamical system just mentioned shows that the algebra $M_{m+1}(\mathcal O _{n-m+1})\cong M_{n+1} (\mathcal O _{n-m+1})$ is a simple quotient of
$\mathcal O _{m,n}$, where $\mathcal O_k$ denotes the usual Cuntz algebra. Indeed, we can define a surjective $*$-homomorphism $\mathcal{O}_{m,n} \to M_{m+1}(\mathcal{O}_{n-m+1})$ by
$$
\alpha_i \mapsto \left\{\begin{array}{cl}
1 \otimes e_{i+1,1} & \If i =1,\ldots,m-1 \\
s_{i-m+1} \otimes e_{m+1,1} & \If i=m,\ldots,n
\end{array} \right. \quad , \quad \beta_j \mapsto 1 \otimes e_{j+1,1} \text{ for } j=1,\ldots,m,
$$
$w \mapsto 1 \otimes e_{1,1}$ and $v \mapsto \sum_{j=2}^{m+1} 1
\otimes e_{j,j}$. However, it is not clear to the authors whether
the same algebra $M_{m+1}(\mathcal O _{n-m+1})$ appears as a simple
quotient of the {\it reduced} tame C*-algebra $\mathcal O^r _{m,n}$.
\end{example}

We now present an example relating our theory with classical symbolic dynamics; hopefully, it can also serve as an exemplification of the general theory of the previous sections. We only consider the case where the alphabet is $\{ 0,1 \}$, but similar statements can be made for an arbitrary finite alphabet, considering a corresponding variation of the separated graph considered below.

\begin{example}
 \label{ex:lamplighter}
  Let $(E,C)$ be the separated graph
described in Figure \ref{fig:lampgroup}, with $C_v=\{ X, Y\}$ and
$X=\{ \alpha _0,\alpha_1\}$ and $Y=\{ \beta _0,\beta _1 \}$.
\begin{center}{
\begin{figure}[H]
\begin{tikzpicture}[scale=0.8]
 \SetUpEdge[lw         = 1.5pt,
            labelcolor = white]
\tikzset{VertexStyle/.style = {draw, shape = circle,fill = white, minimum size=15pt, inner sep=2pt,outer sep=1pt}}

  \Vertex[x=0,y=0]{0}
  \Vertex[x=6,y=0]{1}
  \Vertex[x=3,y=3]{v}  

  \tikzset{EdgeStyle/.style = {->,bend left=35,color={\nicered}}}  
  \Edge[label=$\beta_0$, style={circle,inner sep=0pt }](0)(v)
  \tikzset{EdgeStyle/.style = {->,bend right=35,color={\nicered}}}  
  \Edge[label=$\beta_1$, style={circle,inner sep=0pt }](1)(v)  
  \tikzset{EdgeStyle/.style = {->,bend right=35,color={\niceblue}}}  
  \Edge[label=$\alpha_0$, style={circle,inner sep=0pt }](0)(v) 
  \tikzset{EdgeStyle/.style = {->,bend left=35,color={\niceblue}}}  
  \Edge[label=$\alpha_1$, style={circle,inner sep=0pt }](1)(v) 
\end{tikzpicture}
\caption{The separated graph underlying the lamplighter group}
\label{fig:lampgroup}
\end{figure}
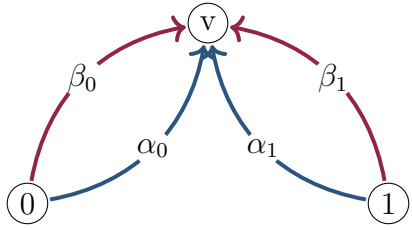}
\end{center}
\vspace{-0.7cm}
This example has been considered in \cite[Example 9.7]{AE}, where it is observed that
$$v\mathcal O (F, D)v\cong C(\mathcal X) \rtimes_{\sigma} \Z  , \qquad v\Lab _K(F,D)v \cong C_K(\mathcal X)\rtimes _{\sigma} \Z , $$
where $\mathcal X =\{ 0, 1 \}^{\Z}$ and $\sigma $ is the usual shift
homeomorphism on $\mathcal X$. Note that $\mathcal O (E,C) =
\mathcal O ^r(E,C)$ in this case. Indeed it can be easily seen that
$(E,C)$-dynamical systems are in one-to-one correspondence with
usual dynamical systems $(\mathcal Y, \varphi)$, where $\varphi$ is
a homeomorphism of the compact Hausdorff space $\mathcal Y$, with
the additional information of a partition $\mathcal Y=\mathcal
Y_0\sqcup \mathcal Y_1$ of $\mathcal Y$ into clopen subsets
$\mathcal Y_0,\mathcal Y_1$. To obtain the corresponding
$(E,C)$-system, take $\Omega := \mathcal Y'\sqcup \mathcal Y$, where
$\mathcal Y'=\mathcal Y_0'\sqcup \mathcal Y_1'$ is a disjoint copy
of $\mathcal Y$, and where $\Omega _v:=\mathcal Y$, $\mathcal Y_i'$
correspond to the vertices labeled by $i$, the maps $\alpha_i$
correspond to the identification of elements of $\mathcal Y'_i$ with
elements of $\mathcal Y_i$, and the maps $\beta_i$ are induced by
the homeomorphism $\varphi$. It is easily checked that $v(C(\Omega)
\rtimes \mathbb F) v\cong C(\mathcal Y)\rtimes_{\varphi} \Z$ in this
situation. The unique equivariant continuous map $\Omega \to \Omega
(E,C)$ predicted by \cite[Corollary 6.11]{AE}, restricted to
$\mathcal Y$, is the fundamental map in symbolic dynamics (see e.g.
\cite[\S 6.5]{LM}) sending each element $x$ in $\mathcal Y$ to the
sequence $(a_n)_{n\in \Z}$ recording to which of the sets $\mathcal
Y_0$ or $\mathcal Y_1$ belongs $\varphi ^n(x)$, that is, $a_n= i
\iff \varphi^n(x)\in \mathcal Y_i$, for $n\in \Z$.

 We here give a dynamical interpretation in terms of the associated canonical sequence $\{(E_n, C^n)\}$ of bipartite separated graphs. Let $\mathcal X=\{ 0, 1 \}^{\Z}$ as above.  For a finite word $a_1a_2\cdots a_n\in \{ 0,1 \}^n$, we will write
 $$ [a_1a_2 \cdots a_{i-1}\ul{a_i} a_{i+1}\cdots a_n] := \{ x\in \mathcal X : x_{j-i}= a_j \text{ for } j=1,2,\dots , n \}.$$

 The space $\Omega (E,C)$ is of the form $\Omega (E,C)= \mathcal X'\sqcup \mathcal X$, where $\mathcal X'\cong \mathcal X= \{0,1\}^{\Z}$. By a slight abuse of language,
 we will identify $\mathcal X'$ with $\mathcal X$ notationally, so that we will obtain
 two partitions of the space $\{ 0,1 \}^{\Z}$ into clopen subsets for each separated graph $(E_n,C^n)$, corresponding to the sets of vertices $E_n^{0,0}$ and $E_n^{0,1}$ respectively, and maps
 between these clopen subsets corresponding to the edges $E_n^1$.

 The first layer of the sequence $\{ (E_n, C^n) \}_{n\ge 0 }$ corresponds to a trivial decomposition $\mathcal X^0= \mathcal X$ and to the decomposition
 $$\mathcal X^1_0= [\underline{0}], \qquad  \mathcal X^1_1= [\ul{1}].$$
 The maps corresponding to the edges are the maps
 $\alpha _i \colon \mathcal X^1_i \to \mathcal X$ and $\beta _i \colon \mathcal X^1_i  \to \mathcal X$ defined by
 $\alpha _i = \text{id}|_{\mathcal X_i^1}$ and $\beta _i= \sigma|_{\mathcal X_i^1}$.
 We describe now the clopen sets corresponding to the separated graph $(E_n, C^n)$ for any $n\ge 1$. Let $w = a_1a_2\cdots a_n\in \{0,1 \}^n$ be a word of length $n$.
 If $n=2m$ is even, define
 $$\mathcal X^n_w := [a_1\cdots a_m\ul{a_{m+1}}a_{m+2}\cdots a_{2m}].$$
 If $n=2m+1$ is odd, define
 $$\mathcal X^n_w := [a_1\cdots a_{m}\ul{a_{m+1}}a_{m+2}\cdots a_{2m+1}] .$$
 The set $F_n^0= E_n^{0,0}$ has exactly $2^n$ elements, and thus $F_{n+1}^0= E_n^{0,1}$ has exactly $2^{n+1}$ elements. The vertices in $E_n^{0,0}$ correspond to a decomposition
 $$\mathcal X= \bigsqcup _{w\in \{ 0,1 \}^n } \mathcal X^n_w. $$

 Let $n= 2m$ and take $w\in \{ 0,1 \}^n$. Then $C^n_w= \{ X^w_1 , X^w_2 \}$, where $X^w_1= \{ \alpha^{n+1}_{w0}, \alpha^{n+1}_{w1} \}$ and $X^w_2 = \{ \beta^{n+1}_{0w} , \beta^{n+1}_{1w} \}$.
 The edges $\alpha^{n+1}_{wi}$ correspond to the maps, denoted in the same way,  $\alpha^{n+1}_{wi} \colon \mathcal X^{n+1}_{wi} \to \mathcal X^n_w$, where each is simply the identity on the respective domain.
 Similarly, the edges $\beta^{n+1}_{iw}$ correspond to the maps  $\beta^{n+1}_{iw} \colon \mathcal X^{n+1}_{iw} \to \mathcal X^n_w$ given by the restriction of $\sigma$ to the respective domains.

 Now let $n= 2m+1$ and take $w\in \{ 0,1 \}^n$. In this case, we have  $C^n_w= \{ X^w_1 , X^w_2 \}$,
 where $X^w_1= \{ \alpha^{n+1}_{0w}, \alpha^{n+1}_{1w} \}$ and $X^w_2 = \{ \beta^{n+1}_{w0} , \beta^{n+1}_{w1} \}$.
The edges $\alpha^{n+1}_{iw}$ again correspond to the maps, denoted in the
same way,  $\alpha^{n+1}_{iw} \colon \mathcal X^{n+1}_{iw} \to
\mathcal X^n_w$, acting as the identity on the respective domains, while the edges $\beta^{n+1}_{wi}$ correspond to the maps  $\beta^{n+1}_{wi} \colon \mathcal X^{n+1}_{wi} \to \mathcal X^n_w$ given by the restriction of $\sigma^{-1}$ to the respective domains.

Now it is quite easy to observe that $\Omega (E,C)_v \cong \mathcal
X$ canonically. Indeed, by \cite[p. 3008]{AE2}, we have $\Omega
(E,C)_v \cong \varprojlim (F_{\infty}^{0,2j}, r_{2j})$, and in this
case, the maps $r_{2j}\colon F_{\infty}^{0,2j} \to
F_{\infty}^{0,2j-2}$ are  the maps sending $awb$ to $w$, where $w\in
\{0,1 \}^{2j-2}$ and $a,b\in \{ 0,1 \}$.

 \end{example}

We can now relate the ideals of $\mathcal O (E,C)$ with some sets of
words, and the quotients of $\mathcal O (E,C)$ with the subshifts of
the shift $(\mathcal X,\sigma )$. Recall from \cite[Chapter 1]{LM}
that a {\it subshift} of $\mathcal X= \{ 0,1 \}^{\Z}$ is a subspace
$\mathcal X_{\mathcal F}$ which can be described as the set of all
elements $x$ in $\mathcal X$ not containing any block from a fixed
family $\mathcal F$ of finite words in the alphabet $\{ 0,1\}$. (A
block of $x$ is a finite subsequence of consecutive terms in $x$.)
The family $\mathcal F$ is called the family of {\it forbidden
words} of the subshift. By \cite[Theorem 6.1.21]{LM}, the subshifts
of $\mathcal X$ are exactly the invariant closed subsets of
$\mathcal X$. A subshift $Z$ is said to be {\it of finite type} if
there exists a finite set $\mathcal F$ such that $Z=\mathcal
X_{\mathcal F}$.

\begin{proposition}
 \label{prop:subshifts}
 Let $(E,C)$ be the separated graph described in Example \ref{ex:lamplighter}, and adopt the notation used there.
 \begin{enumerate}
  \item[\textup{(1)}] For each subset $\mathcal F$ of words, there is a unique hereditary $D^{\infty}$-saturated subset $H_{\mathcal F}$ of the separated graph $(F_{\infty}, D^{\infty})$ such that
  $\mathcal F$ and $H_{\mathcal F}$ generate the same subshift, that is $\mathcal X_{\mathcal F} = \mathcal X_{H_{\mathcal F}}$.
  \item[\textup{(2)}] $\mathcal X_{\mathcal F}$ is a shift of finite type if and only if there is some $n <\infty$ such that $H_{\mathcal F}$ is generated by $H:=H_{\mathcal F}\cap E_n^0$. In this case, $\theta^{(E_n/H,C^n/H)}$ and $\sigma$ are Kakutani equivalent, so in particular there are Morita-equivalences
$$C_K(X_\mathcal{F}) \rtimes_\sigma \mathbb{Z} \sim L_K^{\textup{ab}}(E_n/H,C^n/H) \andspace C(\mathcal X_{\mathcal F}) \times_{\sigma} \Z  \sim \mathcal O (E_n/H, C^n/H).$$
 \end{enumerate}
\end{proposition}

\begin{proof}
(1):  By Example \ref{ex:lamplighter}, the set of vertices of the graph $F_{\infty}$
 can be identified with the set of all finite words in the alphabet $\{ 0,1 \}$.

 Given a set $\mathcal F$ of words, the space $\mathcal X_{\mathcal F}$ is a closed invariant subset of $\mathcal X$. One can easily show that
 $$ \mathcal X\setminus \mathcal X_{\mathcal F} = \bigcup_{n=0}^{\infty} \bigcup _{w\in W_n} \bigcup_{j\in \Z} \sigma ^j (\mathcal X_w^n) ,$$
where for each $n\ge 0$, the set $W_n$ is the set of words of length
$n$ containing a block coming from $\mathcal F$, and $\mathcal
X_w^n$ are the subsets of $\mathcal X$ defined in Example
\ref{ex:lamplighter}. Observe that the set $W:=
\bigcup_{n=0}^{\infty} W_n$ is precisely the hereditary closure of
$\mathcal F$ in the graph $F_{\infty}$. The set $H_{\mathcal F}$ is
the $D^{\infty}$-saturation of $W$, and it generates the same open
invariant subset as $\mathcal F$ and as $W$. We therefore have
$\mathcal X_{\mathcal F} = \mathcal X_{H_{\mathcal F}}$. The
uniqueness of $H_{\mathcal F}$ comes from Theorem
\ref{thm:idealsOr}.

(2): The first statement follows immediately from (1). Using the above identification of $\theta^{(E_n,C^n)}$, we see that $\bigsqcup_{v \in E^{0,0} \cap H} \Omega(E_n,C^n)_v=X_\mathcal{F}$, so that the restriction of $\theta^{(E_n/H,C^n/H)}$ to the full, clopen subspace 
$$\bigsqcup_{v \in E_n^{0,0} \setminus H} \Omega(E_n/H,C^n/H)_v$$
is directly dynamically equivalent to $\sigma$ via the group homomorphism $\mathbb{F}((E_n/H)^1) \to \mathbb{Z}$ given by 
\begin{itemize}
\item $\alpha_{wi}^{n+1} \mapsto 0$ and $\beta_{iw}^{n+1} \mapsto 1$ if $n$ is even,
\item $\alpha_{iw}^{n+1} \mapsto 0$ and $\beta_{wi}^{n+1} \mapsto -1$ if $n$ is odd.
\end{itemize}
In particular, $\theta^{(E_n/H,C^n/H)}$ and $\sigma$ are Kakutani equivalent, so we obtain Morita equivalences as above. This concludes the proof.
\end{proof}

 \begin{example}
  \label{exam:cycles}
  The cycles in the different graphs $(E_n, C^n)$ are determined by the complements of some hereditary $C^n$-saturated subsets of $E_n^0$.
  They correspond to periodic orbits in the shift $(\mathcal X,\sigma )$.
  For instance, consider the word $w= 0110$, which is primitive, that is, it is not a square. The successive rotations of $w$ are the words
  $$w_1=w, \quad w_2= 1100, \quad w_3 = 1001, \quad w_4= 0011 .$$
  Consider the separated graph $(E_3,C^3)$, and the $C^3$-saturation $H$ of the set
  $$E^{0,1}_3\setminus \{w_1,w_2,w_3,w_4\}.$$
  It is easily seen that $E_3^{0,1} \setminus H = \{w_1,w_2,w_3,w_4 \}$, and $E_3^{0,0}\setminus H = \{v_1,v_2,v_3,v_4 \}$, where
  $$v_1= 110, \quad v_2 = 100 , \quad v_3= 001, \quad v_4= 011 .$$
  The separated graph $(E_3/H, C^3/H)$ is essentially given by a cycle of length $4$. We have
  $$\mathcal O  (E,C) /I(H_{\infty}) \cong \mathcal O  (E_3/H,C^3/H) \cong M_8 ( C(\mathbb T)) .$$

  Note that if we start with a word which is not primitive, for instance $w= 0101$, then there is a non-trivial $C^2$-saturation $H_2$ induced by $H$ in $(E_2,C^2)$,  and there is a non-trivial $C^1$-saturation $H_1$ induced by $H_2$ in $(E_1, C^1)$, which is such that $E_1^{0,1} \setminus H_1 = \{ 10, 01 \}$ (that is, $H_1= \{00, 11 \}$),
  and everything is reduced to the $2$-cycle generated by $01$.
  \end{example}

  The periodic orbits give trivial examples of minimal subshifts. More interesting examples are provided by the minimal Cantor subshifts, whose underlying
  space is a Cantor set. 
  By \cite{sugisaki}, every strong orbit equivalence class of minimal Cantor systems contains a minimal Cantor subshift. By combining this with \cite[Theorem 2.1]{GPS}, 
  we see that, for every 
  $*$-isomorphism class of $C^*$-algebra crossed products $C(\Omega) \rtimes \Z$ of minimal actions on the Cantor set $\Omega$, there is a representative coming from a minimal subshift. 
  These $C^*$-algebras are classified by their ordered Grothendieck groups with distinguished order-units, which may be any simple dimension group with order-unit
  (see \cite[Theorems 1.12 and 1.14]{GPS}). See also \cite{DM} and \cite{hoynes} for further information about 
  the conjugacy classes of minimal subshifts.
  We remark that these examples imply that, in spite of the results in \cite{Lolk1}, the theory of
  separated graph C*-algebras leads to non-trivial examples (i.e.
  not coming from ordinary directed graphs)
  of simple nuclear C*-algebras with stable rank one and real rank zero,
  since it is well-known that the C*-algebras associated to minimal
  Cantor systems enjoy these properties (see \cite{Put90} and \cite[p. 184]{Elliott93}). By the results of the second-named author (\cite{Lolk1})
  this is not possible for the tame C*-algebra of a separated graph, but we see now
  that, factoring out a suitable maximal ideal generated by
  projections of the algebra $\mathcal O (E,C)$ appearing in Example
  \ref{prop:subshifts}, we may
  obtain such examples.

  The following example appears for instance in \cite{DE}.

 \begin{example}[The even shift]
  \label{exam:ruy}
  Consider the subshift $\mathcal Y$ of $(\mathcal X, \sigma)$ defined by taking as a set of forbidden words
  $\mathcal F = \{01^{2n+1}0 \mid n \ge 0\}$. That is, in a word of $\mathcal Y$ there is always an even number of consecutive 1's between two 0's.
  In this case, we have $H^{(2)}= \{ 010 \}$, $H^{(4)} = \{ 01110 \} \cup H_4'$, where $H_4'$ is the family of words of length
  $4$ or $5$ containing as subwords the word $010$. In general
  $$H^{(2i)}= \{ 0 1^{2i-1}0 \} \cup H_{2i}',$$
  where $H_{2i}'$ is the set of words of length $2i$ or $2i+1$ containing some subword of the form $01^{2j-1}0$ with $j<i$.
 This gives rise to a subshift which is not finite.
 \end{example}

\section{A complete description of the ideals of finite type}\label{sect:GeneralIdeals}

In this section, we completely determine the structure of the \textit{finite type} ideals of $\mathcal{O}^r(E,C)$:
\begin{definition}
Let $(E,C)$ denote a finite bipartite separated graph and let $(F_\infty,D^\infty)$ be the separated Bratteli diagram of $(E,C)$. Given an arbitrary ideal $J \ideal \mathcal{O}^r(E,C)$, there is some $H_J \in \mathcal{H}(F_\infty,D^\infty)$ for which $I(H_J)=(J \cap C(\Omega(E,C))) \rtimes_r \mathbb{F}$. We will say that any $H \in \mathcal{H}(F_\infty,D^\infty)$ is of finite type if $H=H^n$ for some $n \ge 0$, and an ideal $J$ is of finite type if $H_J$ is so. Finally, the lattices of finite type vertex sets and ideals will be denoted by $\mathcal{H}_\textup{fin}(F_\infty,D^\infty)$ and $\mathcal{I}_{\textup{fin}}(\mathcal{O}^r(E,C))$, respectively.
\end{definition}
Given any partial action $\theta \colon G \act \Omega$ and a point $x \in \Omega$, the stabiliser of $x$ is the subgroup
$$\text{Stab}(x):=\{g \in G \mid x \in \Omega_{g^{-1}} \text{ and } \theta_g(x)=x\}.$$
Recall that $\theta$ is \textit{called topologically free} if, given any open subspace $U$ and any $1\ne g \in G$, there exists $x \in U$ with $g \notin \text{Stab}(x)$. If $\theta$ is topologically free, then $C_0(\Omega) \rtimes_r G$ has the \textit{intersection property} by \cite[Theorem 29.5]{Exel}, i.e.~any non-zero 
ideal $J \ideal C_0(\Omega) \rtimes_r G$ has non-trivial intersection $C_0(U)=J \cap C_0(\Omega)\ne \{ 0\}$. It follows that $J$ contains the non-zero induced ideal $I:=C_0(U) \rtimes_r G$, but it does not tell us anything about the quotient $J/I$. The problem arises when 
the restriction of $\theta$ to $Z:=\Omega \setminus U$ is not topologically free, and a partial action is said to be \textit{essentially free} if all such restrictions are 
topologically free. However, while topological freeness of $\theta^{(E,C)}$ is quite frequent (we recall the characterisation given in \cite{AE} just below), essential freeness 
is extremely rare as shown in \cite{Lolk1}. In this section, we will introduce a weakening of topological freeness that still allows one to obtain information about the ideals, 
and show that it is always enjoyed by $\theta^{(E,C)}$. In particular, the restriction $\theta^{(E,C)}\vert_Z$ to any closed invariant subspace $Z$ of finite type will also have this property. 
From these observations, we can completely characterize the structure of finite type ideals of $\mathcal{O}^r(E,C)$. It is also worth noting that our methods, when applied to non-separated graphs, yield a complete description of the ideals of the graph $C^*$-algebra. 

We first recall \textit{Condition} (\textit{L}) of \cite{AE}, using a slightly different terminology: 

\begin{definition}
Consider any finite bipartite separated graph $(E,C)$. A vertex $v \in E^0$ is said to \textit{admit a choice} if there exists an admissible path $\beta=e\alpha$ with $s(\beta)=v$, and an element $X_e \ne X \in C_{r(e)}$ with $\vert X \vert \ge 2$. 
The graph is said to satisfy Condition (L) if for every 
simple cycle $\sigma$, the base vertex $s(\sigma)$ admits a choice.
\end{definition}

\begin{theorem}[{\cite[Theorem 10.5]{AE}}]
Let $(E,C)$ denote a finite, bipartite separated graph. Then $\theta^{(E,C)}$ is topologically free if and only if $(E,C)$ satisfies Condition \textup{(}L\textup{)}.
\end{theorem}

We now introduce the appropriate weakening of topological freeness.

\begin{definition}
Let $\theta \colon G \act \Omega$ be any partial action of a discrete group on a locally compact Hausdorff space. For any $x \in \Omega$, we shall say that $\theta$ is
\begin{itemize}
\item  \textit{topologically free in} $x$ if, given any $1 \ne g \in \text{Stab}(x)$ and an open neighbourhood $U$ of $x$, there exists $y \in U$ with $g \notin \text{Stab}(y)$.
\item \textit{strongly topologically free in} $x$ if, given any $1 \ne g_1,\ldots,g_n \in \text{Stab}(x)$ and any open neighbourhood $U$ of $x$, there exists $y \in U$ with $g_1,\ldots,g_n \notin \text{Stab}(y)$.
\end{itemize}
We will denote by $\Omega^{\textup{TF}}$ the set of points $x \in \Omega$ in which $\theta$ is topologically free. If $\theta$ is strongly topologically free in every $x \in \Omega^{\textup{TF}}$, it is said to be \textit{relatively strongly topologically free}.
\end{definition}

Observe that if $x$ is 
an interior point of $\Omega^{\textup{TF}}$, then $\theta$ is automatically strongly topologically free in $x$. In particular, topologically free partial actions are strongly topologically free in all points.

\begin{remark}\label{rem:IsolatedPoints}
We will now expand a bit on the situation for $\theta^{(E,C)}$, and we first borrow a bit of graph theory from \cite[Section 3]{Lolk1}. If $v \in E^0$ does not admit a choice, 
every closed path $\alpha$ based at $v$ has a unique word decomposition $\alpha=\gamma^{-1}\beta\gamma$ for a (possibly trivial) admissible path $\gamma$ and a cycle $\beta$, and we will 
say that $\alpha$ is a \textit{simple closed path} if $\gamma$ does not repeat a vertex, and $\beta$ is a simple cycle. The set  
$$\mathbb{F}_v:=\{\text{closed paths based at }v\} \cup \{1\}$$
defines a subgroup $\mathbb{F}_v \le \mathbb{F}$, and every closed path based at $v$ is a reduced product of simple closed paths. It follows that such $v$ admits a 
unique simple closed path (up to inversion) if and only if $\mathbb{F}_v \cong \mathbb{Z}$. Moreover, $\Omega(E,C)_v=\{\xi\}$ is a one-point set and $\text{Stab}(\xi)=\mathbb{F}_v$, 
so $\xi \notin \Omega(E,C)^{\textup{TF}}$ if and only if $v$ admits a closed path. It follows from the proof of \cite[Theorem 10.5]{AE} that every $\eta \in \Omega(E,C) \setminus \Omega(E,C)^{\textup{TF}}$ is of the form $\eta=\theta_\gamma(\xi)$ 
for some $\gamma \in \xi$ with $\xi$ as above -- in particular, the complement $\Omega(E,C) \setminus \Omega(E,C)^{\textup{TF}}$ is discrete.
\end{remark}

Before we progress any further, let us consider a 
somewhat trivial, but not uninteresting, example.

\begin{example}\label{ex:Integers}
If $\theta \colon G \act \Omega$ is any partial action, and $x \in \Omega^{\textup{TF}}$ satisfies $\text{Stab}(x) \cong \mathbb{Z}$, then $\theta$ is automatically strongly topologically free in $x$. Indeed, given $1 \ne g_1,\ldots,g_n \in \text{Stab}(x)$, we may write $g_i=g^{k_i}$ with $g$ a generator of $\text{Stab}(x)$, and we can safely assume that $k_i > 0$ for all $i$. For any open neighbourhood $U$ of $x$, we consider $h:=g^{k_1 \cdots k_n} \in G$ and pick $y \in U$ with $h \notin \text{Stab}(y)$ using topological freeness in $x$. Then we obviously have $g_i \notin \text{Stab}(y)$ as well, so $\theta$ is indeed strongly topologically free in $x$.
\end{example}

The main result of \cite{Lolk2} is a characterisation of nuclearity and exactness of $\mathcal{O}^{(r)}(E,C)$ in terms of a graph theoretic \textit{Condition} (\textit{N}). Another equivalent condition is that every stabiliser of $\theta^{(E,C)}$ is either trivial or isomorphic to $\mathbb{Z}$, so from the above example we obtain the following proposition:

\begin{proposition}
If $(E,C)$ satisfies Condition \textup{(}N\textup{)}, then the restriction of $\theta^{(E,C)}$ to any closed invariant subspace is relatively strongly topologically free.
\end{proposition}

For general separated graphs, we can only handle restrictions to closed invariant subspaces of finite type.

\begin{proposition}\label{prop:GraphActIsRSTF}
If $(E,C)$ is any finite bipartite separated graph, then $\theta^{(E,C)}$ is relatively strongly topologically free.
\end{proposition}

\begin{proof}
Whenever we write $\beta\alpha$ for admissible paths $\alpha$ and $\beta$ in the following, we shall mean the concatenated product, that is, we do not allow for cancellation.
Now consider any $\xi \in \Omega(E,C)^{\text{TF}}$ and $g_1,\ldots,g_n \in \text{Stab}(\xi)$, assuming, without loss of generality, that no pair $g_i,g_j$ of distinct elements generate a rank one subgroup. For every $i$, we may write $g_i= \mu_i^{-1}\sigma_i\mu_i$ for cycles $\sigma_i$.
Since the action is topologically free in $\xi$, there 
exist admissible paths $\beta_i$ and edges $x_i$ with $\vert X_{x_i} \vert \ge 2$, such that $x_i^{-1}\beta_i\mu_i \in \xi$. Either $x_i^{-1}\beta_i\sigma_i$ or $x_i^{-1}\beta_i\sigma_i^{-1}$ must be admissible, and without loss of generality, we can assume the former. Now given any open neighbourhood $U$ of $\xi$, we have 
$$\xi \in \Omega(E,C)_B:=\{\eta \in \Omega(E,C) \mid B \subset \eta\} \subset U$$
for a sufficiently big ball $B:=\xi^N$. We may of course assume that $N> |x_i^{-1}  \beta_i \mu_i |$ for all $i$. Picking $l_i \ge 1$ such that 
$\vert \sigma_i^{l_i}\mu_i \vert \ge N$ and $x_i \ne y_i \in X_{x_i}$,
we then consider the set 
$$\omega:=B \cup \{y_i^{-1}\beta_i\sigma_i^{l_i}\mu_i \mid i =1,\ldots,n\}.$$
It should be clear that there exists $\eta \in \Omega(E,C)$ with $\omega \subset \eta$, so in particular $\eta \in U$. Finally, observe that $g_i \notin \text{Stab}(\eta)$ by construction since $x_i^{-1}\beta_i\mu_i \in B \subset \eta$ and
$$x_i^{-1}\beta_i\mu_i \notin \eta \cdot \mu_i^{-1}\sigma_i^{-l_i}\mu_i=\theta_{g_i^{l_i}}(\eta),$$
so $\theta$ is indeed strongly topologically free in $\xi$.
\end{proof}

Now let us instead consider a non-example.

\begin{example}
For $n \in \mathbb{Z}$, we define $f_n^1,f_n^2 \in \{0,1\}^{\mathbb{Z}^2}$ by
$$f_n^1(a,b)=\left\{\begin{array}{cl}
0 & \If b > n \\
1 & \If b \le n
\end{array} \right. 
\andspace
f_n^2(a,b)=\left\{\begin{array}{cl}
0 & \If a > n \\
1 & \If a \le n
\end{array} \right. .
$$
Then the $\mathbb{Z}^2$-shift on $\{0,1\}$ restricts to an action $\theta \colon \mathbb{Z}^2 \act \Omega$ on the compact Hausdorff space $\Omega:=\{f_n^1,f_n^2,0,1 \mid n \in \mathbb{Z}\}$. Every $f_n^i$ is isolated with $f_n^i \to 0$ for $n \to -\infty$ and $f_n^i \to 1$ for $n \to \infty$. By construction, $\text{Stab}(f_n^1)=\mathbb{Z} \oplus 0$ and $\text{Stab}(f_n^2)= 0 \oplus \mathbb{Z}$ for all $n$, so $\Omega^{\textup{TF}}=\{0,1\}$. However, $\theta$ is not strongly topologically free in $0$ or $1$ since any $f_n^i$ is either fixed by $a$ or $b$. In conclusion, relative strong topological freeness is not automatic.
\end{example}

\begin{remark}
We have no examples of partial actions of free groups that are not relatively strongly topologically free, but we suspect that such examples exist. However, it is notable that whenever a partial action $\theta \colon \mathbb{F} \act \Omega$ of a free group is topologically free in $x$, and we consider only two elements $g_1,g_2 \in \text{Stab}(x)$, then given any open neighbourhood $U$ of $x$ we \textit{can} find $y \in U$ with $g_1,g_2 \notin \text{Stab}(y)$. Indeed, we may assume that $g_1$ and $g_2$ do not generate a free subgroup of rank one, so that the commutator $[g_1,g_2]=g_1^{-1}g_2^{-1}g_1g_2$ is non-trivial. Let $V_i$ be an open neighbourhood of $x$ so that $\theta_{g_i}(V_i) \subset U$ and consider $V:=U \cap V_1 \cap V_2$. Applying topological freeness with respect to $[g_1,g_2]$, we obtain $y \in V$ with $[g_1,g_2] \notin \text{Stab}(y)$. Obviously, we cannot have $g_1,g_2 \in \text{Stab}(y)$, and if $g_1,g_2  \notin \text{Stab}(y)$, then we are done. We may therefore assume that exactly one of $g_1$ and $g_2$ belongs to $\text{Stab}(y)$. If $g_1 \in \text{Stab}(y)$, we consider $y \ne y':=\theta_{g_2}(y) \in U$ instead. Observe that if $g_1 \in \text{Stab}(y')$, then 
$$\theta_{[g_1,g_2]}(y)=\theta_{g_1^{-1}g_2^{-1}g_1}(y')=\theta_{g_1^{-1}g_2^{-1}}(y')=\theta_{g_1^{-1}}(y)=y,$$
which contradicts our choice of $y$. We conclude that $g_1,g_2 \notin \text{Stab}(y')$ as desired. In case $g_2 \in \text{Stab}(y)$, we apply the exact same argument with $g_1$ and $g_2$ reversed: this is possible since $[g_2,g_1]=[g_1,g_2]^{-1} \notin \text{Stab}(y)$.
\end{remark}

It is evident from the definition that $\Omega \setminus \Omega^{\textup{TF}}$ is an open invariant subspace, so there is a corresponding ``obstruction ideal'':

\begin{definition}
Given a partial action $\theta \colon G \act \Omega$ of a discrete group on a locally compact Hausdorff space, we define an ideal $J^o=J_\theta^o:=C_0(\Omega \setminus \Omega^{\textup{TF}}) \rtimes_r G$ of $C_0(\Omega) \rtimes_r G$.
\end{definition}

We now put relative strong topological freeness to work. The proof of the following theorem is modelled over that of \cite[Theorem 29.5]{Exel}, but the statement is somewhat more general.

\begin{theorem}\label{thm:GeneralIdeals}
Let $\theta \colon G \act \Omega$ denote a partial action of a discrete group $G$ on a locally compact Hausdorff space. Suppose that $G$ is exact and that $\theta$ is relatively strongly topologically free. If $J \cap C_0(\Omega)=\{0\}$ for an ideal $J \ideal C_0(\Omega) \rtimes_r G$, then $J \subset J^o$.
\end{theorem}
\begin{proof}
Assuming that $J \not\subset J^o$, we take any $a \in J \setminus J^o$ and consider $w:=a^*a \in J \setminus J^o$. Now consider the commutative diagram
\begin{center}
\begin{tikzpicture}[>=angle 90]
\matrix(a)[matrix of math nodes,
row sep=2em, column sep=2.5em,
text height=1.5ex, text depth=0.25ex]
{C_0(\Omega) \rtimes_r G & C_0(\Omega^{\text{TF}}) \rtimes_r G \\
C_0(\Omega) & C_0(\Omega^{\text{TF}}) \\};
\path[->](a-1-1) edge node[above]{$p_*$} (a-1-2);
\path[->](a-2-1) edge node[below]{$p$} (a-2-2);
\path[->](a-1-1) edge node[left]{$E$} (a-2-1);
\path[->](a-1-2) edge node[right]{$F$} (a-2-2);
\end{tikzpicture},
\end{center}
where $E$ and $F$ are the canonical conditional expectations. By exactness of $G$, we have $\ker(p_*)=J^o$, hence $p_*(w) \ne 0$. Faithfulness of $F$ then implies
$$p(E(w))=F(p_*(w)) \ne 0,$$
so $f:=E(w)$ attains a non-zero value on $\Omega^{\text{TF}}$. Let $x_0 \in \Omega^{\text{TF}}$ be such that $\vert f(x_0) \vert = \sup_{x \in \Omega^{\text{TF}}} \vert f(x) \vert$ and take any $0 < \varepsilon < \frac{\vert f(x_0) \vert}{2}$. By Urysohn's Lemma (applied to the one point compactification of $\Omega$), there exists $u \in C_0(\Omega)$ with $0 \le u \le 1$, $u(x_0)=1$ and $u(x) = 0$ whenever $\vert f(x) \vert \ge \vert f(x_0) \vert+ \varepsilon/4$. 
Consider $z:=uw \in J \setminus J^o$ and note that $E(z)=uE(w)= uf$, hence $\vert f(x_0) \vert \le \Vert E(z) \Vert \le \vert f(x_0) \vert+ \varepsilon/4$. We claim that there exists a function $h \in C_0(\Omega)$ such that
\begin{enumerate}[(1)]
\item $0 \le h \le 1$,
\item $\Vert h E(z) h \Vert > \Vert E(z) \Vert - \varepsilon$,
\item $\Vert hzh-hE(z)h \Vert < \varepsilon$.
\end{enumerate}
To see this, we first pick $b \in C_0(\Omega) \rtimes_{\textup{alg}} G$ with $\Vert z - b \Vert < \varepsilon/4$ and write $b=b_1+\sum_{g \in T} b_g\delta_g$ for a finite set $T \subset G \setminus \{1\}$. If $T=\emptyset$, 
then the claim is easily verified, so we may assume that $T \ne \emptyset$. Since $\theta$ is strongly topologically free in $x_0$, we can find $x_1 \in \Omega$ with $\vert b_1(x_1) - b_1(x_0) \vert < \varepsilon/4$ and $T \cap \text{Stab}(x_1) = \emptyset$. We may apply \cite[Lemma 29.4]{Exel} for every $g \in T$ to obtain $h_g \in C_0(\Omega)$ with $0 \le h_g \le 1$, $h_g(x_1)=1$ and $\Vert h_g (b_g\delta_g) h_g \Vert < \frac{\varepsilon}{2 \vert T \vert}$. Setting $h:=\prod_{g \in T} h_g$, (2) then follows from the calculation
\begin{align*}
\Vert h E(z) h \Vert  &> \Vert h E(b)h \Vert- \varepsilon/4 = \Vert hb_1h \Vert - \varepsilon/4 \\
&\ge \vert h(x_1)b_1(x_1)h(x_1) \vert- \varepsilon/4 = \vert b_1(x_1) \vert- \varepsilon/4  \\
&> \vert b_1(x_0) \vert - \varepsilon/2 > \vert E(z)(x_0) \vert - 3\varepsilon/4 = \vert f(x_0) \vert - 3\varepsilon/4 \\
&\ge \Vert E(z) \Vert - \varepsilon.
\end{align*}
In order to check (3), we first observe that
\begin{align*}
\Vert hbh-hb_1h \Vert &= \Vert \sum_{g \in T} h(b_g\delta_g)h \Vert \le \sum_{g \in T} \Vert h_g(b_g\delta_g)h_g \Vert < \varepsilon/2,
\end{align*}
hence
\begin{align*}
\Vert hzh - hE(z)h \Vert \le \Vert hzh- hbh \Vert + \Vert hbh - hb_1h \Vert + \Vert hb_1h - hE(z)h \Vert < \varepsilon.
\end{align*}
Having verified the claim, we let $\pi$ denote the quotient map $C_0(\Omega) \rtimes_r G \to (C_0(\Omega) \rtimes_r G)/J$. Since $z \in J$ and $J \cap C_0(\Omega)=0$, we have
\begin{align*}
\Vert E(z) \Vert &< \Vert h E(z) h \Vert + \varepsilon = \Vert \pi(hE(z)h - hzh) \Vert + \varepsilon \\
&\le \Vert hE(z)h-hzh \Vert + \varepsilon< 2\varepsilon.
\end{align*}
But at the same time, $\Vert E(z) \Vert \ge \vert f(x_0) \vert > 2 \varepsilon$, a contradiction.
\end{proof}

We need to specialize the above theorem a bit before we can apply it to our setting. The following is an ever useful observation.

\begin{lemma}\label{lem:IsolatedPoints}
Let $\theta \colon G \act \Omega$ denote a partial action of a discrete group on a locally compact Hausdorff space. If $x \in \Omega$ is isolated, then
$$1_x (C_0(\Omega) \rtimes_{(r)} G)1_x \cong C_{(r)}^*(\textup{Stab}(x)),$$
where $1_x$ denotes the indicator function in $x$.
\end{lemma}
\begin{proof}
By \cite[Proposition 6.1 and Corollary 6.3]{AEK}, we have embeddings
$$C_{(r)}^*(\text{Stab}(x)) \cong C(\{x\}) \rtimes_{(r)} \text{Stab}(x) \hookrightarrow C_0(\Omega) \rtimes_{(r)} \text{Stab}(x) \hookrightarrow C_0(\Omega) \rtimes_{(r)} G,$$
and the composition clearly maps onto the corner $1_x (C_0(\Omega) \rtimes_{(r)} G)1_x$.
\end{proof}

\begin{definition}\label{def:Setup}
Let $\theta \colon G \act \Omega$ denote a partial action on a locally compact Hausdorff space, and suppose that $\mathcal{U} \subset \mathbb{O}^G(\Omega)$ is a collection of open invariant subsets $U \subset \Omega$ for which the restriction $\theta\vert_{Z_U}$ to $Z_U:=\Omega \setminus U$ satisfies the following two conditions:
\begin{enumerate}
\item $\theta\vert_{Z_U}$ is relatively strongly topologically free,
\item the space $W_U:=Z_U \setminus Z_U^{\textup{TF}}$ is discrete.
\end{enumerate}
Observe that if $U \subset V$ for $U,V \in \mathcal{U}$, then $W_{U} \setminus V \subset W_V$. For any $U$, we may therefore choose a set of representatives $\Lambda_U$ for the orbit space $W_U/G$ such that $\Lambda_U \setminus V \subset \Lambda_V$ whenever $U \subset V$. We then introduce a set
$$\mathcal{I}_\mathcal{U}(\theta):=\Big\{\big(U,(I_U^x)_{x \in \Lambda_U}\big) \mid U \in \mathcal{U}, I_U^x \text{ is a proper ideal of } C_r^*(\text{Stab}(x)) \Big\}$$
and equip it with the partial order
$$\big(U,(I_U^x)_{x \in \Lambda_U}\big) \le \big(V,(I_V^x)_{x \in \Lambda_V}\big) \Leftrightarrow U \subset V \text{ and } I_U^x \subset I_V^x \text{ for all } x \in \Lambda_U \setminus V.$$
For notational simplicity, we will usually write $I_U^\bullet=(I_U^x)_{x \in \Lambda_U}$, and we finally write $\mathcal{I}_\mathcal{U}(C_0(\Omega) \rtimes_r G)$ for the collection of ideals $J \ideal C_0(\Omega) \rtimes_r G$ satisfying $J \cap C_0(\Omega)=C_0(U)$ for some $U \in \mathcal{U}$. 
\end{definition}

\begin{corollary}\label{cor:IdealLattice}
Let $\theta \colon G \act \Omega$ denote a partial action of an exact group on a locally compact Hausdorff space with the setup from Definition~\ref{def:Setup}. Then there is a canonical order isomorphism 
$$\mathcal{I}_\mathcal{U}(\theta) \to \mathcal{I}_\mathcal{U}(C_0(\Omega) \rtimes_r G), \quad (U,I_U^\bullet) \mapsto J(U,I_U^\bullet),$$
with the following properties:
\begin{enumerate}
\item[\textup{(1)}] $J(U,I_U^\bullet) \cap C_0(\Omega)=C_0(U)$.
\item[\textup{(2)}] The quotient $J(U,I_U^\bullet)/(C_0(U) \rtimes_r G)$ is canonically Morita equivalent to $\bigoplus_{x \in \Lambda_U} I_U^x$.
\end{enumerate}
\end{corollary}

\begin{proof}
Let us write $J_U^o:=J_{\theta\vert_{Z_U}}^o$. By Definition~\ref{def:Setup}(2), every orbit in $W_U$ is clopen, so there is a canonical identification
$$J_U^o \cong \bigoplus_{x \in \Lambda_U} C_0(G.x) \rtimes_r G.$$
From Lemma~\ref{lem:IsolatedPoints}, $C_r^*(\text{Stab}(x))$ sits as a full corner of $C_0(G.x) \rtimes_r G$ for all $x \in \Lambda_U$, and we let $\tilde{I}_U^x$ denote the ideal generated by $I_U^x$ inside $C_0(G.x) \rtimes_r G$. We may then define an ideal $\tilde{I}_U:=\bigoplus_{x \in \Lambda_U} \tilde{I}_U^x \ideal J_U^o$, and letting $\pi_U$ denote the quotient map $C_0(\Omega) \rtimes_r G \to C_0(Z_U) \rtimes_r G$, we finally set $J(U,I_U^\bullet):=\pi_U^{-1}(\tilde{I}_U)$. We proceed to verify the properties of the map $(U,I_U^\bullet) \mapsto J(U,I_U^\bullet)$, assuming that $U \subset V$. Observe that the quotient map
$$\pi_{U,V} \colon C_0(Z_U) \rtimes_r G \to C_0(Z_V) \rtimes_r G$$
maps $C_r^*(G.x) \rtimes_r G$ to $\{0\}$ if $x \in V$, and that it maps $C_0(G.x) \rtimes_r G \subset J_U^o$ identically to $C_0(G.x) \rtimes_r G \subset J_V^o$ otherwise. 
It follows that $(U,I_U^\bullet) \le (V,I_V^\bullet)$ if and only if $J(U,I_U^\bullet) \subset J(V,I_V^\bullet)$. Both (1) and (2) should be obvious from the above, and injectivity is 
immediate from these. To see that our map is surjective, take any ideal $J \ideal C_0(\Omega) \rtimes_r G$ with $J \cap C_0(\Omega)=C_0(U)$ for $U \in \mathcal{U}$, 
and consider $J_U:=J/(C_0(U) \rtimes_r G) \ideal C_0(Z_U) \rtimes_r G$. Since $\theta\vert_{Z_U}$ is assumed relatively strongly topologically free, we see 
that $J_U \subset J_U^o$ from Theorem~\ref{thm:GeneralIdeals}. In particular, we may write $J_U \cong \bigoplus_{x \in \Lambda_U} \tilde{I}_U^{x}$ for ideals $\tilde{I}_U^x \ideal C_0(G.x) \rtimes_r G$. 
Observe that $\tilde{I}_U^x$ must be proper,  for otherwise we would have $x \in U$. Setting $I_x^U:=1_x(\tilde{I}_U^x)1_x$, we finally have $J(U,I_U^\bullet)=J$.
\end{proof}

We are almost ready to put our results to use at this point, but we still need to introduce some bookkeeping.

\begin{construction}
Consider any finite bipartite separated graph $(E,C)$, and let $(F_\infty,D^\infty)$ be its Bratteli diagram. We denote by $\mathfrak{C}(E,C)$ the set of vertices $v \in E^0$ which admit no choices and a simple cycle $\alpha$ such that every closed path based at $v$ is a power of $\alpha$, modulo the relation
$$u \sim v \Leftrightarrow \text{$u$ and $v$ belong to the same cycle}.$$
It follows from Remark~\ref{rem:IsolatedPoints} that $\mathfrak{C}(E,C)$ is in canonical bijective correspondence with the orbit space $W/\mathbb{F}$ for 
$$W:=\{ \xi \in \Omega(E,C) \setminus \Omega(E,C)^{\textup{TF}} \mid \text{Stab}(\xi) \cong \mathbb{Z}\}.$$
Observe that there is a natural map $\mathfrak{C}(E,C) \to \mathfrak{C}(E_1,C^1)$ given by $[v] \mapsto [v]$ whenever $v \in E^{0,1}$, and that it is in fact a bijection. This can either be seen by direct arguments, by Theorem~\ref{thm:NBallIsNGraph}, or most easily by \cite[Lemma 5.2]{Lolk1}. Consequently, if $H \in \mathcal{H}_\textup{fin}(F_\infty,D^\infty)$ satisfies $H=H^n$, we may identify $\mathfrak{C}(E_n/H^{(n)},C^n/H^{(n)})$ with $\mathfrak{C}(E_m/H^{(m)},C^m/H^{(m)})$ whenever $m \ge n$; formally, we do this by setting
$$\mathfrak{C}(H):=\varinjlim_{m \ge n} \mathfrak{C}(E_m/H^{(m)},C^m/H^{(m)}).$$
Observe that whenever $H_1 \subset H_2$ for $H_1,H_2 \in \mathcal{H}_{\textup{fin}}(F_\infty,D^\infty)$, we have an inclusion 
$$\{\mathfrak{c} \in \mathfrak{C}(H_1) \mid \mathfrak{c} \not\subset H_2 \} \subset \mathfrak{C}(H_2).$$
Indeed, whenever $m \ge n$ for sufficiently large $n$, we have representations
$$\mathfrak{C}(H_1)=\mathfrak{C}(E_m/H_1^{(m)},C^m/H_1^{(m)}) \andspace \mathfrak{C}(H_1)=\mathfrak{C}(E_m/H_2^{(m)},C^m/H_2^{(m)}),$$
and $H_1^{(m)} \subset H_2^{(m)}$. Consequently, there in an inclusion
$$\{\mathfrak{c} \in \mathfrak{C}(E_m/H_1^{(m)},C^m/H_1^{(m)}) \mid \mathfrak{c} \not\subset H_2^{(m)}\} \subset \mathfrak{C}(E_m/H_2^{(m)},C^m/H_2^{(m)}),$$
and this does not depend on $m$. Letting $\mathbb{O}_p(\mathbb{T})$ denote the collection of proper open subsets of $\mathbb{T}$, we finally define a set
$$\mathcal{I}_{\textup{fin}}(E,C)=\big\{(H,T) \mid H \in \mathcal{H}_\textup{fin}(F_\infty,D^\infty), T \in \mathbb{O}_p(\mathbb{T})^{\mathfrak{C}(H)} \big\}$$
and equip it with the partial ordering
$$(H_1,T_1) \le (H_2,T_2) \Leftrightarrow H_1 \subset H_2 \text{ and } T_1(\mathfrak{c}) \subset T_2(\mathfrak{c}) \text{ for all } \mathfrak{c} \in \mathfrak{C}(H_1) \text{ with } \mathfrak{c} \not \subset H_2.$$
\end{construction}

\begin{theorem}
For any finite bipartite separated graph $(E,C)$, there is a canonical lattice isomorphism 
$$\mathcal{I}_{\textup{fin}}(E,C) \to \mathcal{I}_{\textup{fin}}(\mathcal{O}^r(E,C)) \quad , \quad (H,T) \mapsto I(H,T),$$
with the following properties:
\begin{enumerate}
\item[\textup{(1)}] $H_{I(H,T)}=H$.
\item[\textup{(2)}] The quotient $I(H,T)/I(H)$ is Morita equivalent to $\bigoplus_{\mathfrak{c} \in \mathfrak{C}(H)} C_0(T(\mathfrak{c}))$.
\end{enumerate}
In particular, a finite type ideal of $\mathcal{O}^r(E,C)$ is generated by its projections if and only if it is induced.
\end{theorem}

\begin{proof}
Let $\mathcal{U}:=\{\Omega(E,C)^H \mid H \in \mathcal{H}_\textup{fin}(F_\infty,D^\infty)\}$ and observe that Corollary~\ref{cor:IdealLattice} may be applied thanks to Remark~\ref{rem:IsolatedPoints}, Proposition~\ref{prop:GraphActIsRSTF} and Theorem~\ref{thm:GeneralIdeals}. Since $\mathbb{F}_n$ is $C^*$-simple for all $n \ge 2$ \cite{Powers}, we see that $I_U^x=0$ whenever $\text{Stab}(x) \not\cong \mathbb{Z}$, and proper ideals of $C_r^*(\mathbb{Z}) \cong C(\mathbb{T})$ correspond to proper open subsets $T \subset \mathbb{T}$. Finally, if $H=H^n$ and $U=\Omega(E,C)^H$, then the orbits of points $\xi \in W_U$ with stabiliser $\mathbb{Z}$ correspond canonically to the elements of $\mathfrak{C}(H)$. The above statement is therefore exactly the conclusion that can be drawn from Corollary~\ref{cor:IdealLattice}.
\end{proof}

\begin{remark}
We have no hope of achieving a similar result for arbitrary ideals of $\mathcal{O}^r(E,C)$ except in special cases. Indeed, we suspect that for an infinite type subspace $Z$, the restriction $\theta^{(E,C)}\vert_Z \colon \mathbb{F} \act Z$ need not be relatively strongly topologically free, and one can easily 
find examples where the space $Z \setminus Z^{\textup{TF}}$ is not discrete.
\end{remark}

We finally apply our results to classical graph $C^*$-algebras to provide a new proof for the description of the ideal lattice first obtained by Hong and Szyma{\'n}ski in \cite{HS}. We first recall a bit of terminology and a few results.

\begin{definition}[\cite{BHRSByRuy}]
Let $E$ denote any directed graph and denote the set of hereditary and saturated subsets by $\mathcal{H}(E)$. For any $H \in \mathcal{H}(E)$, there is a set of \textit{breaking vertices} for $H$ given by
$$H_{\infty}^{\textup{fin}}:=\{ v \in E^0 \setminus H \colon \vert r^{-1}(v) \vert = \infty \text{ and } 0 < \vert r^{-1}(v) \cap s^{-1}(E^0 \setminus H) \vert < \infty\},$$
and pairs $(H,B)$ with $B \subset H_\infty^\textup{fin}$ are called \textit{admissible}. For any such pair, one defines a quotient graph $E/(H,B)$ as
\begin{align*}
(E/(H,B))^0 &:=(E^0 \setminus H) \cup \{\beta(v) \mid v \in H_\infty^{\textup{fin}} \setminus B\}, \\
(E/(H,B))^1&:= s^{-1}(E^0 \setminus H) \cup \{\beta(e) \mid e \in E^1, s(e) \in H_\infty^{\textup{fin}} \setminus B\}
\end{align*}
with $r,s$ extended by $r(\beta(e)):=r(e)$ and $s(\beta(e)):=\beta(s(e))$. We finally order the admissible pairs $(H,B)$ by
$$(H_1,B_1) \le (H_2,B_2) \Leftrightarrow H_1 \subset H_2 \text{ and } 
B_1 \setminus B_2 \subset H_2.$$
\end{definition}

\begin{theorem}\label{thm:BPS}
Let $E$ denote any directed graph. There exists a partial action $\theta^E \colon \mathbb{F} \act \partial E$ of $\mathbb{F}:=\mathbb{F}(E^1)$ on a totally disconnected, locally compact Hausdorff space $\partial E$ with the following properties:
\begin{enumerate}
\item[(\textup{1)}] $C^*(E) \cong C_0(\partial E) \rtimes \mathbb{F}(E^1)$ canonically.
\item[\textup{(2)}] There is a canonical lattice isomorphism
$$\{(H,B) \mid H \in \mathcal{H}(E) \text{ and } B \subset H_\infty^\textup{fin}\} \to \mathbb{O}^\mathbb{F}(\partial E), \quad (H,B) \mapsto U(H,B),$$
and $\theta^E\vert_{\partial E \setminus U(H,B)} \approx \theta^{E/(H,B)}$ for every admissible pair $(H,B)$. In particular, 
$$C^*(E)/I(H,B) \cong C^*(E/(H,B)),$$
where $I(H,B)$ is the ideal induced from $U(H,B)$.
\end{enumerate}
\end{theorem}
\begin{proof}
(1) is proven in \cite[Chapter 37]{Exel}, and (2) follows from the description of gauge-invariant ideals in \cite{BHRSByRuy} (one can also prove (2) directly with fairly little work).
\end{proof}

As for the separated graphs, we have to introduce some bookkeeping:

\begin{definition}
For any directed graph $E$, we let $\mathfrak{C}(E)$ denote the set of vertices $v \in E^0$ which admit a cycle without any entries, modulo the relation
$$u \sim v \Leftrightarrow \text{$u$ and $v$ belong to the same cycle}.$$
Observe that if $(H,B)$ is an admissible pair, then $\mathfrak{C}(E/(H,B)) = \mathfrak{C}(E/H)$ since the additional vertices in $E/H$ are all sources. Given any $H \in \mathcal{H}(E)$, we set $\mathfrak{C}(H):=\mathfrak{C}(E/H)$ and observe that if $H_1 \subset H_2$, then we have an inclusion
$$\{\mathfrak{c} \in \mathfrak{C}(H_1) \mid \mathfrak{c} \not\subset H_2\} \subset \mathfrak{C}(H_2).$$ 
We may in turn define a lattice
$$\mathcal{I}(E):=\{(H,B,T) \mid H \in \mathcal{H}, B \subset H_\infty^{\textup{fin}}, T \in \mathbb{O}_p(\mathbb{T})^{\mathfrak{C}(H)} \}$$
ordered by
\begin{align*}
(H_1,B_1,T_1) \le (H_2,B_2,T_2) \Leftrightarrow &(H_1,B_1) \le (H_2,B_2) \text{ and } T_1(\mathfrak{c}) \subset T_2(\mathfrak{c}) \\ &\text{ for all } \mathfrak{c} \in \mathfrak{C}(H_1) \text{ with } \mathfrak{c} \not\subset H_2.
\end{align*}
Finally, we denote the ideal lattice of $C^*(E)$ by $\mathcal{I}(C^*(E))$, and given any $J \in \mathcal{I}(C^*(E))$, we write $(H_J,B_J)$ for the admissible pair satisfying $J \cap C_0(\partial E) = C_0(U(H_J,B_J))$.
\end{definition}

\begin{theorem}
For any directed graph $E$, there is a canonical lattice isomorphism
$$\mathcal{I}(E) \to \mathcal{I}(C^*(E)), \quad (H,B,T) \mapsto I(H,B,T),$$
with the following properties:
\begin{enumerate}
\item[\textup{(1)}] $H_{I(H,B,T)}=H$ and $B_{I(H,B,T)}=B$.
\item[\textup{(2)}] The quotient $I(H,B,T)/I(H,B)$ is Morita equivalent to $\bigoplus_{\mathfrak{c} \in \mathfrak{C}(H,B)} C_0(T(\mathfrak{c}))$.
\end{enumerate}
\end{theorem}

\begin{proof}
Let $\mathcal{U}:=\mathbb{O}^\mathbb{F}(\partial E)$ and observe that Corollary~\ref{cor:IdealLattice} applies due to Theorem~\ref{thm:BPS} and Example~\ref{ex:Integers}, since any stabiliser is either trivial or isomorphic to $\mathbb{Z}$. For any $U=U(H,B) \in \mathcal{U}$, the orbits of points $x \in W_U$ with $\text{Stab}(x) \cong \mathbb{Z}$ correspond to the elements of $\mathfrak{C}(H)$, so Corollary~\ref{cor:IdealLattice} reduces to the above statement.
\end{proof}

\begin{remark}
Observe that Theorem~\ref{thm:GeneralIdeals} also provides a new proof of Szyma{\'n}ski's general Cuntz-Krieger Uniqueness Theorem \cite{Szy} when applied to the boundary path space action $\theta^E$.
\end{remark}

\begin{remark}
We finally remark that Theorem~\ref{thm:GeneralIdeals} has an analogue for algebraic crossed products $C_K(\Omega) \rtimes G$, where $C_K(\Omega)$ denotes the 
algebra of compactly supported locally constant function $\Omega \to K$ when $K$ is given the discrete topology. The assumptions will 
have to be slightly different: on one hand, there is no need for exactness of the group, but on the other hand, one needs the space $\Omega$ to be totally 
disconnected to have sufficiently many continuous functions. The proof should be a bit simpler, although one will have to avoid $C^*$-techniques. The description of the lattice of ideals of the Leavitt path algebra $L_K(E)$ of an arbitrary graph $E$, obtained in \cite[Theorem 2.8.10]{AAS}, can also be obtained using this approach.  
\end{remark}

\section{$\mathcal V$-simplicity}
\label{sec:Vsimplicity}

We continue our investigation of the ideal structure with the study of $\mathcal V$-simplicity, that is, we want to compute the algebras of graphs $(E,C)$ for which the monoid $M(F_{\infty}, D^{\infty})$ is order simple. By Theorem \ref{thm:idealsOr}, this is equivalent to saying that $I\cap C(\Omega (E,C))= 0$ for every
proper ideal $I$ of $\mathcal O^r( E, C)$, or $\mathcal{H}(F_\infty,D^\infty)=\{\emptyset,F_\infty^0\}$. We will simply say that $(E,C)$ is {\it simple} in this case, and any of the algebras $\mathcal O (E,C)$, $\mathcal O^r  (E,C)$ and $\Lab (E,C)$ will be called $\mathcal V$\textit{-simple}. The second named author proves similar results in \cite[Section 4]{Lolk1} by studying minimality of the partial action $\theta^{(E,C)}$; our study here, on the other hand, is purely graph-theoretical.

We state the main result of this section right away:

   \begin{theorem}
    \label{thm:dichotomy}
Let $(E,C)$ be a simple finite bipartite separated graph.
     Then $L_K(E,C)=\Lab_K(E,C)$ and $C^*(E,C) = \mathcal O (E,C)$, and one of the following holds:
\begin{enumerate}
\item[\textup{(1)}] $L_K(E,C)$ is isomorphic to a simple Leavitt path algebra, and $C^*(E,C)$ is isomorphic to a simple graph $C^*$-algebra of a non-separated graph.
\item[\textup{(2)}] $L_K(E,C)$ is Morita-equivalent to $K[\mathbb{F}_n]$ and $C^*(E,C)$ is Morita-equivalent to $C^*(\mathbb{F}_n)$, where $\mathbb F_n$ is a free group of rank $n$ with $1 \le n<\infty $.
\end{enumerate}
 In the latter case, $\mathcal{O}^r(E,C)$ is Morita equivalent to $C^*_r(\mathbb{F}_n)$.
        \end{theorem}
From Theorem \ref{thm:dichotomy} and the fact that the reduced group C*-algebra $C_r^*(\mathbb F_n)$ of a free group $\mathbb F_n$ of rank $n>1$ is simple \cite{Powers}, we obtain:

   \begin{corollary}
    \label{cor:dichotomy}
   Let $(E,C)$ be a finite bipartite separated graph. If $\mathcal O^r(E,C)$ is $\mathcal V$-simple, then either $\mathcal O^r(E,C)$ is isomorphic to a simple $C^*$-algebra, or it is Morita equivalent to $C(\mathbb T)$.
   \end{corollary}

We develop the proof in various steps.  We begin with a simple observation.

\begin{lemma}
 \label{lem:all-layers-simple}
 Let $(E,C)$ be a finite bipartite separated graph. Then $(E,C)$ is simple if and only if $\mathcal H (E_n,C^n) = \{ \emptyset, E_n^0 \}$ for all $n\ge 0$.
 \end{lemma}

\begin{proof}
 Let $n\ge 0$ be given. It follows from Lemma \ref{lem:hersatfromE} that there is an injective order-preserving map
 $\mathcal H  (E_n,C^n) \to \mathcal H(F_{\infty}, D^{\infty})$, so $\mathcal H(E_n, C^n)$
 is trivial if $ \mathcal H (F_{\infty}, D^{\infty}) $ is trivial.

 Conversely assume that $\mathcal H (E_n,C^n)$ is trivial for all $n\ge 0$.
 If $H\in  \mathcal H(F_{\infty}, D^{\infty})$, then $H= \cup _{n=0} H^{(n)}$, where $H^{(n)}:= H\cap E_n^0\in \mathcal H (E_n,C^n)$
for all $n$. If $H^{(n)}= E_n^0$ for some $n$, then
$H=F_{\infty}^0$. Otherwise $H^{(n)} = \emptyset $ for all $n$, so
$H=\emptyset$.
 \end{proof}

A useful property of the separated graphs $(E_n,C^n)$ associated to a finite bipartite separated graph $(E,C)$ is the following: if $X\in C^n_v$ and $n\ge 1$, then $s(x)\ne s(y)$
whenever $x,y$ are different elements of $X$. This follows immediately from the definition of these graphs.

We now obtain some necessary conditions for $\mathcal V$-simplicity.

\begin{lemma}
 \label{lem:necess-simple1}
 Let $(E,C)$ be a simple finite bipartite separated graph.
 Then for all $v\in F_{\infty}^0$, there is at most one $X\in D^{\infty}_v$ such that $|X|> 1$.
 \end{lemma}

\begin{proof}
 It suffices to show the result for all $v\in E^{0,0}$. Suppose there exist distinct $X,Y$ in $C_v$ such that $|X|>1$ and $|Y| >1$.
 Take a vertex
 $$w := v(x_1,y_1,z_1,\dots ,z_t) \in E_1^{0,1} ,$$
 where $x_1\in X$, $y_1\in Y$ and $z_i\in Z_i$, where $C_v= \{ X,Y,Z_1,\dots , Z_t \}$. The singleton $\{ w \}$ is then hereditary and $C^1$-saturated
 in $E_1$, because the elements $X$ of $C^1$ having edges that start at $w$ are of the form $X=X(x)$, where $x\in \{x_1,y_1,z_1,\dots , z_t \}$, and so they all have more than one element,
 and moreover the sources of the vertices of $X$ are all different.  So $\mathcal H (E_1,C^1) $ is non-trivial, contradicting Lemma \ref{lem:all-layers-simple}.
 \end{proof}

At this point we can already show the coincidence between
the universal graph algebras and their tame quotients.

\begin{corollary}
 \label{cor:C*=O}
If $(E,C)$ is a simple finite bipartite separated graph, then the natural maps $L_K(E,C) \to \Lab_K(E,C)$ and $C^*(E,C)\to \mathcal O (E,C)$ are isomorphisms.
\end{corollary}

\begin{proof}
Let $n\ge 0$. By Lemma \ref{lem:necess-simple1}, we have
$[ee^*,ff^*]=0$ for $e,f\in E_n^1$. Therefore it follows from
\cite[Theorem 5.1(a)]{AE} that the maps $L_K(E_n,C^n) \to L_K(E_{n+1},C^{n+1})$ and $C^*(E_n,C^n) \to
C^*(E_{n+1}, C^{n+1})$ are isomorphisms. Since this holds for each
$n\ge 0$, we get $L_K(E,C) \cong \Lab_K(E,C)$ and $C^*(E,C) \cong \mathcal O (E,C)$.
\end{proof}

 \begin{lemma}
\label{lem:necess-simple2}
Let $(E,C)$ be a simple finite bipartite separated graph.
 Then, for all $w\in E_n^{0,1}$ with $n\ge 1$, there exists $X\in C^n$ such that $|X| =1$ and $s(X)= \{w\}$.
  \end{lemma}

 \begin{proof}
  The only point where we use that $n\ge 1$ is the property that all the source vertices of edges coming from the same set $X\in C$ are distinct.
  Thus, it suffices to prove the result for an arbitrary finite bipartite separated graph $(E,C)$ such that, for every $X\in C$, we have $s(x)\ne s(y)$
  whenever $x,y$ are distinct elements of $X$. Using this condition, we get that if $w\in E^{0,1}$ and $| X|>1$ for all $X\in C$ such that
  $w\in s(X)$, then $\{ w \}$ is a non-trivial hereditary and $C$-saturated subset of $(E,C)$, which contradicts our hypothesis.
   \end{proof}

   \begin{definition}
    \label{def:typesA-B}
    Let $(E,C)$ be a simple finite bipartite separated graph.
    A vertex $v\in E^{0,0}$ is said to be of {\it type A} in case there is a (unique) $X\in C_v$ such that $| X | > 1$, and $v$ is said to be of {\it type B}
    in case $|X|=1$ for all $X\in C_v$. Note that, by Lemma \ref{lem:necess-simple1}, every vertex $v\in E^{0,0}$ is either of type A or of type B.

 If $v\in E^{0,0}$ is of type A, we denote by $X^v$ the unique element in $C_v$ having more than one element.
  \end{definition}

 \begin{lemma}
  \label{lem:necess-simple3}
Let $(E,C)$ be a simple finite bipartite separated graph. Then for
each $w\in E^{0,1}$ there exists at most one $X\in C$ such that
$|X|= 1$, $s(X)= \{ w  \}$, and $X\in C_v$ for a vertex $v\in
E^{0,0}$ of type A.
   \end{lemma}

 \begin{proof}
  Suppose that $X,Y$ are distinct, $X\in C_v$, $Y\in C_{v'}$, for $v,v'$ vertices of type A, that $|X| = |Y| = 1$, and that
  $s(X)=s(Y)= \{ w \}$. Let $X= \{x\}$ and $Y=\{ y \}$. Then $X(x)$ and $X(y)$ are two distinct elements of $C^1_w$, and $|X(x)|>1$, $| X(y)| >1$ because
  there are $X'\in C_v$ and $Y'\in C_{v'}$ with $|X'| >1$ and $|Y'| >1$. This contradicts Lemma \ref{lem:necess-simple1}.
  \end{proof}

 We say that a type B vertex $v$ is {\it of type $B_1$} if, given any $w\in E^{0,1}$, there is at most one $X\in C_v$ such that $|X|=1 $
 and $s(X)= \{ w \}$. Say that $v$ is type $B_2$ in case $v$ is not type $B_1$.

 The following definitions apply to a general finite bipartite separated graph $(E,C)$. Recall that, for $e\in E^1$, we denote by $X_e$ the unique element of $C$
 such that $e\in X_e$.

 \begin{definition}
  \label{def:1-connectedness}
  Let $(E,C)$ be a finite bipartite separated graph. An admissible path $\gamma $ is said to a {\it $1$-path} in case all the edges $e\in E^1$ appearing in the path
  (with exponent $\pm 1$) satisfy that $|X_e|=1$. Length zero paths are also considered $1$-paths.  Two vertices $v,w$ of $E$ are said to be $1$-connected, denoted by $v\sim w$, if there
  is a $1$-path from $v$ to $w$. A $1$-cycle is a $1$-path which is also a cycle.
  \end{definition}

 \begin{remark}
  \label{rem:1-conn-equivrel}
  Observe that  the relation $\sim$ on $E^0$ is an equivalence relation. In fact one can easily show that if $\gamma _1, \gamma _2$ are $1$-paths with $r(\gamma_1) = s(\gamma_2)$, then the reduced
  product $\gamma_2 \cdot \gamma _1$ (i.e. the path obtained after cancellation of terms $ee^{-1}$ or $e^{-1}e$ in the concatenation of $\gamma_2$ and $\gamma_1$)
  is also a $1$-path, which gives the transitivity of the relation $\sim$.
  It is obvious that $\sim $ is symmetric and reflexive.
  \end{remark}

  \begin{lemma}
  \label{lem:necess-simple4}
  Let $(E,C)$ be a simple finite bipartite separated graph.
 The following hold:
 \begin{enumerate}[{\normalfont (a)}]
  \item Two distinct vertices of type A are not $1$-connected.
 \item The vertices of type A and the vertices of type $B_2$ are not $1$-connected.
  \item Let $v$ be a vertex of type A and let $w$ a vertex of type B such that there is a $1$-cycle based at $w$. Then $v$ is not $1$-connected to $w$.
  \end{enumerate}
  \end{lemma}

 \begin{proof}
(a) Let $v, v'$ be two distinct vertices of type A. We show that $v$ is not $1$-connected to $v'$ by induction on the length of a minimal $1$-path between $v$ and $v'$.
Suppose that there is a path $e_2e_1^{-1}$ from $v$ to $v'$ such that $|X_{e_i}| = 1$ for $i=1,2$. Let $w$ be the vertex in $E^{0,1}$ such that $\{ w \} = s(X_{e_1})=s(X_{e_2})$.
Then $X(e_1)$ and $X(e_2)$ are two different elements of $C^1_w$ having more than one element each, contradicting Lemma \ref{lem:necess-simple1}.

Assume that there are no $1$-paths of length $\le 2(n-1)$ between
two vertices of type A, and let $\gamma $ be a path of length $2n$
between two vertices $v,v'$ of type A. Write $\gamma =
e_{2n}e_{2n-1}^{-1}\cdots e_2 e_1^{-1}$, where each $X_{e_i}$ has
only one element (namely $e_i$). Then the vertices $w,w'$ of $E_1$
such that $\{ w \} = s(X_{e_1})$ and $\{ w' \} = s(X_{e_{2n}})$ are
of type A in $E_1$, because $v$ and $v'$ are of type A (in $E$), and
moreover there is a $1$-path in $E_1$ from $w$ to $w'$ of length
$2(n-1)$, leading to a contradiction. This shows that there is no
$1$-path of length $\le 2n$ between distinct vertices of type A.

(b) Let $v$ be a vertex of type A and let $v'$ be a vertex of type $B_2$. We show that $v$ is not $1$-connected to $v'$ by induction on the length of a minimal $1$-path between $v$ and $v'$.
Suppose that there is a path $e_2e_1^{-1}$ from $v$ to $v'$ such that $|X_{e_i}| = 1$ for $i=1,2$. Then in the separated graph $(E_1, C^1)$,
  we have a vertex $w$ with $\{ w \} = s(X_{e_1})= s(X_{e_2})$ and $X, Y \in C^1_w$, where $X= X(e_1)$ and $Y= X(e_2)$. If there is $Z\in C_{v'}$ such that $Z\ne X_{e_2}$,  $|Z| = 1$ and $s(Z)=s(X_{e_2})$, then
  setting $Z=\{ z \}$, we get $X(z)\ne Y$ and $s(X(z))= s(Y)$. Now let $Y= \{ y' \}$ and $X(z) = \{z' \}$. Note that, since $v$ is of type  A, we have $| X| = |X(e_1)| >1$. In conclusion, we have
  that $C^1_w$ has at least three sets $X, Y, X(z)$, with $|X|>1$, $|Y|= |X(z)|= 1$ and $s(Y)=s(X(z))$. Thus, in $(E_2,C^2)$ we have a vertex $\{w'\}= s(Y)=s(X(z))$ with $C_{w'}$ containing two sets $X(y')$ and $X(z')$
  with more than one element. This contradicts Lemma \ref{lem:necess-simple1}.

  Assume now that there is no $Z\in C_{v'}$ such that $Z\ne X_{e_2}$,  $|Z| = 1$ and $s(Z)=s(X_{e_2})$.
  Since $v'$ is of type $B_2$ by hypothesis, there are two distinct sets $Z,T\in C_{v'}$ such that $|Z| = |T| =1 $ and $s(Z)=s(T)$.
  Now the vertex $w$ is type $A$ in $(E_1,C^1)$, and the vertex $w'= s(Z)=s(T)$ is type $B_2$ in $(E_1,C^1)$. Indeed, if $w'$ is of type A, then there is an edge $f$ from $w'$ to a
  vertex $v''\in E^{0,0}$ of type A such that $| X_f | = 1$, and then we could apply the argument in the above paragraph to
  the $1$-path $f^{-1}z$, where $Z= \{ z\}$, to arrive at a contradiction.
  Moreover there is a $1$-path
  $e'_2(e'_1)^{-1}$ from $w$ to $w'$ such that $X(e_2) = \{ e_1'\}$ and $X(z) = \{ e_2' \}$.
  Set $T= \{ t \}$. Then $|X(z)| = |X(t)|= 1$ and $s(X(z))= s(X(t))$ in $(E_1, C^1)$, so that we are in the same situation as before, but replacing
  $(E,C)$ with $(E_1,C^1)$. Consequently, we arrive at a contradiction with the fact that $\mathcal H (E_3,C^3) $ only contains the trivial subsets.

  This shows the case where the length of the path is $2$. If we have a path of length $2n$, then it gives rise to a path of length $2n-2$ between
  a vertex of type $A$ and a vertex of type $B_2$ in the graph $(E_2,C^2)$, leading to a contradiction.

(c) The proof is similar to the proof of (b), we leave the details to the reader.
   \end{proof}

   In the following lemma, we describe the $\mathcal V$-simple algebras corresponding to graphs with only vertices of type A.

 \begin{lemma}
  \label{lem:alltypeA}
  Let $(E,C)$ be a simple finite bipartite separated graph, and suppose that the source vertices of edges coming from the same set $X\in C$ are all distinct.
  If all the vertices in $E^{0,0}$ are of type A, then $L_K(E,C)=\Lab_K(E,C)$ is isomorphic to a Leavitt path algebra $L_K(\overline{E})$, and $C^*(E,C)=\mathcal O (E,C)$ is isomorphic to a graph C*-algebra $C^*(\overline{E})$ of a non-separated graph $\overline{E}$ with $\mathcal{H}(\overline{E})=\{\emptyset,\overline{E}^0\}$.
   \end{lemma}

  \begin{proof}
   By Lemma \ref{lem:necess-simple2} and Lemma \ref{lem:necess-simple3}, for each $w\in E^{0,1}$ there exists a unique $Y\in C$ such that $|Y| = 1$ and $s(Y)= \{ w \}$. We denote by $y_w$ the unique edge in such a group $Y$. It is clear that $E_-^1:=E^1 \setminus \{y_w \mid w \in E^{0,1}\}$ and $E_+^1:=\{y_w \mid w \in E^{0,1}\}$ defines a non-separated orientation in the sense of \cite[Definition 3.8]{Lolk1}, so $L_K(E,C) \cong L_K(\overline{E})$ and $C^*(E,C) \cong C^*(\overline{E})$ by \cite[Proposition 3.11]{Lolk1}, where $\overline{E}$ is the graph obtained from $(E,C)$ by inverting all the edges of $E_+^1$. Moreover, $\overline{E}$ contains no hereditary and saturated subsets since $(E,C)$ contains no hereditary and $C$-saturated subsets.
     \end{proof}

\begin{lemma}\label{lem:typeBdomina}
Let $(E,C)$ be a simple finite bipartite separated graph.  Assume that $v \in E^0$ is of type B, and that $v$ is not $1$-connected to a vertex of type A. Then $v$ is a full projection and there are identifications
$$v\Lab(E,C)v \cong K[\mathbb{F}_v], \quad v\mathcal{O}(E,C)v \cong C^*(\mathbb{F}_v) \andspace v\mathcal{O}^r(E,C)v \cong C^*_r(\mathbb{F}_v),$$
where $\mathbb{F}_v$ is the group of closed paths based at $v$.
\end{lemma}

\begin{proof}
From $(E,C)$ being simple, we see that $v$ is full. We claim that $\Omega(E,C)_v$ is in fact a one-point space. If it were not, then there would exist some admissible path $e\gamma$ with $s(\gamma)=v$ along with $X \in C_{r(e)}$ satisfying $\vert X \vert > 1$ and $X \ne X_e$. We can of course assume that $e \gamma$ is minimal with these properties. Assuming that $\gamma$ passes through a type A vertex, we can write $\gamma=\gamma_2\gamma_1$, where $\gamma_1$ is the maximal initial subpath for which $r(\gamma_1)$ is of type A. But then $e\gamma_2$ is a $1$-path between vertices of type A, contradicting Lemma~\ref{lem:necess-simple4}. We deduce that $e\gamma$ is itself a $1$-path, so that $v$ is $1$-connected to a type A vertex, in conflict with our assumption. It follows that the partial action $\theta^{(E,C)}$ restricts to the trivial global action $\mathbb{F}_v \act \Omega(E,C)_v$, hence
$$v\Lab_K(E,C)v \cong  C_K(\Omega(E,C)_v) \rtimes \mathbb{F}_v \cong K \rtimes \mathbb{F}_v \cong K[\mathbb{F}_v]$$
and
$$v\mathcal{O}^{(r)}(E,C)v \cong  C(\Omega(E,C)_v) \rtimes_{(r)} \mathbb{F}_v \cong \mathbb{C} \rtimes_{(r)} \mathbb{F}_v \cong C^*_{(r)}(\mathbb{F}_v)$$
by Lemma~\ref{lem:IsolatedPoints}.
\end{proof}

     We are now ready to prove Theorem \ref{thm:dichotomy}.

\begin{proof}[Proof of Theorem \ref{thm:dichotomy}]
    Passing, if necessary, to $(E_1, C^1)$, we can assume that for each $X\in C$ we have that $s(x)\ne s(x')$ whenever $x,x'$ are distinct elements of $X$.
    If $E^{0,0}$ consists entirely of vertices of type A, then we can reach (1) by Lemma \ref{lem:alltypeA}.
    If $E^{0,0}$ contains both vertices of type A and of type B, and a vertex of type A is $1$-connected to a vertex of type B, then its $1$-connected
    component cannot contain any $1$-cycle, by Lemma \ref{lem:necess-simple4}. In that case, the vertices of type B in it will vanish in the
    graph $(E_n,C^n)$ for some $n$.

Indeed, let $v$ be a vertex of type A and let $\gamma $ be a
non-trivial $1$-path starting at $v$. By Lemma
\ref{lem:necess-simple4}(a), all vertices in $E^{0,0}$ visited by
$\gamma $ are of type B. If the $1$-connected component of $v$ does
not contain $1$-cycles, then all the vertices visited by $\gamma $
must be distinct. So we may assume that $\gamma $ is of maximal
length. Suppose for instance that
$$\gamma = e_{2r}e_{2r-1}^{-1}\cdots e_2e_1^{-1}$$
is of even length. Then the vertex $s(e_1)$ is of type A in
$(E_1,C^1)$, and there is a $1$-path $\gamma ' =
(e'_{2r})^{-1}e_{2r-1}'\cdots (e_2')^{-1}$ of length $2r-1$ in
$(E_1,C^1)$, where each $e_i'$ belongs to $X(e_i)$, for $i=2,\dots ,
2r$. Moreover all the $1$-paths of $(E_1, C^1)$ starting at $s(e_1)$
are of this form, and we conclude that the $1$-connected component
of the vertex $s(e_1)$ only contains $1$-paths of length $\le 2r-1$.
The case where $\gamma $ is of odd length is treated in the same
way. Now let $n$ be the maximum of the lengths of all the maximal
$1$-paths starting at vertices of type A in $(E,C)$. The above
argument shows that the graph $(E_n, C^n)$ has the property that no
vertex of type A is $1$-connected to a vertex of type B.

    So we can assume, in addition, that no vertex of type A is $1$-connected to a vertex of type B.
    But then Lemma \ref{lem:typeBdomina} applies.
\end{proof}

\section{Primeness}\label{sect:Primeness}

Recall that a ring is called \textit{prime} if the product of any two non-zero ideals is non-zero. In this final section, we characterize primeness of the algebras $\Lab_K(E,C)$ and $\mathcal{O}^r(E,C)$ in purely graph theoretical terms when the configuration space $\Omega(E,C)$ is a Cantor space. In order to check this hypothesis, we also develop a test for the existence of isolated points.
\medskip \\
We first introduce a partial version of a well known concept.

\begin{definition}
A partial action $\theta \colon G \act \Omega$ is called \textit{topologically transitive} if for any two non-empty open subsets $U,U' \subset G$, there exists $g \in G$ such that $\theta_g(U \cap \Omega_{g^{-1}}) \cap U' \ne \emptyset$.
\end{definition}

The main technical tool for our analysis is the fact that for a topologically free partial action $\theta \colon G \act \Omega$ on a totally disconnected compact Hausdorff space, the algebras $C_K(\Omega) \rtimes G$ and $C(\Omega) \rtimes_r G$ are prime if and only if $\theta$ is topologically transitive. This should be clear from the fact that both these crossed products enjoy the intersection property (see \cite[Lemma 4.2]{BCFS} and \cite[Theorem 29.5]{Exel}).

Throughout this section, we will write $\mathfrak{i}_d(\alpha)$ and $\mathfrak{t}_d(\alpha)$ for the initial (i.e right most) and terminal (i.e. left most) edge, respectively, of a non-trivial admissible $\alpha$ when viewed as a path in the double $\widehat{E}$.

\begin{definition}
Let $(E,C)$ denote a finite bipartite separated graph.  If $B \in \mathcal{B}_n(\Omega(E,C))$, then we define a clopen subspace by
$$\Omega(E,C)_B:=\{ \xi \in \Omega(E,C) \mid \xi^n = B\},$$
and we will say that $B$ is an $n$-ball \textit{at} $v$, where $v \in E^0$ is such that $\Omega(E,C)_B \subset \Omega(E,C)_v$. By the \textit{boundary of} $B$, we shall mean the set
$$\partial B:= \{\ter_d(\alpha) \mid \alpha \in B \text{ is maximal}\}.$$
\end{definition}

\begin{definition}
Let $(E,C)$ denote a finite bipartite graph, and consider a \textit{path closed} subset $A \subset E^1 \cup (E^1)^{-1}$, that is, a subset satisfying 
$$\ini_d(\alpha) \in A \Rightarrow \ter_d(\alpha) \in A$$
for any admissible path $\alpha$. From $(E,C)$ being bipartite, this may also be phrased as follows:
\begin{itemize}
\item If $ef^{-1}$ is admissible and $f^{-1} \in A$, then $e \in A$.
\item If $e^{-1}f$ is admissible and $f \in A$, then $e^{-1} \in A$.
\end{itemize}
We will say that such $A$ is $\partial$\textit{-closed} if the following holds:
\begin{enumerate}
\item Assume $\vert s^{-1}(v) \vert \ge 2$ and let $e \in s^{-1}(v)$. If $s^{-1}(v) \setminus \{e\} \subset A$, then $e^{-1} \in A$ as well.
\item Assume $\vert C_v \vert \ge 2$ and let $X \in C_v$. If $Y^{-1} \cap A \ne \emptyset$ for any $Y \in C_v \setminus \{X\}$, then $X \subset A$ as well.
\end{enumerate}
It is clear that an intersection of $\partial$-closed sets is again $\partial$-closed, so every $A \subset E^1 \cup (E^1)^{-1}$ is contained in a minimal $\partial$-closed subset $\overline{A}$. We then associate a set of vertices to $A$ by
$$V(A):=\{v \in E^{0,0} \mid X^{-1} \cap \overline{A} \ne \emptyset \text{ for all } X \in C_v\} \cup \{v \in E^{0,1} \mid s^{-1}(v) \subset \overline{A}\}.$$
\end{definition}

\begin{lemma}\label{lem:Boundary}
Let $(E,C)$ denote a finite bipartite separated graph, and consider a path closed subset $A \subset E^1 \cup (E^1)^{-1}$. Then $v \in V(A)$ if and only if there exists some $n \ge 1$ and an $n$-ball $B\in \mathcal B _n(\Omega (E,C))$ at $v$ such that $\partial B \subset A$.
\end{lemma}
\begin{proof}
First suppose that $v \notin V(A)$. We will show that $\partial B \not\subset \overline{A}$ for any $n \ge 1$ and any $n$-ball $B$ at $v$, and we shall argue by induction on $n$. 
If $n=1$, then this is clear by definition of $V(A)$. Assuming that it holds for some $n \ge 1$, consider any $(n+1)$-ball $B$
and write
$$B^n:=\{\alpha \in B \colon \vert \alpha \vert \le n \}.$$
From our inductive assumption, there exists a maximal admissible path $\alpha \in B^n$ such that $\ter_d(\alpha) \notin \overline{A}$. 
If $\alpha$ is also maximal in $B$, then surely $\partial B \not\subset \overline{A}$, so we may assume that it is not. We then divide into the two cases
$r(\alpha) \in E^{0,0}$ and $r(\alpha) \in E^{0,1}$. In the former, we can write $\alpha=e \beta$, and there exists 
some $X_e \ne Y \in C_{r(\alpha)}$ with $Y^{-1} \cap \overline{A} = \emptyset$. 
By definition of $\Omega(E,C)$, we then have $y^{-1}\alpha \in B$ for some $y \in Y$. In the other case, 
we may instead write $\alpha=e^{-1}\beta$. Since $e^{-1} \notin \overline{A}$, we see that $f \notin \overline{A}$ for some $e \ne f \in s^{-1}(r(\alpha))$, 
and we consider the path $f \alpha \in B$. Either way, we have found an element of the boundary $\partial B$ which is not contained in $\overline{A}$, so in particular not in $A$.

For the converse implication, we first introduce a bit of handy notation, specifically we define a partition $\overline{A}=\bigsqcup_{m=0}^\infty A_m$. First set $A_0:=A$, let $m \ge 0$ and assume that $A_k$ has been defined for all $k \le m$. Then for any $e \in E^1$, we declare that $e^{-1} \in A_{m+1}$ if
$$e^{-1} \notin \bigsqcup_{k=0}^m A_k, \quad \vert s^{-1}(s(e)) \vert \ge 2 \andspace s^{-1}(s(e)) \setminus \{e\} \subset \bigsqcup_{k=0}^m A_k.$$
Similarly, $e \in A_{m+1}$ if
$$e \notin \bigsqcup_{k=0}^m A_k , \quad \vert C_{r(e)} \vert \ge 2 \andspace Y^{-1} \cap \bigsqcup_{k=0}^m A_k \ne \emptyset \text{ for any } Y \in C_{r(e)} \setminus \{X_e\}.$$
Now set 
$$T:=\{e \in E^1 \colon \vert C_{r(e)} \vert = 1 \} \cup \{e^{-1} \in (E^1)^{-1} \colon \vert s^{-1}(s(e)) \vert = 1\}$$
and observe that $A \cap T= \overline{A} \cap T$.
Whenever $\sigma \in \overline{A}$, we write $m_\sigma$ for the number satisfying $\sigma \in A_{m_\sigma}$, and if $X^{-1} \cap \overline{A}\ne \emptyset $ for some $X \in C$, we set $m_X:=\min\{m_{x^{-1}} \mid x^{-1} \in X^{-1}\cap A \}$. Given $v \in V(A)$, we then define
$$n_v:=\left\{\begin{array}{cl}
\max\{m_e \mid e \in s^{-1}(v)\}+1 & \If v \in E^{0,1} \\
\max\{m_X \mid X \in C_v\}+1 & \If v \in E^{0,0}
\end{array}  \right. ,$$
claiming that $\partial B \subset A$ for an $n_v$-ball $B$ at $v$. Specifically, we will show that whenever $n \le n_v$, there is an $n$-ball $B_n$ at $v$ satisfying $\partial B_n \subset \bigsqcup_{k=0}^{n_v-n}A_k$, 
and we proceed by induction over $n$. For $n=1$, this is clear. Assuming that the claim has been verified for some $1 \le n < n_v$,
and letting $i$ be such that $v \in E^{0,i}$, we consider the cases of $i+n$ being odd and even, separately. If it is odd, then there is a unique $(n+1)$-ball $B_{n+1}$ containing $B_n$ given by
$$B_{n+1}=B_n \cup \big\{e \alpha \colon \alpha \in B_n, \vert \alpha \vert = n \text{ and } e \in s^{-1}(r(\alpha)) \setminus \{\ter_d(\alpha)^{-1}\} \big\}$$
with boundary
$$\partial B_{n+1} =(\partial B_n \cap T) \cup \bigcup_{f \in \partial B_n \setminus T} s^{-1}(s(f)) \setminus \{f\}  ,$$
so it follows immediately from the above definition and $A$ being path closed that $\partial B_{n+1} \subset \bigsqcup_{k=0}^{n_v-n-1} A_k$. If $i+n$ is even, then $\partial B_n \setminus T \subset E^1$, and for any $f \in \partial B_n \setminus T$, $Y \in C_{r(f)} \setminus \{X_f\}$, we can choose an edge $e_{f,Y} \in Y$ such that $e_{f,Y}^{-1} \in \bigsqcup_{k=0}^{n_v-n-1} A_k$. Now define an $(n+1)$-ball $B_{n+1}$ by
$$B_{n+1}:=B_n \cup \big\{e_{\ter_d(\alpha),Y}^{-1}\alpha \colon \alpha \in B_n, \vert \alpha \vert = n \text{ and } Y \in C_{r(\alpha)} \setminus \{X_{\ter_d(\alpha)}\} \big\}$$
and observe that
$$\partial B_{n+1} = (\partial B_n \cap T) \cup \big\{e_{f,Y}^{-1} \mid f \in \partial B_n \setminus T \text{ and } Y \in C_{r(f)} \setminus \{X_f\} \big\},$$
hence $\partial B_{n+1} \subset \bigsqcup_{k=0}^{n_v-n-1} A_k$ as desired. Considering the particular case $n=n_v$ and $B:=B_{n_v}$, we finally see that $\partial B \subset A$.
\end{proof}

\begin{definition}
Let $(E,C)$ denote a bipartite separated graph. A \textit{choice path} is an admissible path $\alpha=e \beta$ such that there exists $X_e \ne X \in C_{r(\alpha)}$ with $\vert X \vert \ge 2$. Given 
any edge $f \in E^1$, we will say that $f$ is a \textit{dead end} if there is no choice path starting with $f$, and likewise $f^{-1}$ is a dead end if no choice path starts with $f^{-1}$.
\end{definition}

\begin{proposition}\label{prop:Cantor}
Let $(E,C)$ denote a finite bipartite separated graph, and define
$$A_{\textup{DE}}:=\{\text{dead ends of } E^1 \cup (E^1)^{-1}\}.$$
Then $\Omega(E,C)_v$ contains an isolated point if and only if $v \in V(A_\textup{DE})$. Consequently, $\Omega(E,C)$ is a Cantor space if and only if $V(A_\textup{DE})=\emptyset$.
\end{proposition}

\begin{proof}
Observe first that $A:=A_{\textup{DE}}$ is path closed. It follows from Lemma~\ref{lem:Boundary} that $v \in V(A)$ if and only if there exists a ball $B$ at $v$ such that $\partial B \subset A$. If this is the case, then $\Omega(E,C)_B$ is a one-point space, so $\Omega(E,C)_v$ does indeed contain an isolated point. If $v \notin V(A)$, then given any ball $B$ at $v$, there exists $\sigma \in \partial B$ which is not a dead end. Consequently, $\Omega(E,C)_B$ contains at least two configurations, so $\Omega(E,C)_v$ does not contain any isolated points.
\end{proof}

\begin{remark}\label{rem:CantorImpliesTopFree}
It is worth mentioning that if $\Omega(E,C)$ is a Cantor space, then $\theta^{(E,C)}$ is automatically topologically free. Indeed, by \cite[Theorem 10.5]{AE}, $\theta^{(E,C)}$ is topologically free if and only if for every vertex $v \in E^{0,1}$ on a cycle, there exists some $e \in s^{-1}(v)$ which is not a dead end. And if no such $e$ existed, then $\Omega(E,C)_v$ would be a one-point space.
\end{remark}

\begin{definition}
Let $(E,C)$ denote a finite bipartite separated graph. We will say that two sets $A, A' \subset E^1 \cup (E^1)^{-1}$ \textit{can be linked} if there exist $\sigma \in A$, $\sigma' \in A'$ and an admissible path $\alpha$ such that the concatenation $\sigma^{-1}\alpha \sigma'$ is admissible. If this is not the case, then the pair $A,A'$ is \textit{unlinkable}. Morover, it is \textit{maximal unlinkable} if for any larger pair $A \subset D$, $A' \subset D'$ we have
$$D \text{ and } D' \text{ are unlinkable} \Rightarrow A=D, A'=D'.$$
Finally, $(E,C)$ is said to have \textit{the Linking Property} if $\partial B$ and $\partial B'$ can be linked for any two balls $B,B' \in \mathcal{B}(\Omega(E,C))$.
\end{definition}

\begin{lemma}\label{lem:TopTrans}
If $(E,C)$ has the Linking Property, then $\theta^{(E,C)}$ is topologically transitive, and if $\Omega(E,C)$ is a Cantor space, then the reverse implication holds as well.
\end{lemma}

\begin{proof}
Assume first that $(E,C)$ has the Linking Property. Then given any two open sets $U,U' \subset \Omega(E,C)$, there are balls $B,B'$ such that $\Omega(E,C)_B \subset U$ and $\Omega(E,C)_{B'} \subset U'$. 
Since $B$ and $B'$ can be linked, there exist $\sigma \in \partial B$, $\sigma' \in \partial B'$ and an admissible path $\alpha$ for which $\sigma^{-1}\alpha\sigma'$ is admissible. Let $\gamma \in B$ and $\gamma' \in B'$ be such that $\ter_d(\gamma)=\sigma$ and $\ter_d(\gamma')=\sigma'$, and set $\beta:=\gamma^{-1}\alpha\gamma'$. The situation is depicted below in Figure~\ref{fig:Links}.
\newcommand{\Lo}{0.866}
\begin{center}\begin{figure}[H]
\begin{tikzpicture}[font=\scriptsize,scale=0.7]
  \tikzset{VertexStyle/.style = {draw,shape = circle,minimum size=14pt,inner sep=1pt}}

  \Vertex[x=0,y=0,L=$1$]{u1}
  \Vertex[x=9,y=0, L=$\beta$]{w3}

  \tikzset{VertexStyle/.style = {draw,shape = circle,minimum size=1pt,inner sep=1pt}}
  \SetVertexNoLabel  
  
  \Vertex[x=1,y=0]{u2}
  \Vertex[x=-0.5,y=\Lo]{u3}
  \Vertex[x=-0.5,y=-\Lo]{u4}
  
  \Vertex[x=1.5,y=\Lo]{u5}
  \Vertex[x=1.5,y=-\Lo]{u6}
  \Vertex[x=-1,y=-2*\Lo]{u7}
  \Vertex[x=-1,y=2*\Lo]{u8}
     
  \Vertex[x=1,y=2*\Lo]{u9}
  \Vertex[x=2.5,y=\Lo]{u10}
  \Vertex[x=1,y=-2*\Lo]{u11}
  \Vertex[x=2.5,y=-\Lo]{u12}
  \Vertex[x=-0.5,y=3*\Lo]{u13}
  \Vertex[x=-2,y=2*\Lo]{u14}
  \Vertex[x=-0.5,y=-3*\Lo]{u15}
  \Vertex[x=-2,y=-2*\Lo]{u16}
  
  \Vertex[x=3.5,y=\Lo]{u17}
  \Vertex[x=3.5,y=-\Lo]{u18}
  \Vertex[x=0.5,y=-3*\Lo]{u19}
  \Vertex[x=-1.5-\Lo,y=-3*\Lo]{u21}
  \Vertex[x=-1.5-\Lo,y=-\Lo]{u22}
  \Vertex[x=-3,y=2*\Lo]{u23}
  \Vertex[x=0.5,y=3*\Lo]{u24}
  \Vertex[x=-1,y=4*\Lo]{u25}

  \Vertex[x=4.5,y=-\Lo]{v1}
  \Vertex[x=5.5,y=-\Lo]{v2}
  \Vertex[x=6.5,y=-\Lo]{v3}
  \Vertex[x=7.5,y=-\Lo]{v4}          

  \Vertex[x=8,y=0]{w1}
  
  \Vertex[x=7.5,y=\Lo]{w2}
  \Vertex[x=7,y=2*\Lo]{w4}
  \Vertex[x=10.5,y=\Lo]{w5}
  \Vertex[x=11,y=2*\Lo]{w9}
  \Vertex[x=11,y=0]{w10}  
  \Vertex[x=9,y=2*\Lo]{w8}
  \Vertex[x=8.5,y=3*\Lo]{w11}

  \Vertex[x=9.5,y=\Lo]{w6}      
  \Vertex[x=9.5,y=-\Lo]{w7}        
  
\tikzset{VertexStyle/.style = {draw,shape = circle, minimum size=160,inner sep=1pt}}

\Vertex[x=0.05,y=0]{B1}    

\tikzset{VertexStyle/.style = {draw,shape = circle, minimum size=120,inner sep=1pt}}

\Vertex[x=9,y=0]{B2}        
  \tikzset{EdgeStyle/.style = {->,color=\niceblue}}
  \Edge[labelcolor=none, labelstyle=above](u2)(u1)
  \Edge[labelcolor=none, labelstyle=right](u2)(u5)
  \Edge[labelcolor=none, labelstyle=right](u4)(u7)
  \Edge[labelcolor=none, labelstyle=left](u3)(u8)
  \Edge[labelcolor=none, labelstyle=left](u11)(u6)
  \Edge[labelcolor=none, labelstyle=above](u10)(u17)  
  \Edge[label=$\sigma'$,labelcolor=none, labelstyle=above](u12)(u18)  
  \Edge[labelcolor=none, labelstyle=left](u16)(u21)
  \Edge[labelcolor=none, labelstyle=left](u16)(u22)     
  \Edge[labelcolor=none, labelstyle=above](u14)(u23)
  \Edge[labelcolor=none, labelstyle=above](u13)(u24)
  \Edge[labelcolor=none, labelstyle=left](u13)(u25)          
  \Edge[labelcolor=none, labelstyle=above, label=$\sigma^{-1}$](v3)(v4)            
  \Edge[labelcolor=none, labelstyle=left](w6)(w3)          
  \Edge[labelcolor=none, labelstyle=left](w6)(w5)            
  
  \tikzset{EdgeStyle/.style = {->,color=\nicered}}
  \Edge[labelcolor=none, labelstyle=left](u3)(u1)
  \Edge[labelcolor=none, labelstyle=left](u2)(u6)
  \Edge[labelcolor=none, labelstyle=above](u10)(u5)
  \Edge[labelcolor=none, labelstyle=right](u13)(u8)  
  \Edge[labelcolor=none, labelstyle=above](u16)(u7)
  \Edge[labelcolor=none, labelstyle=above](v1)(u18)
  \Edge[labelcolor=none, labelstyle=above](v1)(v2)
  \Edge[labelcolor=none, labelstyle=above](w1)(v4)          
  \Edge[labelcolor=none, labelstyle=above](w1)(w2)            
  \Edge[labelcolor=none, labelstyle=left](w7)(w3)  
  \Edge[labelcolor=none, labelstyle=left](w6)(w8)
  \Edge[labelcolor=none, labelstyle=left](w10)(w5)      

  \tikzset{EdgeStyle/.style = {->,color=\nicegreen}}
  \Edge[labelcolor=none, labelstyle=right](u4)(u1)
  \Edge[labelcolor=none, labelstyle=left](u9)(u5)
  \Edge[labelcolor=none, labelstyle=above](u12)(u6)
  \Edge[labelcolor=none, labelstyle=above](u14)(u8)  
  \Edge[labelcolor=none, labelstyle=right](u15)(u7)
  \Edge[labelcolor=none, labelstyle=left](u11)(u19)
  \Edge[labelcolor=none, labelstyle=above](v3)(v2)        
  \Edge[labelcolor=none, labelstyle=above](w1)(w3)            
  \Edge[labelcolor=none, labelstyle=above](w4)(w2)              
  \Edge[labelcolor=none, labelstyle=above](w9)(w5)              
  \Edge[labelcolor=none, labelstyle=above](w11)(w8)                  
  
  \draw[decoration={brace,mirror,raise=4pt},decorate,thick]
  (3.4,-\Lo) -- node[below=6pt] {$\alpha$} ((6.6,-\Lo);
\end{tikzpicture}
\caption{$B'$ (to the left) and $B$ (to the right) linked by the admissible path $\alpha$.}\label{fig:Links}
\vspace{-1cm}
\end{figure}
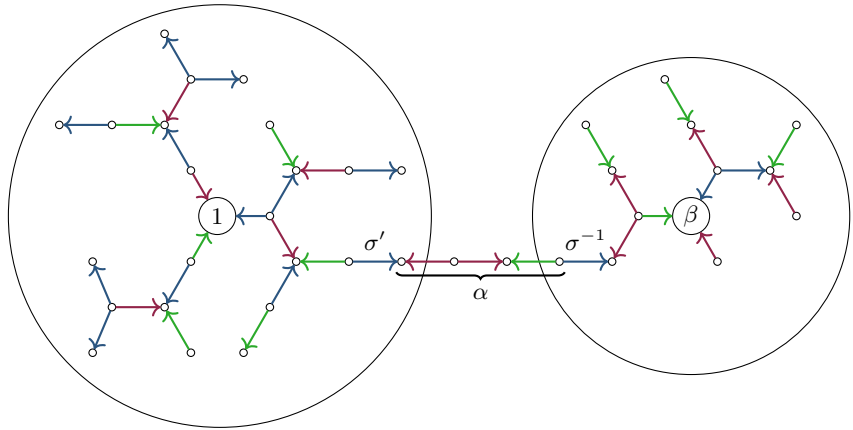
\end{center}
It follows that 
$$\emptyset \ne \theta_\beta(\Omega(E,C)_B \cap \Omega(E,C)_{\beta^{-1}}) \cap \Omega(E,C)_{B'} \subset \theta_\beta(U \cap \Omega(E,C)_{\beta^{-1}}) \cap U',$$
so $\theta^{(E,C)}$ is indeed topologically transitive. Conversely, assume now that $\theta^{(E,C)}$ is topologically transitive and $\Omega(E,C)$ is a Cantor space. Consider any two balls $B,B'$ and assume without loss of generality that both have radius $n$. If $\beta \in \mathbb{F}$ with $\vert \beta \vert < 2n$ is such that 
$$\theta_\beta \big(\Omega(E,C)_B \cap \Omega(E,C)_{\beta^{-1}}\big) \cap \Omega(E,C)_{B'} \ne \emptyset,$$
then take any two distinct points $\xi,\xi'$ in this open set, using the assumption about isolated points. By the Hausdorff property, these can be separated by open neighbourhoods $U_1,U_1'$ of $\xi$ and $\xi'$, respectively, 
contained in the above intersection. We have $\theta _{\beta^{-1}} (U_1) \subseteq \Omega (E,C)_B$, $U_1'\subseteq \Omega (E,C)_{B'}$, and
$$\theta_\beta \big(\theta _{\beta^{-1}} (U_1) \cap \Omega(E,C)_{\beta^{-1}}\big) \cap U_1' = \emptyset.$$
By applying this procedure sufficiently many times, we see that there are non-empty open subsets $V\subseteq \Omega (E,C)_B$ and $V'\subseteq \Omega (E,C)_{B'}$, such that 
$$\theta _{\beta} \big( V\cap \Omega (E,C)_{\beta^{-1}}\big) \cap V' = \emptyset$$
for all $\beta \in \mathbb{F}$ with length $\vert \beta \vert < 2n$. However, by topological transitivity, there is some $\beta \in \mathbb{F}$ for which this intersection is non-empty, 
so $\vert \beta \vert \ge 2n$. It follows that $B$ and $B'$ can be linked.
\end{proof}

We are now in a position to characterize primeness.

\begin{theorem}\label{thm:Prime}
Assume that $\Omega(E,C)$ is a Cantor space. Then either algebra $\Lab_K(E,C)$ or $\mathcal{O}^r(E,C)$ is prime if and only if $V(A)=\emptyset$ or $V(A')=\emptyset$ for all maximal unlinkable pairs $A,A' \subset E^1 \cup (E^1)^{-1}$.
\end{theorem}

\begin{proof}
Observe that, by maximality, $A$ and $A'$ as above are path closed. By Remark~\ref{rem:CantorImpliesTopFree}, $\theta^{(E,C)}$ is topologically free, so $\Lab_K(E,C)$ and $\mathcal{O}^r(E,C)$ are both prime if and only if $\theta^{(E,C)}$ is topologically transitive, which is again equivalent to the Linking Property by Lemma~\ref{lem:TopTrans}. So we really have to check that the Linking Property is equivalent to the above condition. But this is clear from Lemma~\ref{lem:Boundary}.
\end{proof}

\begin{remark}
It is of course also natural to ask if interesting prime algebras can be constructed from a finite bipartite separated graph $(E,C)$, where the configuration space $\Omega(E,C)$ does contain isolated 
points. However, this is not the case: For sufficiently big $n$, there must exist a vertex $v \in E_n^{0,1}$ such that $\Omega(E_n,C^n)_v$ is a one-point space, i.e. every $e \in s^{-1}(v)$ is a dead end. 
By topological transitivity, the orbit of this single point is dense, so in particular $v$ can be connected to any other vertex by an admissible path. But then $(E_n,C^n)$ must satisfy 
Condition (C) of \cite[Definition 3.5]{Lolk1}. Note that $v$ cannot admit 
exactly one simple closed path (up to inversion), for then it would generate an ideal Morita equivalent to $K[\mathbb{Z}]$ or $C(\mathbb{T})$, depending on the situation, 
by \ref{lem:IsolatedPoints}. If $v$ does not admit a closed path, then both algebras degenerate to graph algebras of a non-separated graph by \cite[Theorem 5.7]{Lolk1}. Finally, if $v$ admits at 
least two simple closed paths (up to inversion), then it generates an ideal Morita equivalent with the group algebra $K[\mathbb{F}_v]$ in the algebraic setting and $C^*_r(\mathbb{F}_v)$ in the $C^*$-algebraic (here $\mathbb{F}_v$ denotes the group of all closed paths based at $v$), and the quotient is a classical graph algebra.
\end{remark}

\begin{example}
We now apply our work to a few examples:
\begin{enumerate}
\item If $(E,C)=(E(m,n),C(m,n))$ as in Example~\ref{exam:m,ndyn-system}, then $A_\textup{DE}= \emptyset$ and every pair of edges can be linked, so $\Omega(E,C)$ is a Cantor space and the algebras $\Lab_{m,n}$ and $\mathcal{O}_{m,n}^r$ are prime.
\item Consider the graph $(E,C)$ as pictured just below:
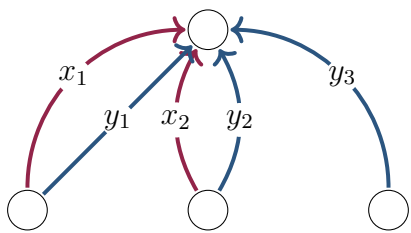
\begin{center}{
\begin{figure}[htb]
\begin{tikzpicture}[scale=0.8]
 \SetUpEdge[lw         = 1.5pt,
            labelcolor = white]
  \tikzset{VertexStyle/.style = {draw, shape = circle,fill = white, minimum size=15pt, inner sep=2pt,outer sep=1pt}}

\SetVertexNoLabel

  \Vertex[x=0,y=0]{1}
  \Vertex[x=3,y=0]{2}
  \Vertex[x=6,y=0]{3}  
  \Vertex[x=3,y=3]{v}  

  \tikzset{EdgeStyle/.style = {->,bend left=45,color={\nicered}}}  
  \Edge[label=$x_1$, style={circle,inner sep=0pt }](1)(v)
  \tikzset{EdgeStyle/.style = {->,bend left=30,color={\nicered}}}  
  \Edge[label=$x_2$, style={circle,inner sep=0pt }](2)(v)  
  \tikzset{EdgeStyle/.style = {->,bend right=0,color={\niceblue}}}  
  \Edge[label=$y_1$, style={circle,inner sep=0pt }](1)(v) 
  \tikzset{EdgeStyle/.style = {->,bend right=30,color={\niceblue}}}  
  \Edge[label=$y_2$, style={circle,inner sep=0pt }](2)(v) 
  \tikzset{EdgeStyle/.style = {->,bend right=45,color={\niceblue}}}  
  \Edge[label=$y_3$, style={circle,inner sep=0pt }](3)(v) 
\end{tikzpicture}
\end{figure}}
\end{center}
Note that $A_{\textup{DE}}=\{y_3^{-1}\}$ has closure $\overline{A_{\textup{DE}}}=\{x_1,x_2,y_1^{-1},y_2^{-1},y_3^{-1}\}$, so $V(A_{\textup{DE}})=\emptyset$. It follows that $\Omega(E,C)$ is a Cantor space. However, the set $A:=\{x_1^{-1},x_2^{-1},y_1,y_2,y_3\}$ is not linked to itself and $V(A)=\{s(y_3)\}$, so the algebras are not prime.
\item Now consider the following variation of the above graph:
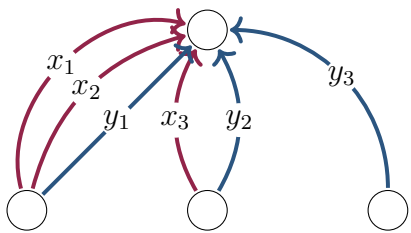
\begin{center}{
\begin{figure}[htb]
\begin{tikzpicture}[scale=0.8]
 \SetUpEdge[lw         = 1.5pt,
            labelcolor = white]
  \tikzset{VertexStyle/.style = {draw, shape = circle,fill = white, minimum size=15pt, inner sep=2pt,outer sep=1pt}}

\SetVertexNoLabel

  \Vertex[x=0,y=0]{1}
  \Vertex[x=3,y=0]{2}
  \Vertex[x=6,y=0]{3}  
  \Vertex[x=3,y=3]{v}  

  \tikzset{EdgeStyle/.style = {->,bend left=60,color={\nicered}}}  
  \Edge[label=$x_1$, style={circle,inner sep=0pt }](1)(v)
  \tikzset{EdgeStyle/.style = {->,bend left=30,color={\nicered}}}  
  \Edge[label=$x_2$, style={circle,inner sep=0pt }](1)(v)  
  \tikzset{EdgeStyle/.style = {->,bend left=30,color={\nicered}}}  
  \Edge[label=$x_3$, style={circle,inner sep=0pt }](2)(v)    
  \tikzset{EdgeStyle/.style = {->,bend right=0,color={\niceblue}}}  
  \Edge[label=$y_1$, style={circle,inner sep=0pt }](1)(v) 
  \tikzset{EdgeStyle/.style = {->,bend right=30,color={\niceblue}}}  
  \Edge[label=$y_2$, style={circle,inner sep=0pt }](2)(v) 
  \tikzset{EdgeStyle/.style = {->,bend right=45,color={\niceblue}}}  
  \Edge[label=$y_3$, style={circle,inner sep=0pt }](3)(v) 
\end{tikzpicture}
\end{figure}}
\end{center}
Once again we have $V(A_{\textup{DE}})=\emptyset$, so $\Omega(E,C)$ is indeed a Cantor space. Observe that there is a unique pair of 
maximal unlinkable subsets, namely 
$$A=A':=\{x_1,x_2,x_3,y_1^{-1},y_2^{-1},y_3^{-1}\},$$ 
and $V(A)=\emptyset$. It follows that $\Lab_K(E,C)$ and $\mathcal{O}^r(E,C)$ \textit{are} prime in this case.
\end{enumerate}
\end{example}

\begin{remark}
We finally remark that topological transitivity of $\theta^{(E,C)}$ can be phrased very simply in terms of the separated Bratteli diagram $(F_\infty,D^\infty)$. Since the vertices of $F_\infty$ correspond to the balls of $\Omega(E,C)$, and there is a direct dynamical equivalence $\theta^{(E_n,C^n)} \xrightarrow{\approx} \theta^{(E,C)}$ for any $n$, the partial action $\theta^{(E,C)}$ is topologically transitive if and only if for all $n$, any two vertices $u,v \in E_n^0$ 
can be connected by an admissible path in $(E_n,C^n)$, or, alternatively, that any two vertices in $F_ {\infty}$ lay over a common vertex $w\in F_{\infty}^0$. Consequently, one can give another proof of Theorem~\ref{thm:Prime} by checking 
that the graph theoretical condition passes from $(E,C)$ to $(E_1,C^1)$. 
\end{remark}

\end{document}